\documentclass[11pt,psamsfonts,reqno]{amsart}
\usepackage{amsmath,amsthm,amssymb,amscd,amsfonts,amsbsy}
\usepackage{subfigure,graphicx}
\usepackage{afterpage,rotating,pdflscape}
\usepackage{color}
\usepackage{cite} 
\usepackage{overpic}
\usepackage{epstopdf}
\usepackage{hyperref,url}
\usepackage{array,multirow,multicol}
\usepackage{adjustbox,blindtext}
\usepackage[dvipsnames]{xcolor}
\usepackage{float}
\usepackage{mathabx}
\usepackage{hyperref}
\newcommand*{\defeq }{\mathrel{\vcenter{\baselineskip0.5ex \lineskiplimit0pt
                     \hbox{\scriptsize.}\hbox{\scriptsize.}}}%
                     =}
                     \newcommand*{\eqdef }{=\mathrel{\vcenter{\baselineskip0.5ex \lineskiplimit0pt
                     \hbox{\scriptsize.}\hbox{\scriptsize.}}}%
                     }
\usepackage[autostyle]{csquotes}

\newcolumntype{R}{>{\displaystyle}r}
\newcommand{\R}{\ensuremath{\mathbb{R}}}

\newcommand{\CO}{\ensuremath{\mathcal{O}}}

\newcommand{\CF}{\ensuremath{\mathcal{F}}}
\newcommand{\ov}{\overline}

\newcommand{\T}{\theta}
\newcommand{\f}{\varphi}
\newcommand{\al}{\alpha}
\newcommand{\la}{\lambda}

\newcommand{\sgn}{\mathrm{sign}}

\makeatletter
\newcommand{\tpitchfork}{%
	\raise-0.3ex\vbox{
		\baselineskip\z@skip
		\lineskip-.52ex
		\lineskiplimit\maxdimen
		\m@th
		\ialign{##\crcr\hidewidth\smash{$-$}\hidewidth\crcr$\pitchfork$\crcr}
	}%
}

\def\p{\partial}
\def\e{\varepsilon}

\newtheorem {theorem} {Theorem} 

\newtheorem {proposition} {Proposition}

\newtheorem {lemma} {Lemma}
\newtheorem {remark}{Remark}
\newtheorem {definition} {Definition}

\newtheorem {mtheorem} {Theorem}

\usepackage{mathpazo}

\begin{document}

\title[Smoothing of nds near regular-tangential singularities]
{Smoothing of nonsmooth differential systems\\ near regular-tangential singularities\\ and boundary limit cycles}
\author[D. D. Novaes and G. Rond\'{o}n]
{Douglas D. Novaes and Gabriel Rond\'{o}n}

\address{Departamento de Matem\'{a}tica, Instituto de Matem\'{a}tica, Estat\'{i}stica e Computa\c{c}\~{a}o Cient\'{i}fica (IMECC), Universidade Estadual de Campinas (UNICAMP), Rua S\'{e}rgio Buarque de Holanda, 651, Cidade Universit\'{a}ria Zeferino Vaz, 13083--859, Campinas, SP, Brazil}
\email{ddnovaes@unicamp.br}

\address{UNESP - Universidade Estadual Paulista, S\~{a}o Jos\'{e} do Rio Preto, S\~{a}o Paulo, Brazil}
\email{garv202020@gmail.com }



\maketitle

\begin{abstract}

Understanding how tangential singularities evolve under smoothing processes was one of the first problem concerning regularization of Filippov systems. In this paper, we are interested in $C^n$-regularizations of Filippov systems around visible regular-tangential singularities of even multiplicity. More specifically, using Fenichel Theory and Blow-up Methods, we aim to understand how the trajectories of the regularized system transits through the region of regularization. We apply our results to investigate $C^n$-regularizations of boundary limit cycles with even multiplicity contact with the switching manifold.
\end{abstract}

\section{Introduction}

The analysis of differential equations with discontinuous right-hand side dates back to the work of Andronov et. al \cite{AndronovEtAl66}  in 1937. Recently, the interest in such systems has increased significantly, mainly motivated by its wide range of applications in several areas of applied sciences. Piecewise smooth differential systems  are used for modeling  phenomena presenting abrupt behavior changes such as impact and friction in mechanical systems \cite{Brogliato16}, refuge and switching feeding preference in biological systems \cite{Krivan11,PiltzEtAl14}, gap junctions in neural networks \cite{Coombes08}, and many others.

In this paper, we are interested in planar piecewise smooth systems. Formally, let $M$ be an open subset of $\R^2$ and let $N\subset M$ be a codimension 1 submanifold of $M.$ Denote by $C_i,$ $i=1,2,\ldots,k,$ the connected components of $M\setminus N$ and let $X_i: M\rightarrow \R^2,$ for $i=1,2,\ldots,k,$ be vector fields defined on $M.$  A piecewise smooth vector field $Z$ on $M$ is defined by
\begin{equation}\label{dds}
Z(p)=X_i(p)\,\,\textrm{if}\,\, p\in C_i,\,\,\text{for}\,\, i=1,2,\ldots,k.
\end{equation}
Since $N$ is a codimension 1 submanifold of $M,$ for each $p\in N$ there exists a neighborhood $D\subset M$ of $p$ and a function $h:D\rightarrow\R,$ having $0$ as a regular value, such that $\Sigma=N\cap D=h^{-1}(0).$ Moreover, the neighborhood $D$ can be taken sufficiently small in order that $D\setminus \Sigma$ is composed by two disjoint regions $\Sigma^+$ and $\Sigma^-$ such that $X^+=Z|_{\Sigma^+}$ and $X^-=Z|_{\Sigma^-}$ are smooth vector fields. Accordingly, the piecewise smooth vector field \eqref{dds} can be locally described as follows:
\begin{equation*}\label{locdds}
Z(p)=(X^+,X^-)_{\Sigma}=\left\{\begin{array}{l}
X^+(p),\quad\textrm{if}\quad h(p) > 0,\vspace{0.1cm}\\
X^-(p),\quad\textrm{if}\quad h(p) < 0,
\end{array}\right. \quad \text{for}\quad p\in D.
\end{equation*}

\subsection{Filippov Systems} The notion of local trajectories of piecewise smooth vector fields \eqref{dds} was stated by Filippov \cite{Filippov88} as solutions of the following differential inclusion
\begin{equation}\label{FZ}
\dot p\in\CF_Z(p)=\dfrac{X^+(p)+X^-(p)}{2}+\sgn(h(p))\dfrac{X^+(p)-X^-(p)}{2},
\end{equation}
where
\[
\sgn(s)=\left\{
\begin{array}{ll}
-1&\text{if}\,\, s<0,\\

[-1,1]&\text{if}\,\, s=0,\\
1&\text{if}\,\, s>0.
\end{array}\right.
\]
This approach is called Filippov's convention. The piecewise smooth vector field \eqref{dds} is called  a Filippov system when it is ruled by the Filippov's convention. For more information on differential inclusions see, for instance, \cite{Smirnov02}. 

The solutions of the differential inclusion \eqref{FZ} are well described in the literature (see, for instance, \cite{Filippov88}) and have a simple geometrical interpretation. In order to illustrate this convention we define the following open regions on $\Sigma,$
\begin{equation*}\label{filippov}
\begin{array}{l}
\Sigma^c=\{p\in \Sigma:\, X^+h(p)\cdot X^-h(p) > 0\},\\

\Sigma^s=\{p\in\Sigma:\, X^+h(p)<0,\,X^-h(p) > 0\},\\

\Sigma^e=\{p\in\Sigma:\, X^+h(p)>0,\,X^-h(p) < 0\}.\\
\end{array} 
\end{equation*}
Here, $X^{\pm}h(p)=\langle\nabla h(p),X^{\pm}(p)\rangle$ denotes the Lie derivative of $h$ in the direction of the vector fields $X^{\pm}.$ Usually, they are called {\it crossing}, {\it sliding}, and {\it escaping} region, respectively. Notice that the points on $\Sigma$ where both vectors fields $X^+$ and $X^-$ simultaneously point outward or inward from $\Sigma$ constitute, respectively, the {\it escaping} $\Sigma^e$ and {\it sliding} $\Sigma^s$ regions, and the complement of its closure in $\Sigma$ constitutes the {\it crossing region}, $\Sigma^c.$ The complement of the union of those regions in $\Sigma$ constitutes the {\it tangency} points between $X^+$ or $X^-$ with $\Sigma,$ $\Sigma^t.$ 

For $p\in\Sigma^c$ the trajectories either side of the discontinuity $\Sigma,$ reaching $p,$ can be joined continuously, forming a trajectory that crosses $\Sigma^c.$ Alternatively, for $p\in \Sigma^{s,e}=\Sigma^s\cup \Sigma^e$ the trajectories either side of the discontinuity $\Sigma,$ reaching $p,$ can be joined continuously to trajectories that slide on $\Sigma^{s,e}$ following the sliding vector field,
\begin{equation}\label{slid}
Z^s(p)= \dfrac{X^- h(p) X^+(p)- X^+ h(p) X^-(p)}{X^- h(p) - X^+ h(p) },\,\, \text{for} \,\, p\in \Sigma^{s,e}.
\end{equation}
In addition, a singularity of the sliding vector field $Z^s$ is called pseudo-equilibrium. 

In the Filippov context, the notion of {\it$\Sigma$-singular points} also contains the tangential points $\Sigma^t$ constituted by the contact points between $X^+$ and $X^-$ with $\Sigma,$ that is, $\Sigma^t=\{p\in \Sigma:\, X^+h(p)\cdot X^-h(p) = 0\}.$ In this paper, we are interested in contact points of finite degeneracy. Recall that $p$ is a {\it contact of order $k-1$} (or multiplicity $k$) between $X^\pm$ and $\Sigma$ if $0$ is a root of multiplicity $k$ of $f(t)\defeq h\circ \f_{X^\pm}(t,p),$ where $t\mapsto \f_{X^\pm}(t,p)$ is the trajectory of $X^\pm$ starting at $p.$ Equivalently, $$X^\pm h(p) = (X^\pm)^2h(p) = \ldots = (X^\pm)^{k-1}h(p) =0,\text{ and } (X^\pm)^{k} h(p)\neq 0.$$
In addition, an even multiplicity contact, say $2k,$ is called {\it visible} for $X^+$ (resp. $X^-$) when $(X^{+})^{2k}h(p)>0$ (resp. $(X^{-})^{2k}h(p)<0$). Otherwise, it is called {\it invisible}.  In the above definitions, the higher order Lie derivatives $X^ih$ are defined, inductively, by $Xh(p)=\langle \nabla h(p),X(p)\rangle$ and $X^ih(p)=X(X^{i-1}h)(p)$ for $i>1.$

In this paper, we shall focus our attention on {\it visible regular-tangential singularities of multiplicity $2k$} which are formed by a visible even multiplicity of $X^+$ and a regular point of $X^-,$ or vice versa (see Figure \ref{figtan}). 

\begin{figure}[H]
	\begin{center}
		\begin{overpic}[scale=0.63]{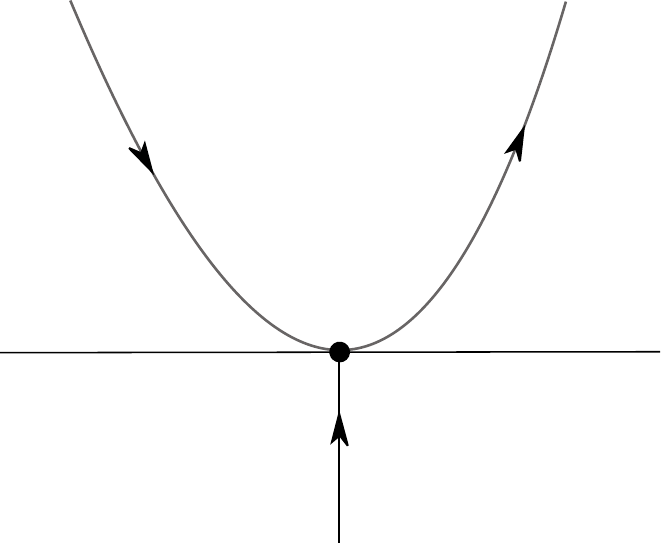}
		\end{overpic}
		\caption{Visible regular-tangential singularity.}
	\label{figtan}
	\end{center}
	\end{figure}

\subsection{Sotomayor-Teixeira Regularization}\label{sec:STreg} Roughly speaking, a smoothing process of a piecewise smooth vector field $Z$ consists in obtaining a one--parameter family of continuous vector fields $Z_{\e}$ converging to $Z$ when $\e\to 0.$ 

A well known smoothing process is the Sotomayor-Teixeira regularization, which was introduced in \cite{ST1995}. Let $\phi:\R\rightarrow\R$ be a $C^{\infty}$ function satisfying $\phi(\pm1)=\pm1,$ $\phi^{(i)}(\pm1)=0$ for $i=1,2,\ldots,n,$ and $\phi'(s)>0$ for $s\in(-1,1).$ Then, a $C^n$-Sotomayor-Teixeira regularization (or just  $C^n$-regularization for short) takes
\begin{equation}\label{regula}
Z_{\e}^{\Phi}(p)=\dfrac{1+\Phi_{\e}(h(p))}{2}X^+(p)+\dfrac{1-\Phi_{\e}(h(p))}{2}X^-(p), \,\, \Phi_{\e}(h)=\Phi(h/\e),
\end{equation}
where $\Phi:\R\rightarrow\R$ is defined as the following $C^n$ function
\begin{equation}\label{Phi}
\Phi(s)=\left\{\begin{array}{ll}
\phi(s)&\text{if}\quad|s|\leqslant 1,\\
\sgn(s)&\text{if}\quad|s|\geqslant1.
\end{array}\right.
\end{equation}
We call $\Phi$ a {\it $C^n$-monotonic transition function}. Proposition \ref{phi_n} provides examples of transition functions.

Notice that the vector field $Z_{\e}^{\Phi}(p)$ coincides with $X^+(p)$ or $X^-(p)$ whether $h(p)\geqslant\e$ or $h(p)\leqslant-\e,$ respectively. In the region $|h(p)|\leqslant \e,$ the vector $Z_{\e}^{\Phi}(p)$ is a linear combination of $X^+(p)$ and $X^-(p).$ 

The Sotomayor-Teixeira regularization is one of the most widespread smoothing process in the research literature. That is mainly because its intrinsic relation with Filippov's convention.  Indeed, in \cite{TexeiraSilva12}, it was shown that the Sotomayor-Teixeira regularization of Filippov systems gives rise to {\it Singular Perturbation Problems}, for which the corresponding reduced dynamics is conjugated to the sliding dynamics \eqref{slid}. This intrinsic relation has also been proven to hold for a more general class of transition function, which includes analytic regularizations (see \cite{LST09} and \cite[Remark 1.4]{kris2020}).  Also, non-monotonic transition functions have been considered in \cite{SMN18,NovJef15}, where a similar relation was obtained for sliding dynamics with hidden terms. In addition to our theoretical interested on Sotomayor-Teixeira regularization,  understanding how nonsmooth systems behave under this classical smooth process is a first step towards more general regularization process.

For more informations on Singular Perturbation Problems see, for instance, \cite{Fenichel79, Jones95}.

Finally, an anonymous referee has called our attention to the fact that the regularization process itself can be viewed in a blow-up setting, in which loss of smoothness along $\Sigma$ when $\e=0$ is resolved via cylindrical blow-up (see, for instance, \cite{BST06,CT11,kris2020,LST07,LST09,TexeiraSilva12}).

\subsection{Main Goal}
Understanding how tangential singularities  evolve under smoothing processes was one of the first problem concerning smoothing of Filippov systems. Indeed, in the earlier work of Sotomayor and Teixeira \cite{ST1995}, it is proven that around a regular-fold singularity of a Filippov system $Z,$ the regularized system $Z_{\e}^{\Phi}$  possesses no singularities. Recently, based on the findings in \cite{TexeiraSilva12}, some works got deeper results by  studying the corresponding slow-fast problems.

In \cite{BonetSeara16} and \cite{RevesEtAl18}, asymptotic methods \cite{MishchenkoRozov80} were used to study $C^n$-regularizations of generic regular-fold singularities and fold-fold singularities, respectively. In  \cite{KrisHogan15},  blow-up methods introduced in \cite{DurRou1996} was adapted to study $C^n$-regularizations of fold-fold singularities. In \cite{kris2017,kris2020},  these blow-up methods was employed to study more general regularizations of a visible regular-fold singularity, respectively.

In this paper, we are interested in $C^n$-regularizations of Filippov systems around visible regular-tangential singularities of general even multiplicity.  More specifically, we aim to understand how the trajectories of the regularized system transit through the regions $h(p)\geqslant \e,$ $|h(p)|\leqslant\e,$ and $h(p)\leqslant -\e.$ Accordingly, we characterize two transition maps, namely  the {\it Upper Transition Map} $U_{\e}(y)$ and the {\it Lower Transition Map} $L_{\e}(x)$ (see Figure \ref{figreg}).  The results are applied to study $C^n$-regularizations of boundary limit cycles with even multiplicity contact with the switching manifold.  

We emphasize that the approach on tangencies of general even multiplicity is the main novelty of the present study and generalizes some previous results on regularized visible regular-fold singularities \cite{BonetSeara16,kris2017,kris2020}.  Some of the main challenges when working in this degenerate context are: the lack of normal forms, which simplify significantly the analysis for the regular-fold singularity, whilst here we have to deal with higher order terms;  and the appearance of other tangencies in the regularized system.

\begin{figure}[h]
	\begin{center}
		\begin{overpic}[scale=0.7]{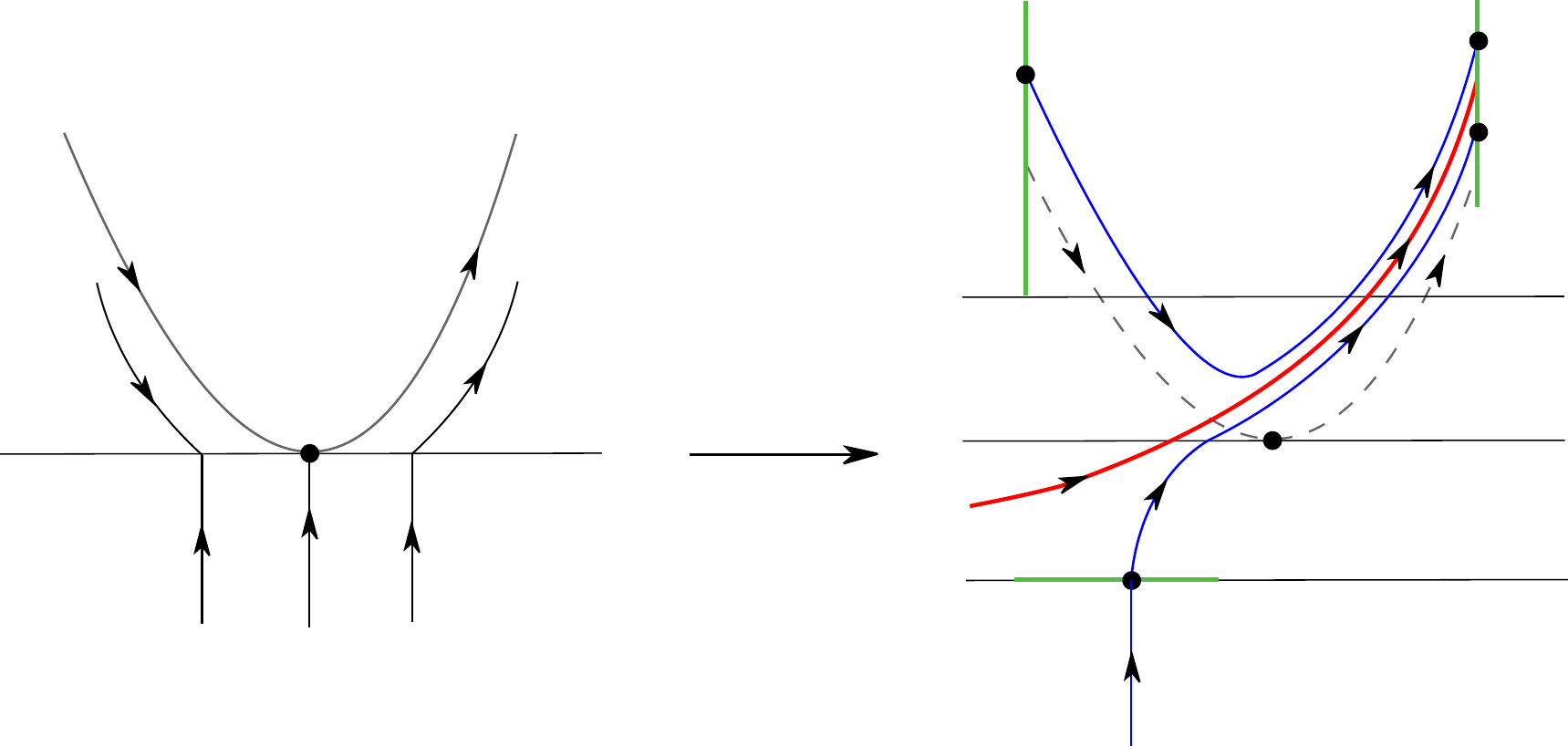}
		\put(70,8){$x$}
		\put(62,42){$y$}
		\put(96,38){$L_\e(x)$}
		\put(96,44){$U_\e(y)$}
\put(96,20){$\Sigma$}
		 \put(96,30){$y=\e$}
		\put(96,12){$y=-\e$}
		\put(40,21){{\small $\Phi-$regularization}}
		\put(35,19){$\Sigma$}
		\end{overpic}
		\caption{Upper Transition Map $U_{\e}(y)$ and Lower Transition Map $L_{\e}(x)$ defined for $C^n$-regularizations of Filippov systems around visible regular-tangential singularities of even multiplicity.}
	\label{figreg}
	\end{center}
	\end{figure}
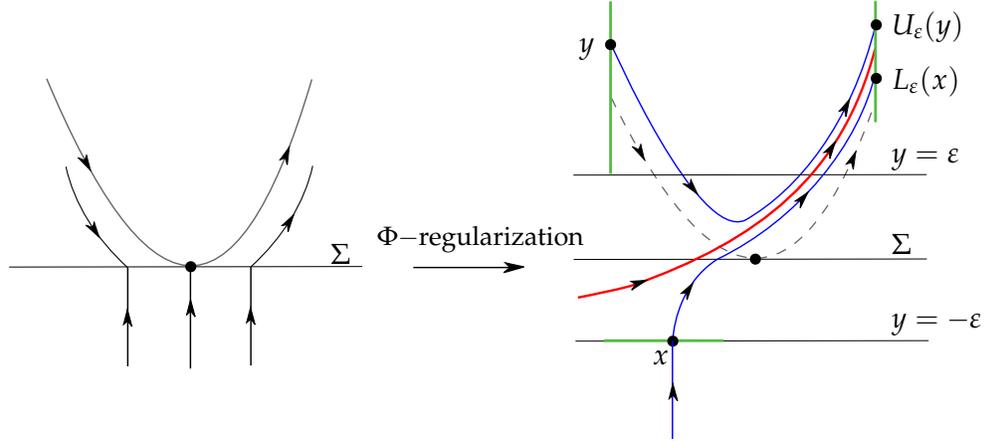

Our first two main results, Theorems \ref{ta} and \ref{tb}, characterize the {\it Upper Transition Map} $U_{\e}(y)$ and the {\it Lower Transition Map} $L_{\e}(x),$ respectively. Theorem \ref{ta} generalizes to regular-tangential singularities of general even multiplicity the results obtained in \cite{BonetSeara16} for regular-fold singularities.  The problem addressed in \cite{BonetSeara16} was also considered in \cite{kris2020} for a more general class of regularizations. Theorem \ref{tb} has not appeared in previous papers. Theorems \ref{ta} and \ref{tb} together characterize completely the behaviour of regularized systems around visible regular-tangential singularities of even multiplicity.

Finally, Theorem \ref{tc} provides sufficient conditions for the existence of an asymptotically stable limit cycle of the regularized system bifurcating from a boundary limit cycle of a Filippov system with degenerated contact with the switching manifold. This extends to our degenerate context  results obtained in \cite{BonetSeara16,kris2020} for boundary limit cycle with fold contact with the switching manifold.  Additional conditions on Theorem \ref{tc} ensuring the uniqueness of such a limit cycle will be provided in Section \ref{sec:nonexistence}.

It is worth mentioning that, in addition to boundary limit cycles studied in Theorem \ref{tc},  Theorems \ref{ta} and \ref{tb} allows to consider regularization of a large class of polycycles in Filippov systems \cite{NR20}.

\subsection{Structure of the paper} In Section \ref{sec:mainresults}, we state our main results Theorems \ref{ta} and \ref{tb}, which characterize the transition maps near $C^n$-regularizations of visible regular-tangential singularities, and Theorem \ref{tc} regarding $C^n$-regularizations of boundary limit cycles. In Section \ref{sec:canprel}, we provide a simpler local expression for Filippov systems around visible regular-tangential singularities as well as some preliminary results. In Section \ref{sec:fenichelmanifold}, we apply blow-up methods to study the Fenichel Manifold associated to the singular perturbation problem arising from $C^n$-regularizations of visible regular-tangential singularities. Then, Theorems \ref{ta}, \ref{tb}, and \ref{tc} are proven in Sections \ref{sec:uflightmap}, \ref{sec:lflightmap}, and \ref{sec:limitcycle}, respectively. In Section \ref{sec:nonexistence}, we study cases of uniqueness and nonexistence of limit cycles for the regularization of boundary limit cycles. Finally, in Section \ref{sec:example}, in light of our results, we perform an analysis of $C^n$-regularizations of piecewise polynomial examples admitting a boundary limit cycle. An Appendix is also provided with some additional computations.

\section{Main results}\label{sec:mainresults}
Let $X^{\pm}$ be $C^{2k},$  $k\geqslant 1,$ vector fields defined on an open subset $V$ of $\R^2$ and let $\Sigma$ be  a $C^{2k}$ embedded codimension one submanifold of $V.$ Suppose that $X^+$ has a visible  $2k$-multiplicity contact with $\Sigma$ at $(0,0)$ and that $X^-$ is pointing towards $\Sigma$ at $(0,0).$ Consider the Filippov system $Z=(X^+,X^-)_{\Sigma}.$ Denote  by $\varphi_{X^{\pm}}$ the flows of $X^{\pm}.$

First, we know that there exists a local $C^{2k}$ diffeomorphism $\varphi_{1}$ defined on a neighborhood  $U\subset\R^2$  of $(0,0)$ such that $\widetilde \Sigma=\varphi_1(\Sigma)=h^{-1}(0),$ with $h(x,y)=y.$ Second, applying the Tubular Flow Theorem for $(\varphi_1)_*X^-$ at $(0,0)$ and considering the transversal section $\widetilde \Sigma,$ there exists a local $C^{2k}$ diffeomorphism $\varphi_{2}$ defined on $U$ (taken smaller if necessary) such that $\widetilde X^-=(\varphi_2\circ\varphi_1)_*X^-=(0,1)$ and $\varphi_{2}(\widetilde\Sigma)=\widetilde\Sigma.$ Clearly, the transformed vector field $\widetilde X^+=(\varphi_2\circ\varphi_1)_*X^+$ still has a visible $2k$-multiplicity contact with $\widetilde \Sigma$ at $(0,0).$ Thus, without loss of generality, we can assume that the Filippov system $Z=(X^+,X^-)_{\Sigma}$ satisfies
\begin{itemize}
\item[{\bf (A)}]  $X^+$ has a visible  $2k$-multiplicity contact with $\Sigma$ at $(0,0),$ $X^+_1(0,0)>0,$ and there exists a neighborhood $U\subset\R^2$ of $(0,0)$ such that $X^-\big|_U=(0,1)$ and $\Sigma\cap U=\{(x,0):\, x\in(-x_U,x_U)\}.$
\end{itemize}

The next result establishes the intersection between the trajectory of $X^+$ (satisfying {\bf (A)}) starting at  $(0,0)$ with some sections (see Figure \ref{figpoints}).
\begin{lemma}\label{y0}
Assume that $X^+$ satisfies hypothesis {\bf (A)}. For $\rho>0,$ $\T>0,$ and $\e>0$ sufficiently small, the trajectory of $X^+$ starting at $(0,0)$ intersects transversally the sections $\{x=-\rho\},$ $\{x=\theta\},$ and $\{y=\e\},$ respectively, at $(-\rho,\ov y_{-\rho}),$ $(\T,\ov{y}_{\theta}),$ and $(\ov{x}^{\pm}_{\e},\e),$ where
\begin{equation}\label{secpoints}
\ov y_{x}=\frac{\al\, x^{2k}}{2k}+\CO(x^{2k+1}) \quad \text{and}\quad \ov{x}^{\pm}_{\e}=\pm\e^{\frac{1}{2k}}\left(\frac{2k}{\alpha}\right)^{\frac{1}{2k}}+\mathcal{O}(\e^{1+\frac{1}{2k}}).
\end{equation}
\end{lemma}

The lemma above is proven in Section \ref{sec:canprel}.
\begin{figure}[H]
	\begin{center} 
		\begin{overpic}[scale=0.23]{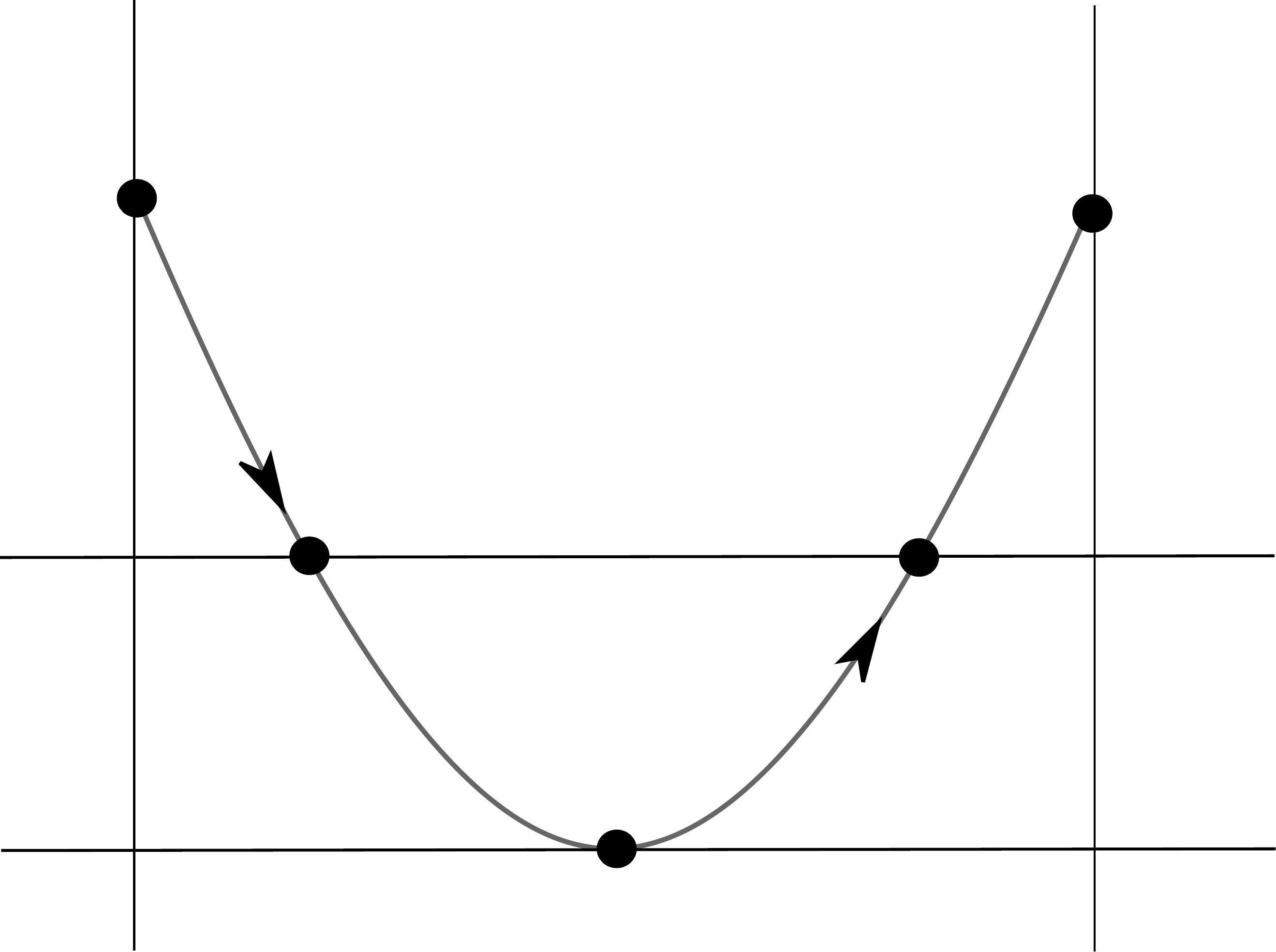}
		\put(73,25){$ \ov{x}^{+}_{\e}$}
		\put(18,25){$ \ov{x}^{-}_{\e}$}
		\put(89,58){$\ov y_{\theta}$}
		\put(-1,58){$\ov y_{-\rho}$}
\put(96,10){$\Sigma$}
		\put(96,33){$y=\e$}
		\put(41,2){$(0,0)$}
		\put(3,-3){$x=-\rho$}
		\put(79,-3){$x=\theta$}
		\end{overpic}
		
		\bigskip
		
		\caption{Transversal intersections of the trajectory of $X^+$ passing through the $2k$-multiplicity contact $(0,0)$ with the transversal sections $\{x=-\rho\},$ $\{x=\theta\},$ and $\{y=\e\}.$}
	\label{figpoints}
	\end{center}
	\end{figure}
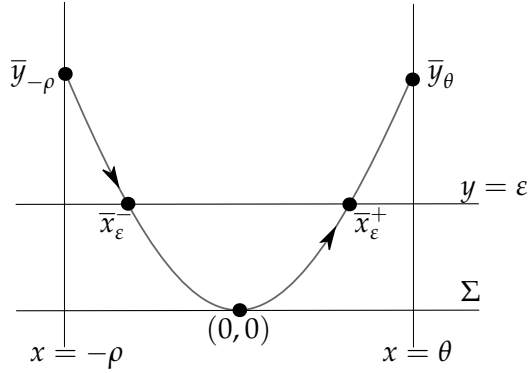
\subsection{Transition maps of the regularized system} We start by defining the set of $C^{n-1}$-monotonic transition functions which are not $C^{n}$ at $\pm 1.$ 

\begin{definition}\label{Cn-1ST}
Denote by $C^{n-1}_{ST}$ the set of $C^{n-1}$-monotonic transition functions $\Phi$ which are not $C^{n}$ at $\pm 1.$ That is, for a $\Phi\in C^{n-1}_{ST}$ given as \eqref{Phi}, then $\phi^{(i)}(\pm 1)=0,$ for $i=1,2,\ldots,n-1,$ and $\phi^{(n)}(\pm 1)\neq0.$ Moreover, one can easily see that $\sgn\big(\phi^{(n)}(\pm1)\big)=(\mp1)^{n+1}.$
\end{definition}

Our first two main results guarantee that under some conditions the flow of the regularized system $Z_{\e}^{\Phi}$ near a visible regular-tangential singularity defines two distinct maps between transversal sections (see Figure \ref{figreg}). Before their statements, we need to establish some notations. 

Given $\Phi\in C^{n-1}_{ST}$ as \eqref{Phi}, with $k\geqslant1,$ and $n\geqslant 2k-1$, define
\begin{equation*}\label{xe}
x_{\e}=\e^{\la^*}\eta+\mathcal{O}\left(\e^{\la^*+\frac{1}{1+2k(n-1)}}\right),
\end{equation*}
where $\la^*\defeq \frac{n}{1+2k(n-1)}$ and $\eta$ is a constant satisfying
 \begin{equation*}\label{sigmakn}
\eta>\left\lbrace\begin{array}{lllll}
0  & if & n>2k-1,\\
-\left(\dfrac{\partial_y X_2^+(0,0)}{\alpha}\right)^{\frac{1}{2k-1}} & if & n=2k-1 \hspace{0.1cm}\text{and}\hspace{0.1cm} k\neq 1,\\      
\end{array}\right.
\end{equation*} 
and
\begin{equation}\label{ye}
\begin{array}{l}
y^\e_{\rho,\la}=\overline{y}_{-\rho}+\e+\mathcal{O}(\e \rho)+\beta \e^{2k\la}+\mathcal{O}(\e^{(2k+1)\la})+\mathcal{O}(\e^{1+\la}), \vspace{0.1cm}\\

\displaystyle y_{\theta}^{\e} =\ov y_{\T}+\e+\mathcal{O}(\e\theta)+\sum_{i=1}^{2k-1}\mathcal{O}(\theta^{2k+1-i} x_\e^i)+\mathcal{O}(x_{\e}^{2k}),
\end{array}
\end{equation}
where $\ov y_{-\rho}$ and $\ov y_{\theta}$ are given by Lemma \ref{y0}, $\la\in(0,\la^*),$ $\rho$ and $\T$ are positive parameters depending (eventually) on $\e$, and $\beta$ is a negative constant which will be defined later on.

\begin{mtheorem}\label{ta}
Consider a Filippov system $Z=(X^+,X^-)_{\Sigma}$ and assume that $X^+$ satisfies hypothesis {\bf (A)} for some $k\geqslant 1.$ For $n\geqslant 2k-1,$ let $\Phi\in C^{n-1}_{ST}$ be given as \eqref{Phi} and consider the regularized system $Z_{\e}^{\Phi}$ \eqref{regula}. Then, there exist $\rho_0,\T_0>0,$ and constants  $\beta<0$ and $c,r>0,$ for which the following statements hold for every $\rho\in(\e^\la,\rho_0],$ $\T\in[x_\e,\T_0],$ $\la\in(0,\la^*),$ with $\la^*\defeq \frac{n}{2k(n-1)+1},$ $q=1-\dfrac{\lambda}{\lambda^*}\in(0,1),$ and $\e>0$ sufficiently small.
\begin{itemize}
\item[(a)] The vertical segments 
\[
\widehat V_{\rho,\la}^{\e}=\{-\rho\}\times [\e,y_{\rho,\la}^{\e}]\,\,\text{ and }\,\, \widetilde V_{\T}^{\e}=\{\T\}\times[y_\T^\e,y_\T^\e+r e^{-\frac{c}{\e^q}}]
\]
and the horizontal segments 
\[
\widehat H_{\rho,\la}^{\e}=[-\rho,-\e^{\la}]\times\{\e\}\,\,\text{ and }\,\, \overleftarrow{H}_{\e}=[x_{\e}-r e^{-\frac{c}{\e^q}},x_\e]\times\{\e\}
\]
are transversal sections for $Z_{\e}^{\Phi}.$

\item[(b)] The flow of $Z_{\e}^{\Phi}$ defines a map $U_{\e}$ between the transversal sections $\widehat V_{\rho,\la}^{\e}$ and $\widetilde V_{\T}^{\e}.$ 
\item[(c)] the trajectories of  $Z_{\e}^{\Phi}$ starting at the section $\widehat V_{\rho,\la}^{\e}$ intersect the line $y=\e$ only in two points before reaching the section $\widetilde V_{\T}^{\e}.$ Moreover, these intersections take place at $\widehat H_{\rho,\la}^{\e}\cup \overleftarrow{H}_{\e}.$
\end{itemize}
The map $U_{\e}$ is called {\it Upper Transition Map} of the regularized system (see Figure \ref{figMAP1}).
\end{mtheorem}

Notice that, in Theorem \ref{ta}, taking item (a) into account, the map $U_{\e}: \widehat V_{\rho,\la}^{\e} \longrightarrow \widetilde V_{\T}^{\e}$ provided by item (b) writes $U_{\e}(y)=y_{\T}^{\e}+\CO(e^{-\frac{c}{\e^q}}).$

\begin{remark}
As commented before, versions of Theorem \ref{ta} for $k=1$ have already been obtained in \cite{BonetSeara16,kris2017,kris2020}. 
Here, for $k=1,$ we have that 
\[
\begin{aligned}
 y_{\theta}^{\e} =&\ov y_{\T}+\e+\mathcal{O}(\e\T)+\mathcal{O}(\T^2x_\e)+\CO(x_\e^2)\\
 =&\dfrac{\alpha\theta^{2}}{2}+\e+\CO(x_\e^2)+\CO(\T^3)+\CO(\e\T)+\CO(\T^2x_\e),
 \end{aligned}
 \]
where $\T\in[x_\e,\T_0]$ can be taken depending on $\e$. This expression coincides with the one obtained in \cite[Theorem 3.3]{kris2017} up to the higher order terms $\CO(\T^3)+\CO(\e\T)+\CO(\T^2x_\e)$.
Such a difference is due to the fact that \cite{kris2017} deals with the normal form of the visible regular-fold singularity, for which such higher order terms vanish. Later on, when applying Theorem \ref{ta}, we shall take $\T=x_{\e}$. In this case, $ y_{\theta}^{\e}=\dfrac{\alpha\theta^{2}}{2}+\e+\CO(x_\e^2)$ for $k=1$, which coincides with the expression provided by \cite{kris2017}.
\end{remark}

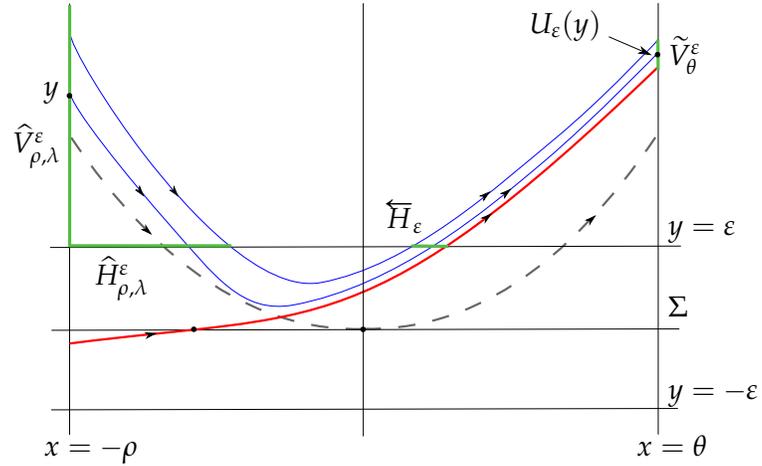
\begin{figure}[H]
	\begin{center}
		\begin{overpic}[scale=0.35]{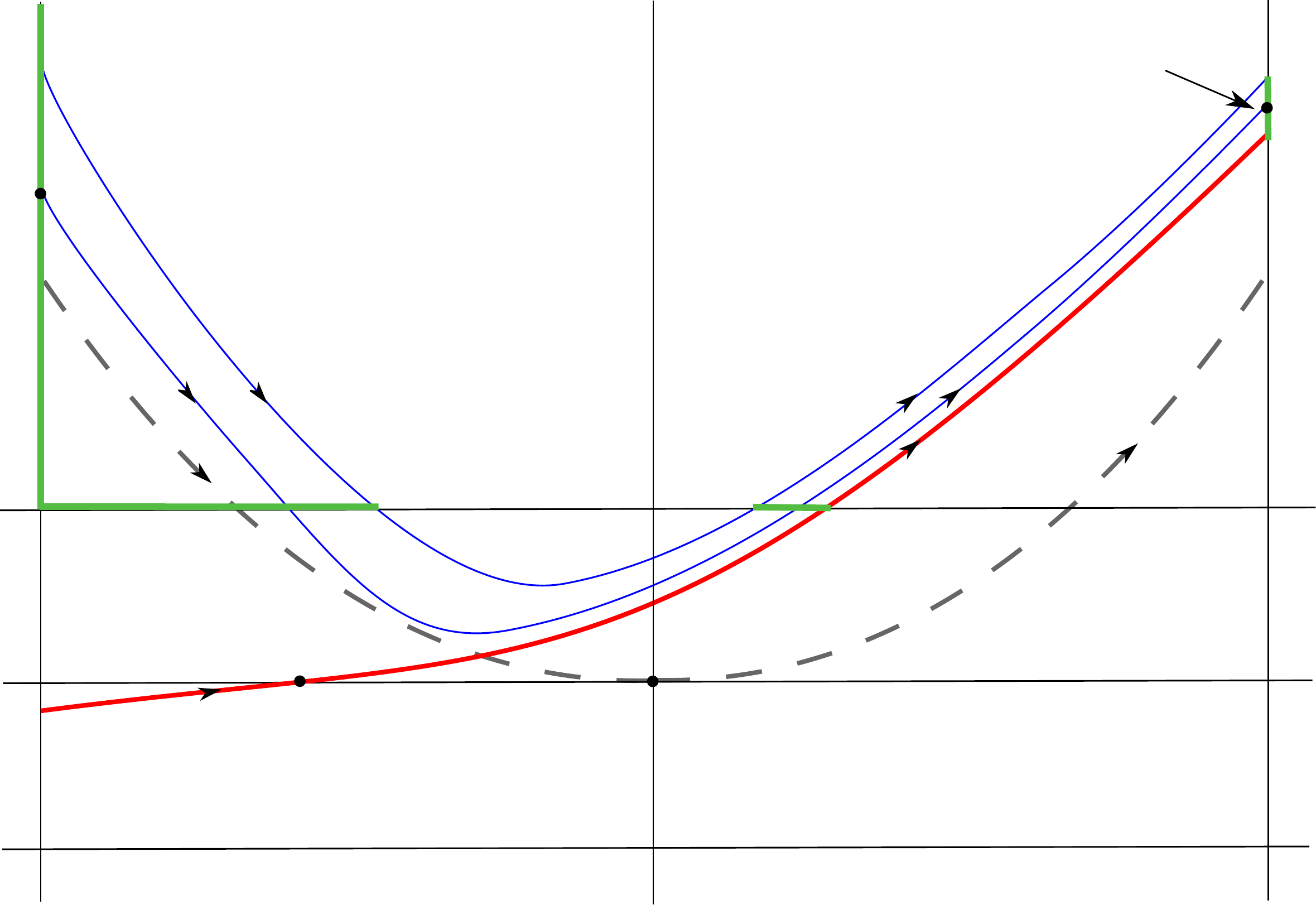}
		\put(-6,45){$\widehat V_{\rho,\la}^{\e}$}
		\put(7,24){$\widehat H_{\rho,\la}^{\e}$}
		\put(98,59){$\widetilde V_{\T}^{\e}$}
		\put(53,33){$\overleftarrow{H}_{\e}$}
		\put(93,-3){$x=\theta$}
		\put(-1,-3){$x=-\rho$}
		\put(98,19){$\Sigma$}
		\put(98,32){$y=\e$}
		\put(98,6){$y=-\e$}
		\put(76,64){$U_{\e}(y)$}
		\put(-1,54){$y$}
		\end{overpic}
	\end{center}
	
	\bigskip
	
	\caption{Upper Transition Map $U_{\e}$ of the regularized system $Z^\Phi_\e.$ The large domain $\widehat V_{\rho,\la}^{\e}$ is contracted into the small $\widetilde V_{\T}^{\e}.$ The dotted curve is the trajectory of $X^+$ passing through the visible $2k$-multiplicity contact with $\Sigma$ with $(0,0).$}
	\label{figMAP1}
	\end{figure}
	
	For the sake of completeness we also characterize the Lower Transition Map.

	 \begin{mtheorem}\label{tb}
	 Consider a Filippov system $Z=(X^+,X^-)_{\Sigma}$ and assume that $X^+$ satisfies hypothesis {\bf (A)} for some $k\geqslant 1.$ For $n\geqslant 2k-1,$ let $\Phi\in C^{n-1}_{ST}$ be given as \eqref{Phi} and consider the regularized system $Z_{\e}^{\Phi}$ \eqref{regula}. Then, there exist $\rho_0,\T_0>0,$ and constants $c,r>0,$ for which the following statements hold for  every $\rho\in(\e^\la,\rho_0],$ $\T\in[x_\e+re^{-\frac{c}{\e^q}},\T_0],$ $\la\in(0,\la^*),$ with $\la^*= \frac{n}{2k(n-1)+1},$ $q=1-\dfrac{\lambda}{\lambda^*}\in(0,1),$ and $\e>0$ sufficiently small.
	 \begin{itemize}
\item[(a)] The vertical segment
\[
\widecheck V_{\T}^{\e}=\{\T\}\times[y_\T^\e-r e^{-\frac{c}{\e^q}},y_\T^\e]
\]
and the horizontal segments 
\[
\widecheck H_{\rho,\la}^{\e}=[-\rho,-\e^{\la}]\times\{-\e\}\,\,\text{ and }\,\, \overrightarrow{H}_{\e}=[x_{\e},x_{\e}+r e^{-\frac{c}{\e^q}}]\times\{\e\}
\]
are transversal sections for $Z_{\e}^{\Phi}.$

\item[(b)] The flow of $Z_{\e}^{\Phi}$ defines a map $L_{\e}$ between the transversal sections $ \widecheck H_{\rho,\la}^{\e}$ and $\widecheck V_{\T}^{\e}.$ \end{itemize}
The map $L_{\e}$ is called {\it Lower Transition Map} of the regularized system (see Figure \ref{figMAP11}).
\end{mtheorem}

Notice that, in Theorem \ref{tb}, taking item (a) into account, the map $L_{\e}: \widecheck H_{\rho,\la}^{\e} \longrightarrow \widecheck V_{\T}^{\e}$ provided by item (b) writes $L_{\e}(x)=y_{\T}^{\e}+\CO(e^{-\frac{c}{\e^q}}).$

\begin{figure}[H]
	\begin{center}
		\begin{overpic}[scale=0.37]{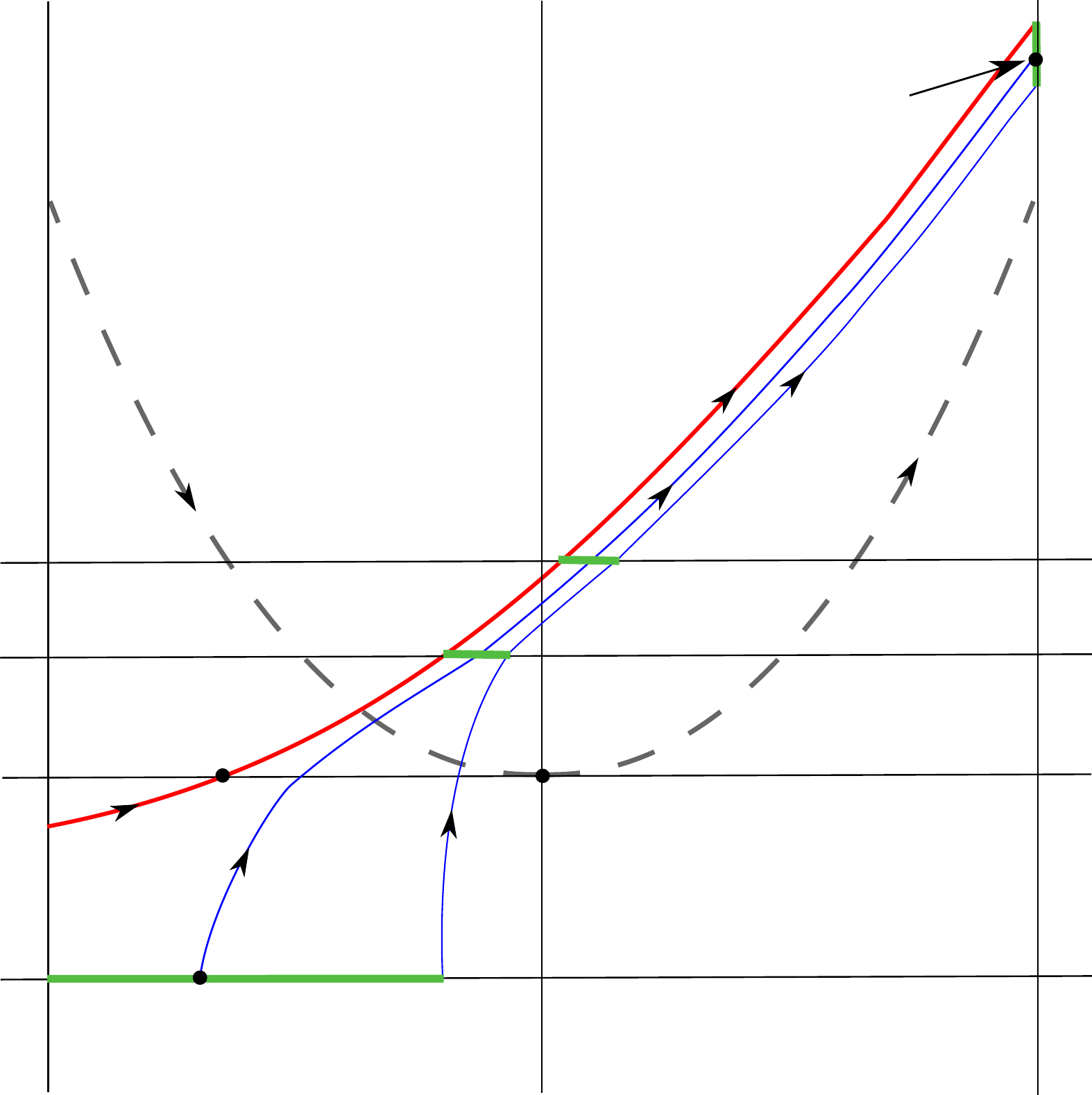}
\put(20,12){$x$}
	  \put(15,3){$\widecheck H_{\rho,\la}^{\e}$}
		\put(96,93){$\widecheck V_{\T}^{\e}$}
		\put(68,88){$L_\e(x)$}
		\put(59,42){$\overrightarrow{H}_{\e}$}
		\put(88,-4){$x=\theta$}
		\put(-3,-4){$x=-\rho$}
		\put(96,31){$\Sigma$}
		\put(96,51){$y=\e$}
		\put(96,13){$y=-\e$}
		\put(96,42.5){$y=\e\widehat{y}_0$}
		\end{overpic}
	\end{center}
	
	\bigskip
	
	\caption{Lower Transition Map $L_{\e}$ of the regularized system $Z^\Phi_\e.$ The large domain $\widecheck H_{\rho,\la}^{\e}$ is contracted into the small $\widecheck V_{\T}^{\e}.$ The dotted curve is the trajectory of $X^+$ passing through the visible $2k$-multiplicity contact with $\Sigma$ with $(0,0).$}
	\label{figMAP11}
	\end{figure}
	
\begin{remark}\label{assumption}
The proofs of Theorems \ref{ta} and \ref{tb} are based on the analysis of the corresponding slow-fast problem associated with the regularized system $Z_{\e}^{\Phi}$ \eqref{regula}  (see Section \ref{sec:STreg}). This analysis relies on the normal hyperbolicity of a related critical manifold. When $n\geqslant\max\{2, 2k-1\},$ we shall see that this critical manifold loses its normal hyperbolicity. This problem is overcome by means of blow-up methods. When $k=n=1,$ we do not face such a problem and the results are directly obtained from Fenichel Theory. Thus, through out the paper, we shall assume that $n\geqslant\max\{2, 2k-1\}.$
\end{remark}
	
\subsection{Regularization of boundary limit cycles} 
Consider a Filippov system $Z=(X^+,X^-)$ and assume that

\begin{itemize} 
\item[{\bf (B)}]\label{H} $X^+$ has a hyperbolic limit cycle $\Gamma,$ which has a $2k$-multiplicity contact with $\Sigma$ at $(0,0)$ and $X^-$ is pointing towards $\Sigma$ at $(0,0).$ In other words, $(0,0)$ is a visible regular-tangential singularity of $Z$ (see Figure \ref{figPol}).
\end{itemize}

\begin{figure}[h]
	\begin{center}
		\begin{overpic}[scale=0.37]{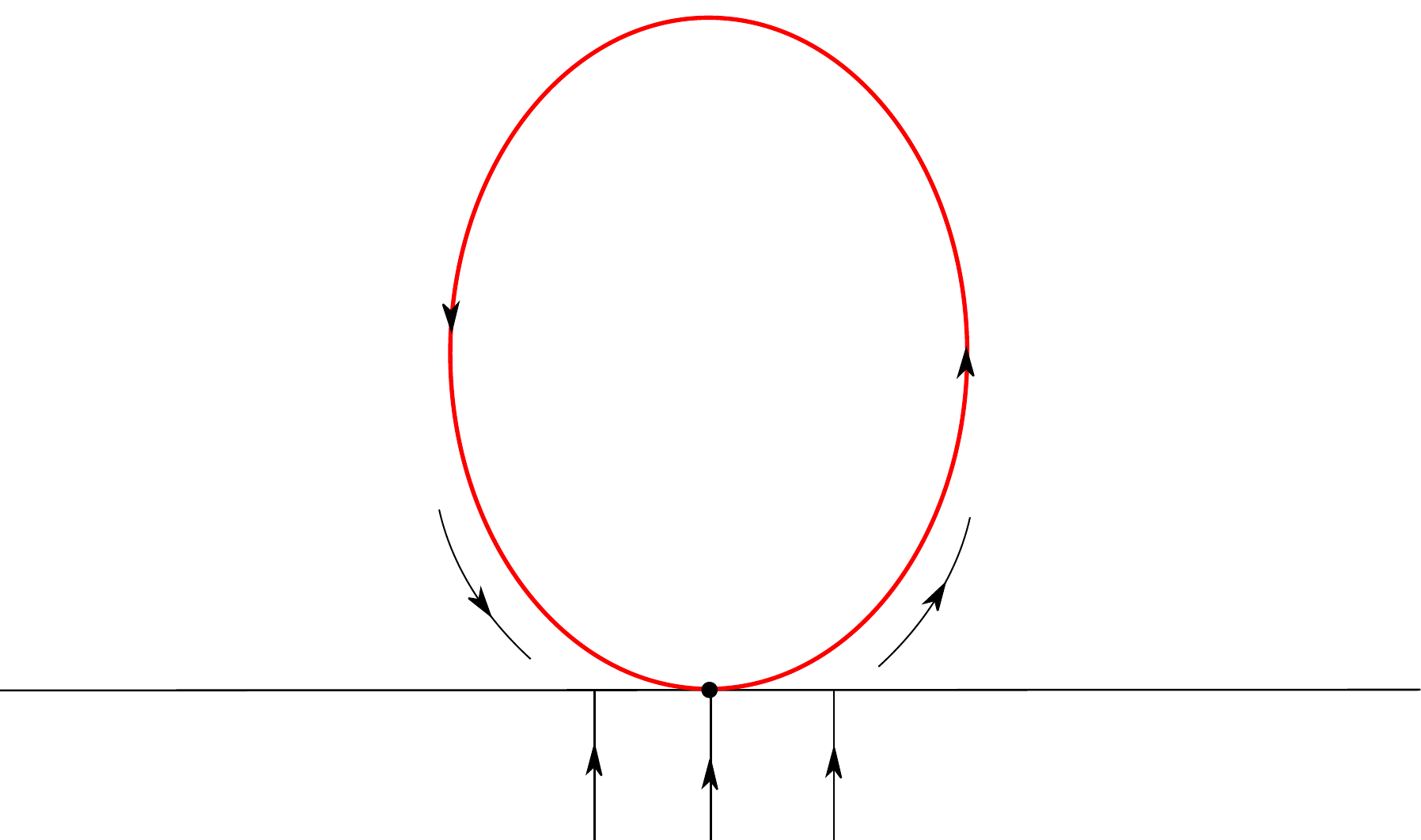}
		\put(51,6){$0$}
\put(90,12){$\Sigma$}
		\put(65,50){$\Gamma$}
		\end{overpic}
		\caption{Boundary limit cycle of $Z.$}
	\label{figPol}
	\end{center}
	\end{figure}

Our third main result establishes conditions under which the regularized vector field  $Z_{\e}^{\Phi}$ has an asymptotically stable limit cycle $\Gamma_{\e}$ converging to $\Gamma.$

\begin{mtheorem}\label{tc}
Consider a Filippov system $Z=(X^+,X^-)_{\Sigma}$ and assume that $X^+$ satisfies hypothesis {\bf (B)} for some $k\geqslant 1.$ For $ n\geqslant 2k-1,$ let $\Phi\in C^{n-1}_{ST}$ be given as $\eqref{Phi}.$ Then, the following statements hold.
\begin{enumerate}
	\item[(a)] Given $0<\la<\la^*=\frac{n}{1+2k(n-1)},$ if the limit cycle $\Gamma$ is unstable, then there exists $\rho>0$ such that  the regularized system $Z_{\e}^{\Phi}$ \eqref{regula} does not admit limit cycles passing through the section  $\widehat H_{\rho,\la}^{\e}=[-\rho,-\e^{\la}]\times\{\e\},$ for $\e>0$ sufficiently small.
	\item[(b)] Given $\frac{1}{2k}<\la<\la^*=\frac{n}{1+2k(n-1)},$ if the limit cycle $\Gamma$ is asymptotically stable, then there exists $\rho>0$ such that the regularized system $Z_{\e}^{\Phi}$ \eqref{regula} admits a unique limit cycle $\Gamma_{\e}$ passing through the section $\widehat H_{\rho,\la}^{\e}=[-\rho,-\e^{\la}]\times\{\e\},$ for $\e>0$ sufficiently small. Moreover, $\Gamma_{\e}$ is asymptotically stable and $\e$-close to $\Gamma$ in the Hausdorff distance. 
\end{enumerate}
\end{mtheorem}

\begin{remark}Versions of Theorem \ref{tc} for $k=1$ have already been obtained in \cite{BonetSeara16,kris2020}. More specifically, these papers considered regularizations of a 1-parameter family of Filippov systems $Z_{\alpha}$ undergoing a grazing bifurcation, for which the configuration we are studying appears for the critical value of the parameter $\alpha=0$. In particular, assuming that the critical limit cycle of $Z_{0}$ is unstable, \cite[Theorem 1.7]{kris2020} obtained the existence of a continuous curve $\alpha=\ov \alpha(\e),$ with $\ov\alpha(0)=0,$ of saddle-node bifurcations, such that 
two limit cycles exist for $\alpha<\ov\alpha(\e)$ and precisely one for $\alpha=\ov \alpha(\e)$. Comparing \cite[Theorem 1.7]{kris2020} with our Theorem \ref{tc}, it follows that $\ov \alpha(\e)<0.$
\end{remark}

\begin{remark}\label{2assumption2}
The results provided by Theorem \ref{tc} strongly rely on the topological type of the regular-tangential singularity. Namely, hypothesis {\bf (B)} implies that any neighborhood of the regular-tangential singularity has non-empty intersection with the sliding region $\Sigma^s$ of the switching manifold $\Sigma$ (see Figure \ref{fig_s_e}a).

In hypothesis {\bf (B)}, if it is assumed that $X^-$ points outwards $\Sigma$ at $(0,0),$ then any neighborhood of the regular-tangential singularity has non-empty intersection with the the escaping region $\Sigma^e$ of the switching manifold $\Sigma$ (see Figure \ref{fig_s_e}b). In this case, a version of Theorem \ref{tc} can be obtained by reversing the time (that is, by multiplying the vector field by $-1$). Accordingly,  the regularized system $Z_{\e}^{\Phi}$ \eqref{regula} admits an unstable limit cycle $\Gamma_{\e}$, as stated in item (b),  provided that the limit cycle $\Gamma$ is unstable. Otherwise, the regularized system $Z_{\e}^{\Phi}$ does not admits limit cycles, as stated in item (a).

\begin{figure}[h]
	\begin{center}
		\begin{overpic}[scale=0.6]{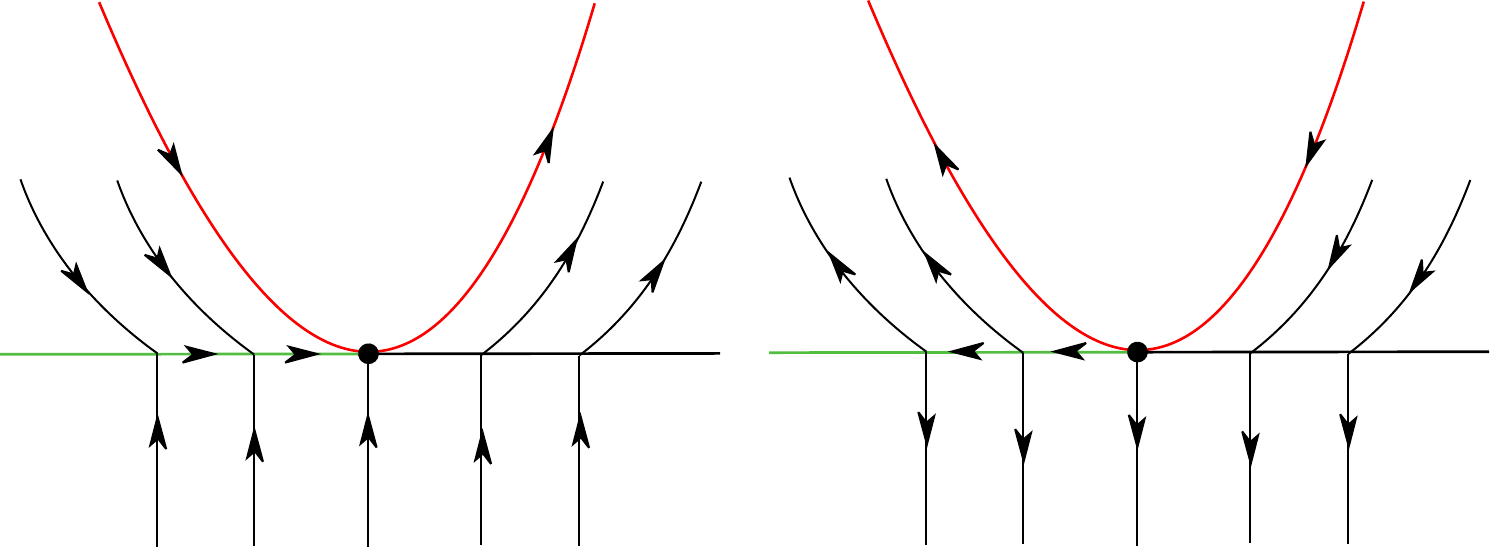}
	
\put(45,14){$\Sigma^c$}
\put(96,14){$\Sigma^c$}
\put(0,14){$\Sigma^s$}
\put(53,14){$\Sigma^e$}
	\put(24,-5){$(a)$}
	\put(76,-5){$(b)$}
		\end{overpic}
		\bigskip
		\caption{Topological types of visible regular-tangential singularities of even multiplicity. In $(a)$, $X^-$ points towards $\Sigma$ at $(0,0),$ then any neighborhood of the regular-tangential singularity has non-empty intersection with the sliding region $\Sigma^s$. In $(b)$, $X^-$ points outwards $\Sigma$ at $(0,0),$ then any neighborhood of the regular-tangential singularity has non-empty intersection with the escaping region $\Sigma^e$.}
	\label{fig_s_e}
	\end{center}
	\end{figure}
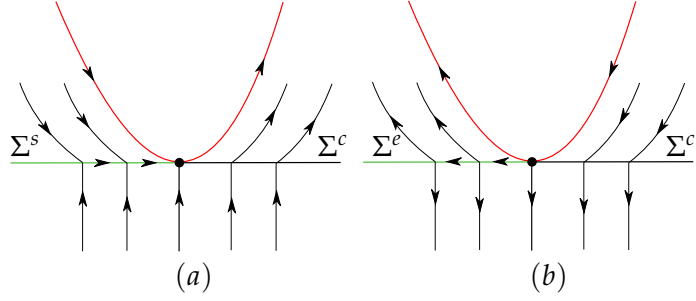

\end{remark}

\begin{remark}\label{assumption2}
 Statement (a) and (b) of Theorem \ref{tc} guarantee, respectively, the nonexistence and uniqueness of limit cycles in a specific compact set with nonempty interior. However, since this set degenerates into $\Gamma$ when $\e$ goes to $0,$ it is not ensured, in general, the nonexistence and uniqueness of limit cycles converging to $\Gamma$. Nevertheless, if
we assume, in addition, that $X^+$ has locally a unique isocline $x=\psi(y)$ of $2k-$multiplicity contacts with the straight lines $y=cte,$ then we get the nonexistence and uniqueness of limit cycles converging to $\Gamma$ (see Section \ref{sec:nonexistence}).
\end{remark}

\section{Canonical Form and Preliminary Results} \label{sec:canprel}

In this section, we first provide a simpler local expression for Filippov systems satisfying hypothesis {\bf (A)} in a neighborhood of the visible regular-tangential singularity. Denote $X^{\pm}=(X_1^{\pm},X_2^{\pm}).$ Since $X_1^+(0,0)> 0,$ we can take the neighborhood $U$ smaller in order that $X_1^+(x,y)> 0$ for all $(x,y)\in U.$ 
Performing a time rescaling in $X^+,$ we get $\widehat{X}^+(x,y)=(1,f(x,y)),$ with the function $f$ given by $f(x,y)=X_2^+(x,y)/X_1^+(x,y).$  Clearly, the vector fields $X^+$ and $\widehat{X}^+$ have the same orbits in $U$ with the same orientation. Notice that, for $(x,y)\in U,$ we have
\[
\begin{array}{rl}
X^+h(x,y) = & X_2^+(x,y)\\ 
= & X_1^+(x,y)f(x,y)\\
=& X_1^+(x,y)\widehat{X}^+h(x,y).\\   
\end{array}
\]
In general, $(\widehat{X}^+)^i h(0,0)=0$ if, and only if, $(X^+)^ih(0,0)=0,$ for all $i=1,\ldots,2k.$ Moreover, 
\[
\begin{array}{l}
\widehat{X}^+h(x,0)=f(x,0)\quad\text{and}\vspace{0.2cm}\\
\displaystyle(\widehat{X}^+)^i h(0,0)=\frac{\partial^{i-1}f}{\partial x^{i-1}}(0,0), \hspace{0.1cm}\forall i=1,\ldots,2k.
\end{array}
\]
Therefore, expanding $f(x,0)$ around $x=0,$ we get
\[
f(x,0)=\sum_{i=0}^{2k-1} \frac{1}{i!}\frac{\partial^{i}f}{\partial x^{i}}(0,0)x^{i}+g(x)=\al x^{2k-1}+g(x),
\]
where $\alpha=\dfrac{(\widehat{X}^+)^{2k} h(0,0)}{(2k-1)!}>0$ and $g(x)=\CO(x^{2k})$ is a $C^{2k}$ function. Consequently, the function $f(x,y)$ writes
\[
f(x,y)=\alpha x^{2k-1}+g(x)+y\vartheta(x,y),
\]
where $\vartheta$ is a $C^{2k}$ function. Finally, dropping the hat, the Filippov system $Z=(X^+,X^-)_{\Sigma}$ on $U$ becomes  
\begin{equation}\label{Xnf}
\begin{array}{l}
X^{+}(x,y)=(1,\alpha x^{2k-1}+g(x)+y\vartheta(x,y)),\vspace{0.2cm}\\
X^{-}(x,y)=(0,1),
\end{array}
\end{equation}
with $\al>0.$ Moreover, $\partial_y X^+_2(0,0)=\vartheta(0,0).$

Now, we are ready to prove Lemma \ref{y0}.

\begin{proof}[Proof of Lemma \ref{y0}]
Let us consider the differential equation induced by the vector field $X^+$
\begin{equation}\left\lbrace\begin{array}{rl}\label{cs}
x'  = & 1,\\
y'  = & \alpha x^{2k-1}+g(x)+y\vartheta(x,y).\\    
\end{array}\right.\end{equation}
Denote by $(x(t),y(t))$ the solution of system \eqref{cs} satisfying $x(0)=0$ and $y(0)=0.$ Thus, $x(t)=t$ and $y(t)$ satisfies the following differential equation
\begin{equation*}\label{eq0}
y' =  \alpha t^{2k-1}+g(t)+y\vartheta(t,y).
\end{equation*}
Therefore, $y^{(i)}(0)=0$ for $i=0,1,\ldots,2k-1$ and $y^{(2k)}(0)=(2k-1)!\alpha.$ Thus, the Taylor series of  $y(t)$ around $t=0$ writes
$$y(t)=\frac{\alpha t^{2k}}{2k}+\mathcal{O}(t^{2k+1}).$$ 
Hence, taking $\rho>0$ and $\T>0$ sufficiently small, we conclude that the trajectory of $X^+$ starting at $(0,0)$ intersects the sections $\{x=-\rho\}$ and $\{x=\theta\}$ at the points defined in \eqref{secpoints} $(-\rho,\ov y_{-\rho})$ and $(\T,\ov y_{\T}),$ respectively. These intersections are transversal, because $X^+_1(x,y)=1$ for every $(x,y)\in U.$

Now, we shall study the intersection $y(t)=\e,$ so define $\kappa(t,\e)=y(t)-\e.$ Consider the change of variables $s=t^{2k}$ and define the function 
\[
\zeta(s,\e)=\kappa(s^{\frac{1}{2k}},\e)=\dfrac{\alpha s}{2k}-\e+\mathcal{O}(s^{\frac{2k+1}{2k}}).
\] 
Since $\zeta(0,0)=0$ and $\frac{\partial\zeta}{\partial s}(0,0)=\frac{\alpha}{2k}>0,$ by the \textit{Implicit Function Theorem}, there exists a unique smooth function $s(\e)$ such that $\zeta(s(\e),\e)=0$ and $s(0)=0.$ Moreover, 
$$s'(0)=-\dfrac{\dfrac{\partial\zeta}{\partial \e}(0,0)}{\dfrac{\partial\zeta}{\partial s}(0,0)}=\dfrac{2k}{\alpha}.$$  
Thus, the Taylor expansion of $s(\e)$ around $\e=0$ writes
$$s(\e)=\e\dfrac{2k}{\alpha}+\mathcal{O}(\e^2).$$ 
Since, $s(\e)>0$ for $\e>0$ sufficiently small, we can define $t^{\pm}(\e)=\pm(s(\e))^{\frac{1}{2k}}.$ Therefore, 
$$t^{\pm}(\e)=\pm\e^{\frac{1}{2k}}\left(\frac{2k}{\alpha}\right)^{\frac{1}{2k}}+\mathcal{O}(\e^{1+\frac{1}{2k}}).$$
Hence, the trajectory of $X^+$ starting at $(0,0)$ intersects the section $\{y=\e\}$ at the points $\left(\ov{x}^{\pm}_{\e},\e\right)$ defined in \eqref{secpoints}. 
We conclude this proof by showing that these intersections are transversal for $\e>0$ small enough. Indeed, suppose that
$X^+_2\left(\ov{x}^{\pm}_{\e},\e\right)=0.$  Thus,
$$\alpha (\ov{x}^{\pm}_{\e})^{2k-1}+(\ov{x}^{\pm}_{\e})^{2k-1}\widetilde{g}(\ov{x}^{\pm}_{\e})+\e \vartheta(\ov{x}^{\pm}_{\e},\e)=0,$$ and, consequently, $(\ov{x}^{\pm}_{\e})^{2k-1}=-\frac{\e \vartheta(\ov{x}^{\pm}_{\e},\e)}{\alpha+\widetilde{g}(\ov{x}^{\pm}_{\e})},$ where $\widetilde{g}=\CO(x)$ is a continuous function such that $g(x)=x^{2k-1}\widetilde{g}(x).$ Thus,
$$\left|(\ov{x}^{\pm}_{\e})^{2k-1}\right|=\left| \frac{\vartheta(\ov{x}^{\pm}_{\e},\e)}{\alpha+\widetilde{g}(\ov{x}^{\pm}_{\e})}\right|\e\leqslant \max_{\e\in[0,\e_0], x\in \ov{B}}\left| \frac{\vartheta(x,\e)}{\alpha+\widetilde{g}(x)}\right|\e= C\e,
$$ 
which implies that $\ov{x}^{\pm}_{\e}=\mathcal{O}(\e^{\frac{1}{2k-1}})$ and, therefore, $2k/\al=0.$ This is a contradiction. Here, $B\subset\R$ is a neighbourhood of $0.$ Hence, $X^+_2\left(\ov{x}^{\pm}_{\e},\e\right)\neq0$ for $\e>0$ sufficiently small.\end{proof}

The next lemma is a technical result which will be useful for proving our main Theorems.

\begin{lemma} \label{lemma2} Let $\sigma$ be a real number. The trajectory $(u(t),v(t))$ of the planar vector field $F(u,v)=(1,-u^{2k-1}-v^n+\sigma)$ satisfying  $u(0)=u_{0}$ and $v(0)=v_{0}>0$ intersects $v=0$ at the point $(u^*,0)$ with $u^*>\sigma^{\frac{1}{2k-1}}.$
\end{lemma}
\begin{proof}
For each positive real number $\mu,$ with $\mu^n>\sigma,$ let $\mathcal{B}_{\mu}\subset\R^2$ be defined as the following compact region,
$$\mathcal{B}_{\mu}=\left\{(u,v)\big|(-\mu^n+\sigma)^{\frac{1}{2k-1}}\leqslant u\leqslant-v+\mu+(1+\sigma+\delta)^\frac{1}{2k-1}, 0\leqslant v \leqslant \mu\right\},$$
where $\delta>0$ is such that $1+\sigma+\delta>0$ (see Figure \ref{fig}). 

First, we shall see that the trajectories of $F$ enter the region $\mathcal{B}_{\mu}$ through $\p \mathcal{B}_{\mu}\setminus \mathcal{L}_{\mu},$ where $\mathcal{L}_{\mu}=\{(u,v)| \sigma^{\frac{1}{2k-1}}\leqslant u\leqslant \mu+(1+\sigma+\delta)^{\frac{1}{2k-1}}, v=0\}.$ Denote
\[
\begin{array}{rl}
\mathcal{B}^{+}_{\mu}=&\left\{(u,v)\big| u=-v+\mu+(1+\sigma+\delta)^{\frac{1}{2k-1}}, 0\leqslant v \leqslant \mu\right\},\\

\mathcal{B}^{-}_{\mu}=&\left\{(u,v)\big| u=(-\mu^n+\sigma)^{\frac{1}{2k-1}}, 0\leqslant v < \mu\right\},\\

\mathcal{B}^{*}_{\mu}=&\left\{(u,v)\big| (-\mu^n+\sigma)^{\frac{1}{2k-1}}< u\leqslant (1+\sigma+\delta)^{\frac{1}{2k-1}}, v=\mu\right\},\\

\mathcal{B}^{\#}_{\mu}=&\left\{(u,v)\big| (-\mu^n+\sigma)^{\frac{1}{2k-1}}\leqslant u< \sigma^{\frac{1}{2k-1}}, v=0\right\}.
\end{array}
\]
Notice that $\p \mathcal{B}_{\mu}\setminus \mathcal{L}_{\mu}=\mathcal{B}^{+}_{\mu}\cup \ov{ \mathcal{B}^{-}_{\mu}}\cup  \ov{\mathcal{B}^{*}_{\mu}}\cup \mathcal{B}^{\#}_{\mu}.$

Let $n^+=(1,1)$ be a normal vector to $\mathcal{B}^{+}_{\mu}.$ Since $F\big|_{\mathcal{B}^{+}_{\mu}}=(1,-u^{2k-1}-v^n+\sigma),$ we get
	$$\begin{array}{lllll}
\langle n^+,F\rangle  &= & \langle(1,1),(1,-u^{2k-1}-v^n+\sigma)\rangle\\
&\leqslant & 1-u^{2k-1}+\sigma\\
& = &\displaystyle 1-(-v+\mu+(1+\sigma+\delta)^{\frac{1}{2k-1}})^{2k-1}+\sigma\\ 
& \leqslant & 1+(\mu-\mu-(1+\sigma+\delta)^{\frac{1}{2k-1}})^{2k-1}+\sigma\\ 
& = &\displaystyle -\delta<0.\\
\end{array}$$
Hence, $F$ points inward  $\mathcal{B}_{\mu}$ along $\mathcal{B}^{+}_{\mu}.$ Now, let $n^-=\left(1,0\right)$  be a normal vector to $\mathcal{B}^{-}_{\mu}.$ Since, $F\big|_{\mathcal{B}^{-}_{\mu}}=(1,\mu^n-v^n)$ we get $\langle F,n^-\rangle=1>0$ and, then, $F$ also points inward  $\mathcal{B}_{\mu}$ along $\mathcal{B}^{-}_{\mu}.$  Let $n^*=(0,1)$ be a normal vector to $\mathcal{B}^{*}_{\mu}.$ Since $F\big|_{\mathcal{B}^{*}_{\mu}}=(1,-u^{2k-1}-\mu^n+\sigma),$ we get $\langle F,n^*\rangle=-u^{2k-1}-\mu^n+\sigma<0$ and, then, $F$ points inward  $\mathcal{B}_{\mu}$ along $\mathcal{B}^{*}_{\mu}.$ Finally, let $n^{\#}=(0,1)$ be a normal vector to $\mathcal{B}^{\#}_{\mu}.$ Since  $F\big|_{\mathcal{B}^{\#}_{\mu}}=(1,-u^{2k-1}+\sigma),$ we get $\langle F,n^{\#}\rangle=-u^{2k-1}+\sigma>0$ and, then, $F$ points inward  $\mathcal{B}_{\mu}$ along $\mathcal{B}^{\#}_{\mu}$

It remains to study the behavior of the trajectory of $F$ passing through the point $p_1=\big((-\mu^n+\sigma)^{\frac{1}{2k-1}},\mu\big).$  
Consider the function $h_1(u,v)=v-\mu,$ then
\[
\begin{array}{rl}
Fh_{1}(p_{1}) = & \langle \nabla h_{1}(p_1),F(p_{1})\rangle = 0,\vspace{0.2cm}\\

F^2h_{1}(p_{1}) =&\langle\nabla Fh_{1}(p_{1}),F(p_{1})\rangle= -(2k-1)\left(-\mu^n+\sigma\right)^{\frac{2(k-1)}{2k-1}}<  0.
\end{array}
\]
Consequently, $F$ has a quadratic contact with the straight line $v=\mu$ at $p_1$ and the trajectory passing through $p_1$ stays, locally, below this line. Given that $\dot u=1,$ we conclude that the flow enters the region $\mathcal{B}_{\mu}$ through $p_{1}$ (see Figure \ref{fig}). 

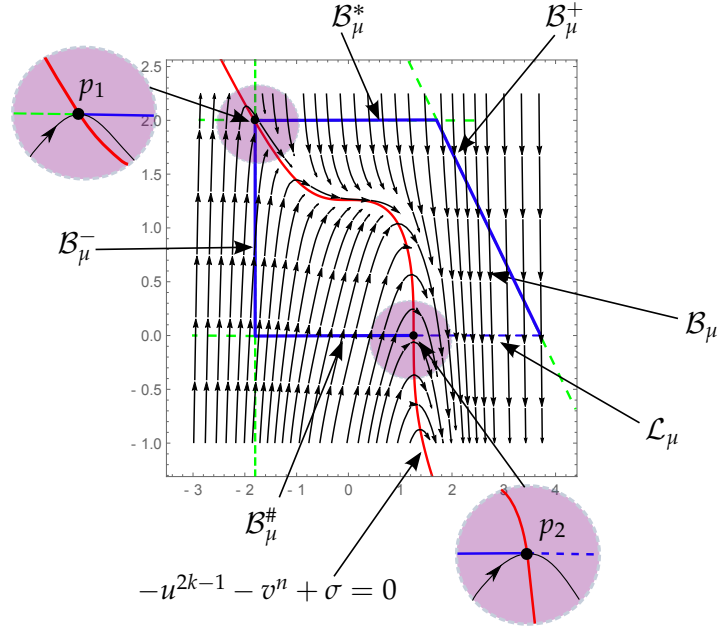
\begin{figure}[h]
	\begin{center}
		\begin{overpic}[scale=0.5]{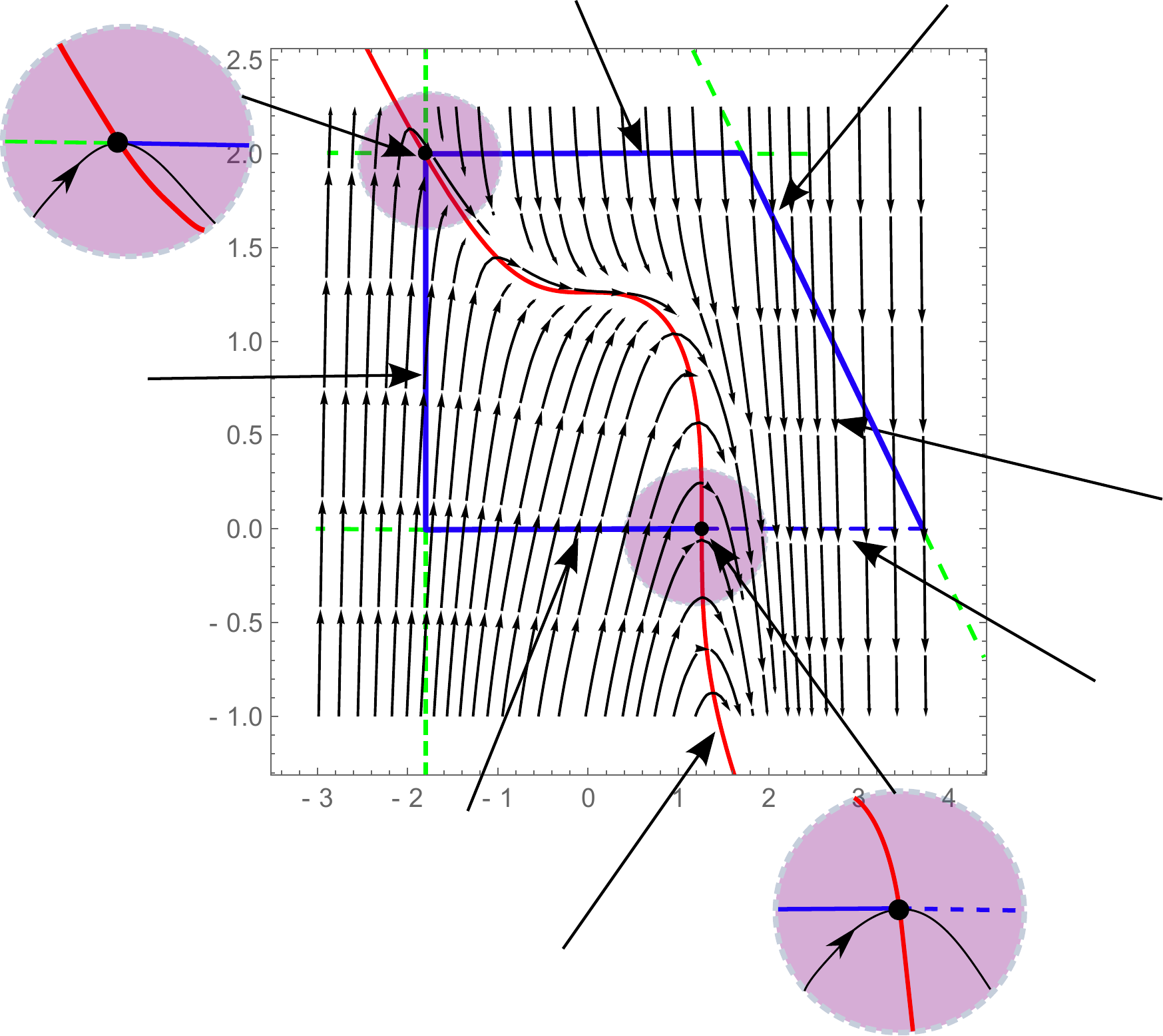}
		\put(48,90){$\mathcal{B}_{\mu}^*$}
		\put(101,44){$\mathcal{B}_{\mu}$} 
		\put(7,56){$\mathcal{B}_{\mu}^-$}
		\put(35,14){$\mathcal{B}_{\mu}^{\#}$}
		\put(19,4){$-u^{2k-1}-v^n+\sigma=0$}
		\put(80,90){$\mathcal{B}_{\mu}^+$}
		\put(95,28){$\mathcal{L}_\mu$} 
		
		\put(10,80){$p_{1}$}
		\put(79,14){$p_{2}$}
		\end{overpic}
	\end{center}
	\caption{The vector field $F$ and the region $\mathcal{B}_\mu.$ The red curve represents the isocline $-u^{2k-1}-v^n+\sigma=0.$}
	\label{fig}
	\end{figure}

Now, given $p=(u_0,v_0)\in\R^2$ with $v_0>0$ there exists $\mu_{0}$ such that $p\in\mathcal{B}_{\mu_{0}}.$ From the comments above, we known that the trajectory of $F$ passing through $p$ cannot leave the region $\mathcal{B}_{\mu}$ through $\p \mathcal{B}_{\mu}\setminus \mathcal{L}_{\mu}.$ Thus, assume by contradiction that the semi-orbit $\gamma_{p}^+=\{(u(t),v(t))|\,t\geqslant 0\}$ is contained in the compact region $\mathcal{B}_{\mu_0}.$ From the \textit{Poincar\'{e}--Bendixson Theorem} $\omega(p)\subset \mathcal{B}_{\mu_0}$ either contains a singularity of $F$ or is a periodic orbit of $F.$ In the last case,  $\textrm{int}(\omega(p))$ contains  a singularity of $F.$ Both cases contradicts the fact that $F$ does not admit singularities. 
Hence, $\gamma_{p}^+$ must leave the region $\mathcal{B}_{\mu_0}$ through $\mathcal{L}_{\mu_0}.$ In other words, there exists $t_0>0$ such that $(u(t_0),v(t_0))=(u^*,0)$ with $u^*\geqslant\sigma^{\frac{1}{2k-1}}.$

We conclude this proof by showing that $u^*>\sigma^{\frac{1}{2k-1}}.$ Indeed, let $p_{2}=\big(\sigma^{\frac{1}{2k-1}},0\big)$ and define the function $h_{2}(u,v)=v.$ Then
\[
\begin{array}{rl}
Fh_{2}(p_{2})  = & \langle\nabla h_{2}(p_{2}),F(p_{2})\rangle=0,\vspace{0.2cm}\\

F^2 h_{2}(p_{2}) =&  \langle\nabla Fh_{2}(p_{2}),F(p_{2})\rangle\vspace{0.1cm}\\
= & \left\lbrace\begin{array}{lllll}\label{lambdaeta}
-(2k-1)  & if & k=1,\\
-(2k-1)\sigma^{\frac{2(k-1)}{2k-1}} & if & k>1.\\      
\end{array}\right.
\end{array}
\]
If $k=1$ or $\sigma\neq 0,$ then $F^2 h_{2}(p_{2})<0.$ Consequently, $F$ has a quadratic contact with the straight line $v=0$ at $p_2$ and the trajectory passing through $p_2$ stays, locally, below this line (see Figure \ref{fig}). If $k>1$ and $\sigma=0,$ then $F^2 h_{2}(p_{2})=0.$ In addition, one can see that $F^{j} h_{2}(p_{2})=0$ for $j\in\{1,\ldots,2k-1\}$ and $F^{2k} h_{2}(p_{2})=-(2k-1)!<0.$ Thus, $F$ has an even multiplicity contact with the straight line $v=0$ at $p_2$ and the trajectory passing through $p_2$ also stays, locally, below this line (see Figure \ref{fig}). Therefore, $p_2\notin \gamma_p^+$ and, consequently, $u^*>\sigma^{\frac{1}{2k-1}}.$\end{proof}

\section{Extension of the Fenichel Manifold} \label{sec:fenichelmanifold}

Consider a Filippov system $Z=(X^+,X^-)_{\Sigma}$ and assume that $X^+$ satisfies hypothesis {\bf (A)} for some $k\geqslant 1.$ For $n\geqslant\max\{2, 2k-1\},$ let $\Phi\in C^{n-1}_{ST}$ be given as $\eqref{Phi}.$  From the comments of the previous section, we can assume that $Z,$ restricted to a neighborhood $U\subset\R^2$ of $(0,0),$ is given as \eqref{Xnf}. Thus, the regularized system $Z_{\e}^{\Phi},$ defined in \eqref{regula}, 
leads to the following differential system
\begin{equation}\label{regsys}
Z^\Phi_{\e}: \left\lbrace\begin{array}{l}
\dot{x} = \dfrac{1}{2}\left(1+\Phi_{\e}(y)\right),\vspace{0.2cm}\\
\dot{y} =   \dfrac{1}{2}
	\left( \alpha x^{2k-1}+g(x)+y\vartheta(x,y)\right)
	\left(1+\Phi_{\e}(y)\right)+\dfrac{1}{2}\left(1-\Phi_{\e}(y)\right),
\end{array}\right.
\end{equation}
for $(x,y)\in U$ and $\e>0$ sufficiently small. Recall that $\Phi_{\e}(y)=\Phi(y/\e).$ 

Now, we shall study the regularized system \eqref{regsys} restricted to the band of regularization $|y|\leqslant \e.$ Notice that $\Phi_{\e}(y)=\phi(y/\e)$ for $|y|\leqslant \e.$ In this case,  system \eqref{regsys} can be written as a {\it slow-fast problem}. Indeed, taking $y=\e \widehat{y},$ we get the so-called {\it slow system},
\begin{equation}\left\lbrace\begin{array}{rl}\label{slowsystem}
\dot{x}  = & \dfrac{1}{2}\left(1+\phi(\widehat{y})\right),\vspace{0.2cm}\\
\e\dot{\widehat{y}} =  &   \dfrac{1}{2}\left(
	\left( \alpha x^{2k-1}+g(x)+\e\widehat{y}\vartheta(x,\e\widehat{y})\right)
	\left(1+\phi(\widehat{y})\right)+\left(1-\phi(\widehat{y})\right)\right),
\end{array}\right.\end{equation}
defined for $|\widehat y|\leqslant 1.$ Performing the time rescaling $t=\e\tau,$ we obtain the so-called {\it fast system},
\begin{equation}\label{fastsystem}
\ov{Z}^\Phi_{\e}: \left\lbrace\begin{array}{l}
x'  = \dfrac{\e}{2}\left(1+\phi(\widehat{y})\right),\vspace{0.2cm}\\
\widehat{y}' =   \dfrac{1}{2}\left(
	\left( \alpha x^{2k-1}+g(x)+\e\widehat{y}\vartheta(x,\e\widehat{y})\right)
	\left(1+\phi(\widehat{y})\right)+\left(1-\phi(\widehat{y})\right)\right).
\end{array}\right.\end{equation}
Clearly, systems \eqref{slowsystem} and \eqref{fastsystem} are equivalent for $\e\neq0.$ Taking $\e=0$ in the fast system, we get the {\it layer problem}
\begin{equation}\label{layerproblem}
\ov{Z}^\Phi_{0}: \left\lbrace\begin{array}{l}
x'  = 0,\vspace{0.2cm}\\
\widehat{y}' =   \dfrac{1}{2}\left(
	\left(\alpha x^{2k-1}+g(x)\right)
	\left(1+\phi(\widehat{y})\right)+\left(1-\phi(\widehat{y})\right)\right),
\end{array}\right.
\end{equation}
which has the following critical manifold
\begin{equation}\label{Sa}S_{a}=\left\lbrace(x,\widehat{y})\big|\widehat{y}=m_0(x)\defeq \phi^{-1}\left(\dfrac{1+\alpha x^{2k-1}+g(x)}{1-\alpha x^{2k-1}-g(x)}\right), -L\leqslant x\leqslant 0\right\rbrace,
\end{equation}
where $L$ is a positive parameter satisfying $\alpha x^{2k-1}+g(x)<0$ for $-L\leqslant x<0.$ Notice that, in this case,
\[
-1<\dfrac{1+\alpha x^{2k-1}+g(x)}{1-\alpha x^{2k-1}-g(x)}<1, \,\,\text{for}\,\, -L\leqslant x<0.
\]
Moreover, 
$$\dfrac{\p \pi_2 \ov{Z}^\Phi_{0}}{\p \widehat y}(x,\widehat{y})= \dfrac{\phi'(\widehat{y})}{2}(\alpha x^{2k-1}+g(x)-1),
$$
where $\pi_2 \ov{Z}^\Phi_{0}$ denote the second component of $\ov{Z}^\Phi_{0}.$
Consequently, the critical manifold $S_{a}$ is normally hyperbolic attracting on $S_{a}\setminus\{(0,1)\}$ and loses hyperbolicity at $(0,1).$ Indeed, $\phi'(\widehat{y})(\alpha x^{2k-1}+g(x)-1)<0$ for all $(x,\widehat{y})\in S_{a}\setminus\{(0,1)\}$ and $\phi'(1)=0.$ Thus, the $\textit{Fenichel Theorem}$ \cite{Fenichel79,Jones95} can be applied for any compact subset of $S_{a}\setminus\{(0,1)\}.$
In what follows we state the Fenichel Theorem for system \eqref{fastsystem} as it is stated in \cite{BonetSeara16}.

\begin{theorem}[Fenichel Theorem]\label{thm:fenichel}
Consider $L$ and $N$ positive real numbers, $L>N.$ There exist positive constants $\e_0,$ $K,$ and $C,$ and a smooth function $m(x,\e),$ defined for $(x,\e)\in[-L,-N]\times[0,\e_0]$ and satisfying $m(x,0)=m_0(x)$ (see \eqref{Sa}), such that the following statements hold.
\begin{itemize}
	\item[(i)] $S_{a,\e}=\{(x,\widehat{y})|\widehat{y}=m(x,\e), -L\leqslant x\leqslant -N\}$ is a normally hyperbolic attracting locally invariant manifold of system \eqref{fastsystem}, for $0<\e<\e_0.$
	\item[(ii)] There exists a neighborhood $W$ of $S_{a,\e},$ which does not depend on $\e,$ such that for any $z_0\in W$ there exists $z^*\in S_{a,\e}$ satisfying
	$$|\varphi_{\ov{Z}_\e^\Phi}(t,z_0)-\varphi_{\ov{Z}_\e^\Phi}(t,z^*)|\leqslant Ke^{-\frac{Ct}{\e}}, \,t\geqslant 0,$$ where $\varphi_{\ov{Z}_\e^\Phi}$ is the flow of system \eqref{slowsystem}.
\end{itemize}
\end{theorem}

The invariant manifold $S_{a,\e}$ is called {\it Fenichel Manifold}. Notice that $S_{a,\e}$ corresponds to an orbit of systems \eqref{slowsystem} and \eqref{fastsystem} for $\e\neq0$ small, which can, eventually, be continued. Such a continuation will be called {\it Extension of the Fenichel Manifold}.
In addition, going back with the change $y=\e \widehat{y},$ one can see that $S_{a,\e}$ also corresponds to an invariant manifold of the regularized system \eqref{regsys}, which, as an abuse of notation, we will also call by Fenichel manifold and denote by $S_{a,\e}$.

In the sequel, in order to extend $S_{a,\e}$ until $\widehat{y}=1,$ we shall study system \eqref{fastsystem} around the degenerate point $(0,1).$ Notice that $1+\phi(\widehat{y})>0$  for $\widehat y$ sufficiently close to $1.$ Thus, performing a positive time transformation we can divide the right-hand side of the differential system \eqref{fastsystem} by  $1+\phi(\widehat{y}),$ obtaining the following equivalent system
\begin{equation}\label{eq3}\left\lbrace\begin{array}{l}
x'  =  \e,\\
\widehat{y}'  =   
	 \displaystyle\alpha x^{2k-1}+g(x)+\e\widehat{y}\vartheta(x,\e\widehat{y})
	+\dfrac{1-\phi(\widehat{y})}{1+\phi(\widehat{y})}.
\end{array}\right.\end{equation}
As an abuse of notation, we are still using the prime symbol $'$ to denote differentiation with respect to the new time variable.
Denote $p(\widehat{y})=(1-\phi(\widehat{y}))/(1+\phi(\widehat{y})).$ Computing the expansion of the function $p$ around $\widehat y=1$ we get
$$p(\widehat{y})=\dfrac{1}{2}\phi^{[n]}(-(\widehat{y}-1))^n(1+(\widehat{y}-1)\Upsilon(\widehat{y}-1))$$ where $$\quad \phi^{[n]}=\dfrac{(-1)^{n+1}}{n!}\phi^{(n)}(1)>0$$ and $\Upsilon$ is a smooth function defined in a neighborhood of $0.$ 
Taking $\widehat y=\widetilde y+1,$ system \eqref{eq3} becomes
\begin{equation*}\left\lbrace\begin{array}{l}\label{eq4}
x'  = \e,\\
\widetilde{y}' =    
	 \alpha x^{2k-1}+g(x)+\e(1+\widetilde{y})\vartheta(x,\e(1+\widetilde{y}))
	+\dfrac{1}{2}\phi^{[n]}(-\widetilde{y})^n(1+\widetilde{y}\Upsilon(\widetilde{y})).
\end{array}\right.\end{equation*}
Now, we consider the extended system 
\begin{equation}\label{eq5}E: \left\lbrace\begin{array}{l}
x'  = \widetilde{\e},\\
\widetilde{y}'  =  
	 \alpha x^{2k-1}+x^{2k-1}\widetilde{g}(x)+\widetilde{\e}(1+\widetilde{y})\vartheta(x,\widetilde{\e}(1+\widetilde{y}))
	+\dfrac{1}{2}\phi^{[n]}(-\widetilde{y})^n(1+\widetilde{y}\Upsilon(\widetilde{y})),\\
	\widetilde{\e}'  = 0,\\
\end{array}\right.\end{equation}
where $g(x)=x^{2k-1}\widetilde{g}(x)$ with $\widetilde{g}=\mathcal{O}(x).$ Notice that, the above differential system keeps the planes $\widetilde{\e}=\text{``constant''}$ invariant. In addition, its restriction to $\widetilde{\e}=0$ corresponds to the layer problem \eqref{layerproblem}  for $\widetilde{y}\leqslant 0$. Thus,
once we have understood the orbits of \eqref{eq5} in an neighborhood of the origin $(x,\widetilde y,\widetilde{\e})=(0,0,0),$ we can understand how the Fenichel manifold $S_{a,\e}$ of  \eqref{fastsystem} behaves in a neighborhood of $(x,\widehat y)=(0,1).$ 

Notice that $\{(x,\widetilde{y},0)|\al x^{2k-1}+x^{2k-1}\widetilde{g}(x)+\dfrac{1}{2}\phi^{[n]}(-\widetilde{y})^n(1+\widetilde{y}\Upsilon(\widetilde{y}))=0\}$ is a set of degenerate singularities of \eqref{eq5}. Thus, in order to study the differential system \eqref{eq5} in a neighborhood of the origin, we shall apply the following blow-up
$$\begin{array}{rcl}
\Psi:\mathbb{S}^2\times \R_+&\rightarrow& \mathbb{R}^3_*\vspace{0.2cm}\\
(\ov{x},\ov{y},\ov{\e},r)&\mapsto&(r^n\ov{x}, r^{2k-1}\ov{y}, r^{1+2k(n-1)}\ov{\e}).\\
\end{array}$$
Here, $$\mathbb{S}^2=\{(\ov{x},\ov{y},\ov{\e})\in\mathbb{R}^3|\ov{x}^2+\ov{y}^2+\ov{\e}^2=1\}\,\,\text{and}\,\, \mathbb{R}^3_*=\mathbb{R}^3\setminus\{(0,0,0)\}.$$
Roughly speaking, the geometric idea of the blow-up method is to ``change'' the nonhyperbolic singularity $(0,0,0)$ by a sphere $\mathbb{S}^2,$ leaving the dynamics away from the origin unchanged. This allow us to blow-up the dynamics around the origin. 
Formally, the map $\Psi$ pulls back the vector field $E\big|_{\mathbb{R}^3_*},$ defined in \eqref{eq5}, to a vector field $\Psi^{*}E$ defined on $\mathbb{S}^2\times\R_+.$ Here, $\Psi^{*}$ denotes the usual {\it pullback},
\[
\Psi^*E(p)=\left(D\Psi(p)\right)^{-1} E(\Psi(p)),\,\, p=(\ov{x},\ov{y},\ov{\e},r).
\]
In order to study the behavior of $\Psi^{*}E$ in a neighbourhood of $\mathbb{S}^2_0=\mathbb{S}^2\times\{0\},$ we have to extend its dynamics to $\mathbb{S}^2_0$ and desingularize it through a time transformation. This provides a new vector field $E^*$ which has its dynamics outside $\mathbb{S}^2_0$ equivalent to $E|_{\mathbb{R}^3_*}.$ Then, we consider two charts of $\mathbb{S}^2\times \R_{\geqslant0},$ namely,  $\kappa_1=(\mathcal{U}_{1}, \psi_1)$ and $\kappa_2=(\mathcal{U}_{2}, \psi_2),$ where
\[ \mathcal{U}_{1}=\{(\ov{x},\ov{y},\ov{\e},r)\in\mathbb{S}^2\times \R_{\geqslant0}|\ov{y}< 0\},\,\,\mathcal{U}_{2}=\{(\ov{x},\ov{y},\ov{\e},r)\in\mathbb{S}^2\times \R_{\geqslant0}|\ov{\e}>0\},\]
and $\psi^{1,2}:\mathcal{U}_{1,2}\rightarrow\R^3$ are the following stereographic-like projections
\[
\displaystyle \psi^1(\ov{x},\ov{y},\ov{\e},r)=\left((-\ov{y})^{\al_1}\,\ov{x},\,(-\ov{y})^{\beta_1}\,r,(-\ov{y})^{\gamma_1}\,\ov{\e}\right),\,\,
\displaystyle \psi^2(\ov{x},\ov{y},\ov{\e},r)=\left(\ov{\e}^{\al_2}\,\ov{x},\,\ov{\e}^{\beta_2}\,\ov{y},\ov{\e}^{\gamma_2}\,r\right),
\]
with
\[
\al_1=\dfrac{-n}{2k-1},\quad \beta_1=\dfrac{1}{2k-1},\quad\gamma_1=\dfrac{-(1+2k(n-1))}{2k-1},
\]
and
\[
\al_2=\dfrac{-n}{1+2k(n-1)},\quad \beta_2=\dfrac{-(2k-1)}{1+2k(n-1)},\quad\gamma_2=\dfrac{1}{1+2k(n-1)}.
\]
The maps $\psi^1$ and $\psi^2$ are constructed by projecting the sets $\mathcal{U}_{1}$ and $\mathcal{U}_{2}$ into the planes $\ov{y}=-1$ and $\ov{\e}=1,$ respectively (see figure \ref{fig9}).

The above charts are used to push forward the vector fields $E^*_i=E^*|_{\mathcal{U}_{i}},$ $i=1,2,$ to vector fields defined on $\R^3,$ $F_i=\psi^i_* E^*_i,$ $i=1,2.$ Here, $\psi^i_*$ denotes the usual {\it pushforward},
\[
\psi^i_* E^*_i(q)=D\psi^i\left((\psi^i)^{-1}(q)\right)E^*_i\left((\psi^i)^{-1}(q)\right),\,\, q\in \psi_i(\mathcal{U}_i).
\]

 \begin{figure}[h]
	\begin{center}
		\begin{overpic}[scale=0.4]{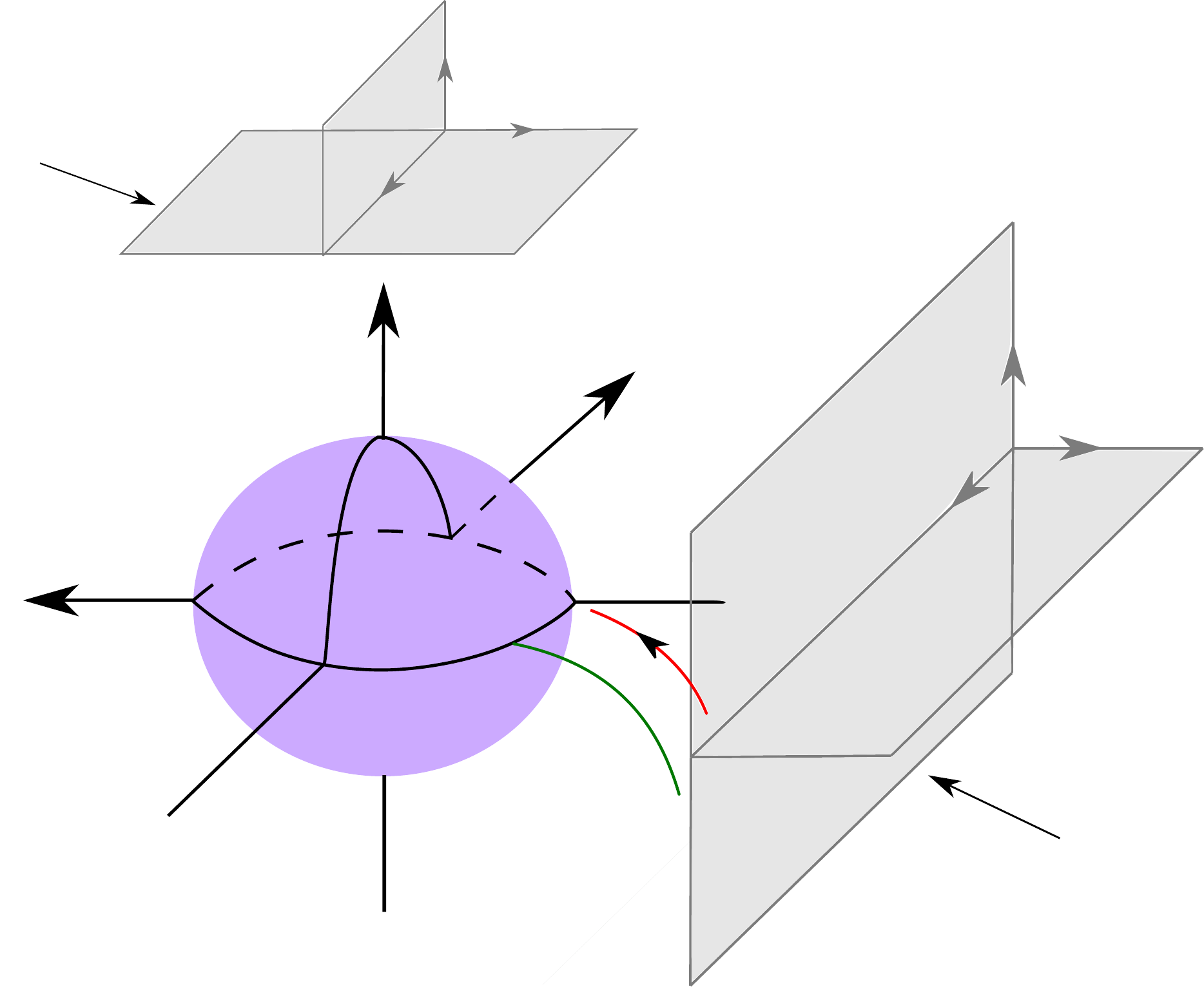}
		\put(53,52){$\ov{x}$}
		\put(-2,31){$\ov{y}$} 
		\put(50,19){{\scriptsize $S_{a}$ \par}}
				\put(90,11){$\kappa_{1}$}
		\put(49.6,26.3){{\scriptsize $S_{a,\e}$ \par}}
		\put(29,57){$\ov{\e}$}
		\put(-3,69){$\kappa_{2}$}
		\put(76,42){{\tiny $x_1$ \par}}
		\put(91,46){{\tiny $r_1$ \par}}
		\put(85,53){{\tiny $\e_1$ \par}}
		\put(38,77){{\tiny $r_2$ \par}}
		\put(29,67){{\tiny $x_2$ \par}}
		\put(44,73){{\tiny $y_2$ \par}}
			\end{overpic}
	\end{center}
	\caption{The 2 charts of blow-up, critical manifold $S_a,$ and Fenichel manifold $S_{a,\e}.$}
	\label{fig9}
\end{figure}
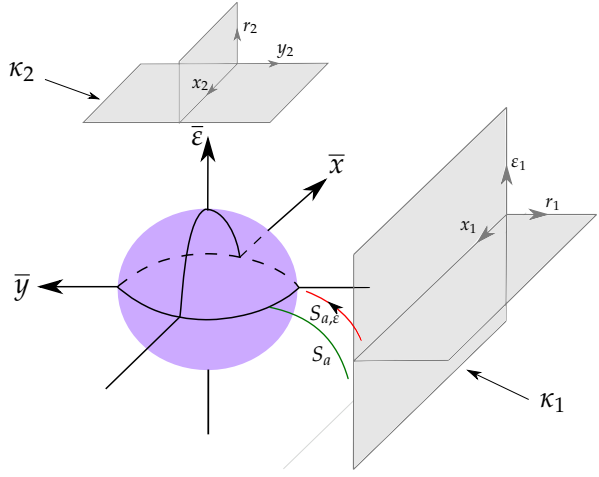

Finally, consider the composition $\Psi_i=\Psi\circ(\psi^i)^{-1},$ $i=1,2.$ Then,
\[
\begin{array}{l}
\displaystyle \Psi_1(x_1,r_1,\e_1)=\left(r_{1}^n x_ {1},\,\, -r_{1}^{2k-1},\,\,r_{1}^{1+2k(n-1)}\e_{1}\right),\vspace{0.2cm}\\
\displaystyle \Psi_2(x_2,y_2,r_2)=\left(r_{2}^n x_ {2},\,\,r_{2}^{2k-1}y_{2},\,\,r_{2}^{1+2k(n-1)}\right).
\end{array}
\]
The vector field $F_i,$ $i=1,2,$ can be directly obtained as $F_i= \Psi_i^* E/r^{(n-1)(2k-1)}.$ 
Notice that we are pulling back the vector field $E$ through $\Psi_i,$ extending $\Psi_i^* E$ to $r_i=0$ and, then, desingularizing it by doing a time transformation (i.e. dividing by $r^{(n-1)(2k-1)}$).

Moreover, $$\mathcal{U}_{12}\defeq \mathcal{U}_1\cap \mathcal{U}_2=\{(\ov{x},\ov{y},\ov{\e},r)\in\mathbb{S}^2\times \R_{\geqslant0}|\ov{y}< 0\text{ and }\ov{\e}> 0\}$$ and the change of coordinates $\psi_{12}:\psi_1(\mathcal{U}_{12})\rightarrow \psi_2(\mathcal{U}_{12})$ which pushes forward $F_1\big|_{\psi_1(\mathcal{U}_{12})}$ to $F_2\big|_{\psi_2(\mathcal{U}_{12})}$ writes
\[
\psi_{12}(x_1,r_1,\e_1)=\left(\e_{1}^{-\frac{n}{1+2k(n-1)}}x_{1}\,,\, -\e_{1}^{-\frac{2k-1}{1+2k(n-1)}}\,,\, r_{1}\e_{1}^{\frac{1}{1+2k(n-1)}}\right).
\]

\subsection{$\textbf{Chart }\kappa_1$}
The differential system associated with the vector field $F_1$ writes
\begin{equation}\label{skappa1}\left\lbrace\begin{array}{lllll}
x_{1}'  = & \e_1+\dfrac{nx_1}{2k-1}[H(x_1,r_1)-J(x_1,r_1,\e_1)],\\
r_{1}'  = & -\dfrac{r_1}{2k-1}[H(x_1,r_1)-J(x_1,r_1,\e_1)],\\
	\e_{1}'  = & \dfrac{(1+2k(n-1))\e_1}{2k-1}[H(x_1,r_1)-J(x_1,r_1,\e_1)],\\
\end{array}\right.\end{equation}
where
$$\begin{array}{rcl}
H(x_{1},r_{1})&=&x_{1}^{2k-1}(\alpha+\widetilde{g}(r_{1}^nx_{1}))+\dfrac{\phi^{[n]}}{2}(1-r_{1}^{2k-1}\Upsilon(-r_{1}^{2k-1})),\\
J(x_1,r_1,\e_1)&=&(r_{1}^{n}-r_{1}^{1-2k+n})\e_{1}\vartheta(r_{1}^{n}x_{1},-r_{1}^{2k(n-1)}(-r_{1}+r_{1}^{2k})\e_{1}).\\
\end{array}$$

First, taking $\e_1=0$ in \eqref{skappa1}, we get that the critical manifold $S_{a},$ in this coordinate system, is given by 
$$S_{a,1}=\Big\{(x_{1},r_{1},0)|H(x_1,r_1)=0,\, -L\leqslant r_{1}^nx_{1}\leqslant 0\Big\}.$$
In what follows, we shall write the critical manifold $S_{a,1}$ locally as a graphic. For this, define $U_1=\{(x_{1},r_{1})|-L\leqslant r_{1}^n x_{1}\leqslant 0\}$ and consider the function $H$ restricted to $U_1.$ Notice that 
$H(x_{1},0)=\alpha x_{1}^{2k-1}+\dfrac{\phi^{[n]}}{2}.$ Thus, for $x_{1}^*=\left(-\dfrac{\phi^{[n]}}{2\alpha}\right)^{\frac{1}{2k-1}}$ we get that $H(x_{1}^*,0)=0.$ Furthermore, 
$$\dfrac{\partial H}{\partial x_{1}}(x_{1}^*,0)=(2k-1)\alpha (x_{1}^*)^{2k-2}=(2k-1)\alpha\left(-\dfrac{\phi^{[n]}}{2\alpha}\right)^{\frac{2k-2}{2k-1}}\neq 0.$$ 
From the \textit{Implicit Function Theorem}, there exist open sets $W_1, V_1\subseteq\mathbb{R}$ such that $(x_{1}^*,0)\in W_1\times V_1\subseteq U_1,$ and a unique smooth function $x_{1}:V_1\rightarrow W_1$ such that  $x_{1}(0)=x_{1}^*$ and $H(x_{1}(r_{1}),r_{1})=0,$ for all $r_{1}\in V_1.$ Moreover,  
$$x_{1}'(0)=\left\{\begin{array}{lllll}
0 & if & k>1,\\
	\dfrac{\phi^{[n]}\Upsilon(0)}{2\alpha}   & if &  k=1.
\end{array}\right.$$
Thus, expanding $x_{1}(r_1)$ around $r_{1}=0,$ we have
$$\begin{array}{lllll}
x_{1}(r_{1})  & = &\displaystyle x_{1}(0)+r_{1}x_{1}'(0)+\mathcal{O}(r_{1}^2)\vspace{0.2cm}\\
	   & = &\displaystyle \left\{\begin{array}{lllll}
x_{1}^*+\mathcal{O}(r_{1}^2) & \text{if} & k>1,\vspace{0.1cm}\\
	x_{1}^*+r_1\dfrac{\phi^{[n]}\Upsilon(0)}{2\alpha}+\mathcal{O}(r_{1}^2)   & \text{if} &  k=1.
\end{array}\right.
\end{array}$$ 
Consequently, 
$$ S_{a,1}\cap (W_1\times V_1\times\{0\})=\{(x_{1},r_{1},0)|\quad x_{1}=x_{1}(r_{1})\}.$$

Notice that $S_{a,1}$ intersects the plane $r_1=0$ (which is equivalent to the sphere $\mathbb{S}_0^2$) at the singularity $(x_1^*,0,0).$ Moreover, 
$$DF_1(x_{1}^*,0,0)=
	\left(\begin{array}{ccc}
		-\dfrac{\phi^{[n]}n}{2} & \omega_{k}^{12} & 1+\omega_{n,k}^{13}\\
		0 & 0 & 0\\
		0 & 0 & 0\\
	\end{array}\right),
	$$ where 
$$\begin{array}{rcl}\omega_{k}^{12}& = &\displaystyle\left\lbrace\begin{array}{lllll}\label{eq77}
0  & if & k>1,\\
\dfrac{(\phi^{[n]})^2n\Upsilon(0)}{4\alpha} & if & k=1,\\      
\end{array}\right.\vspace{0.2cm}\\
\omega_{n,k}^{13}& = &\displaystyle\left\lbrace\begin{array}{lllll}\label{eq7}
0  & if & n>2k-1,\vspace{0.1cm}\\
x^{*}_{1}\vartheta(0,0) & if & n=2k-1 \hspace{0.1cm}\text{and}\hspace{0.1cm} k\neq 1.\\      
\end{array}\right.\\
\end{array}$$
Hence, in the sequel, we shall use the \textit{Center Manifold Theorem} \cite{Carr} to study $F_1$ around the degenerated singularity $(x_1^*,0,0).$ 

One can easily see that $\la_{1}=-\phi^{[n]}n/2,$ $\la_{2}=0,$ and $\la_{3}=0$ are the eigenvalues of $DF(x_{1}^*,0,0)$ associated with the eigenvectors 
\[
v_{1}=(1,0,0),\,\, v_{2}=\left(\frac{2\omega_k^{12}}{\phi^{[n]}n},1,0\right),\,\,\text{and}\,\, v_{3}=\left(\frac{2}{\phi^{[n]}n}(1+\omega_{n,k}^{13}),0,1\right),\]
respectively. 
Thus, consider a box  $\Omega=[\chi,0]\times[0,\rho]\times[0,\nu]$ around $(x_1^*,0,0),$ 
where $\chi<x_1^*$ and $\rho,\nu>0$ are small parameters. 
By the \textit{Center Manifold Theorem} we know that within $\Omega$ there exists a center manifold $W^c=\{(x_{1},r_{1},\e_{1})|\quad x_{1}=k(r_{1},\e_{1})\}$ tangent to the eigenspace generated by $v_{2}$ and $v_{3}$ at the singularity $(x_{1}^*,0,0).$ Moreover, since $(x_1^*,0,0)\in W^c\cap S_{a,1}\neq\emptyset,$ we conclude that $W^c$ contains the critical manifold $S_{a,1}.$ 
Assume that $W^c=\widehat{h}^{-1}(0),$ with $\widehat{h}(x_{1},r_{1},\e_{1})=x_{1}-k(r_{1},\e_{1})$ and $k(0,0)=x^*_1.$ Since $\langle\nabla\widehat{h}(0),v_{2}\rangle=0$ and $\langle\nabla\widehat{h}(0),v_{3}\rangle=0$ we get 
\[
\dfrac{\partial k}{\partial r_1}(0,0)=\dfrac{2\omega_k^{12}}{\phi^{[n]}n}\quad \text{and} \quad \dfrac{\partial k}{\partial \e_1}(0,0)=\dfrac{2}{\phi^{[n]}n}(1+\omega_{n,k}^{13}),\]
respectively. Therefore, 
	$$k(r_{1},\e_{1})=x_{1}^*+r_1\dfrac{2\omega_k^{12}}{\phi^{[n]}n}+\e_{1}\dfrac{2}{\phi^{[n]}n}(1+\omega_{n,k}^{13})+\mathcal{O}_{2}(r_{1},\e_{1}). $$
	
Now, we shall see that the center manifold $W^c$ is foliated by hyperbolas. Indeed, from \eqref{skappa1} we have that $$\dfrac{dr_{1}}{d\e_{1}}=-\dfrac{r_{1}}{\e_{1}(1+2k(n-1))}.$$ Thus, solving the above differential equation, we get that $\e_1\mapsto\e_{1}r_{1}(\e_1)^{1+2k(n-1)}$ is constant on $\e_1.$ This means that, for each $\e>0,$ the surface $$E_{\e}=\{(x_1,r_{1},\e_{1})|\quad \e_{1}r_{1}^{1+2k(n-1)}=\e \}$$ is invariant through the flow of \eqref{skappa1}.  Note that invariance of the sets $E_\e$ can be seen directly from the definition of the blow-up map $\Psi$. Consequently, the manifold $W^c$ is foliated by invariant hyperbolas $\gamma_{\e}=W^c\cap E_{\e},$ ${\e}>0,$ which correspond to orbits of \eqref{skappa1}. Thus, we can write $\gamma_{\e}=\{\varphi_{F_1}(t,\e):t\in I_{\e}\}$ where $\varphi_{F_1}(t,\e)$ is a trajectory of \eqref{skappa1} satisfying 
\begin{equation*}\label{gammaphi}
\varphi_{F_1}(0,\e)=(k(\rho,\e\,\rho^{-(1+2k(n-1))}),\rho,\e\,\rho^{-(1+2k(n-1))})\in W^c\cap \gamma_{\e},
\end{equation*}
and $I_{\e}$ is a neighborhood of the origin.  Hence,  $\Psi_1\gamma_{\e}$ is an orbit of $E$ \eqref{eq5} lying in the plane $\widetilde\e=\e.$ Therefore, (after the translation $\widehat y=1+\widetilde y$) we get it as an orbit \eqref{fastsystem}.

Denote by $S_{a,\e}^1$ the Fenichel manifold $S_{a,\e}$  orbit of \eqref{fastsystem} for $\widetilde\e=\e$ written in the coordinates $(x_1,r_1,\e_1).$
We claim that, for $\e>0$ sufficiently small, the Fenichel manifold $S_{a,\e}^1$ can be continued as an orbit of $F_1$ in $W^c,$ namely $\gamma_{\e}.$ First, noticed that the orbit $\gamma_{\e}$ is $\e-$close to $S_{a,1}$ at $r_{1}=\rho.$ Indeed, from the relation $\e_{1}r_{1}^{1+2k(n-1)}=\e$ satisfied by $\gamma_{\e},$ we see that $\varphi_{F_1}(0,\e)$ approaches to $S_{a,1}=W^c\cap\{\e_1=0\}$ when $\e$ goes to zero. Now, since $S_{a,\e}^1$ is also $\e-$close to $S_{a,1},$ we get that $S_{a,\e}^1$ and $\gamma_{\e}$ are  $\e-$ close to each other at $r_{1}=\rho.$ 
Noticing that $\gamma_{\e}$ and $S_{a,\e}^1$ are related to orbits of \eqref{fastsystem}, which are $\e$-close to each other, we get from item (ii) of Fenichel Theorem \ref{thm:fenichel} that  $d(\varphi_{F_1}(t,\e),S_{a,\e}^1)\leqslant K e^{-\frac{C t}{\e}}.$ Hence, taking any positive time $t_0\in I_{\e}$ we conclude that 
$S_{a,\e}^1$ and $\gamma_{\e}$ are $\mathcal{O}(e^{-\frac{c}{\e}})$ close to each other at $r_1=\rho'<\rho,$ with $c= C t_0>0.$ Therefore, for each $\e>0,$ $\gamma_{\e}$ can be seen as a continuation of $S_{a,\e}^1$ on $W^c$ (see Figure \ref{fig0}). 

Now, at $\e_{1}=\nu$ we have 
\[\begin{array}{rl}
\gamma_{\e}\cap\{\e_1=\nu\}=&\bigg(k\Big((\e\,\nu^{-1})^{\frac{1}{1+2k(n-1)}},\nu\Big),(\e\,\nu^{-1})^{\frac{1}{1+2k(n-1)}},\nu\bigg)\vspace{0.2cm}\\
=&\Big(k(0,\nu),0,\nu\Big)+\CO(\e^{\frac{1}{1+2k(n-1)}}).
\end{array}
\] 
Hence, we conclude that $S_{a,\e}^1\cap\{\e_1=\nu\}$ is $\CO(\e^{\frac{1}{1+2k(n-1)}})$ close to $(k(0,\nu),0,\nu).$
\begin{remark}
Notice that $W^c\cap\{r_1=0\}$ is an orbit of \eqref{skappa1} containing the point $(x_1,r_1,\e_1)=(k(0,\nu),0,\nu)$ which the backward trajectory approaches asymptotically to $(x_1^*,0,0)$ (see Figure \ref{fig0}). Indeed,  
$$W^c\cap\{r_1=0\}=\left\{(x_{1},0,\e_{1})\big|\quad x_{1}=x_{1}^*+\e_{1}\dfrac{2}{\phi^{[n]}n}(1+\omega_{n,k}^{13})+\mathcal{O}(\e_{1}^2)\right\}$$ and, therefore,
$W^c\cap\{r_1=0\}\cap\{\e_1=0\}=\{(x_1^*,0,0)\}.$
\end{remark}

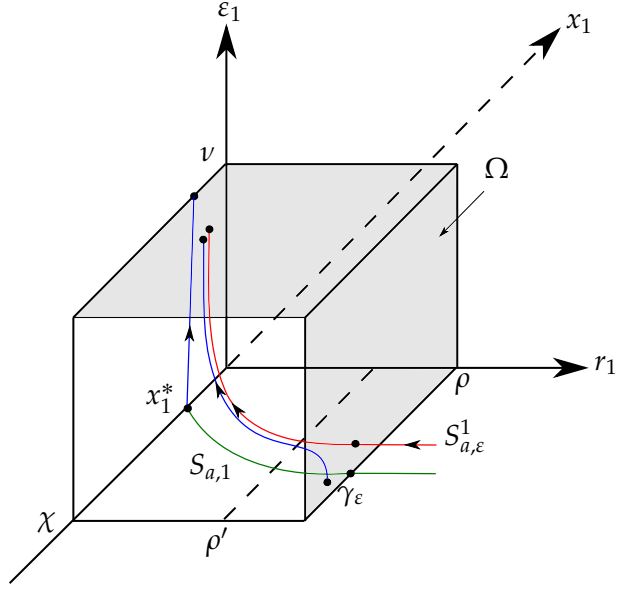
\begin{figure}[h]
	\begin{center}
		\begin{overpic}[scale=0.27]{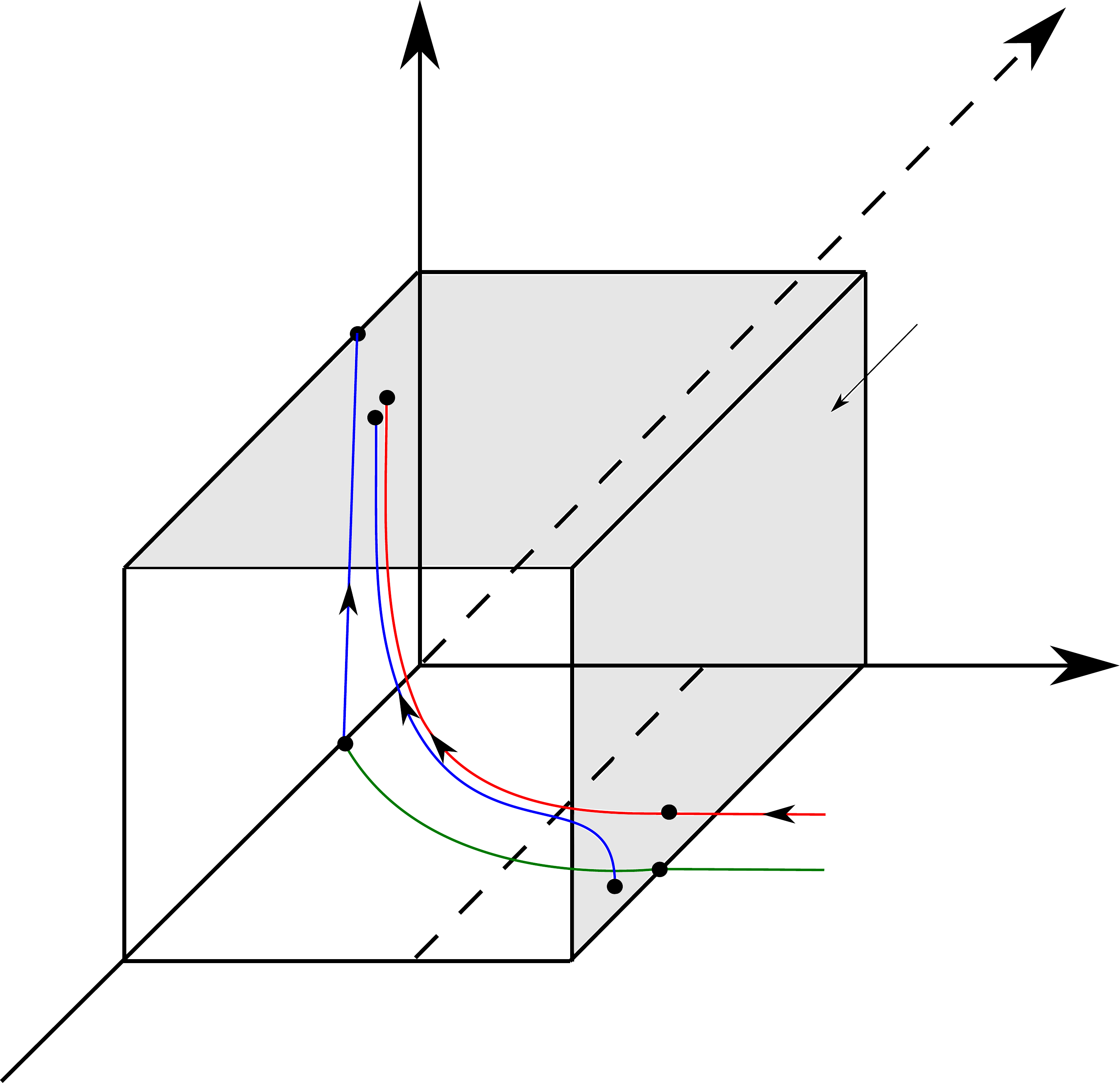}
		\put(96,96){$x_1$}
		\put(101,37){$r_1$} 
		\put(75,24){$S_{a,\e}^1$}
		\put(31,19){$S_{a,1}$}
		\put(56.5,14.5){$\gamma_{\e}$}
		\put(33,73){$\nu$}
		\put(77,34){$\rho$}
		\put(82,70){$\Omega$}
		\put(23.5,31){$x^*_1$}
		\put(5,10){$\chi$}
		\put(34,6){$\rho'$}
		\put(36,98){$\e_1$}
		\end{overpic}
		\bigskip
	\end{center}
	\caption{Behavior of the vector field \eqref{skappa1} around $(x_1^*,0,0).$}
	\label{fig0}
\end{figure}

 In what follows, we shall continue $S_{a,\e}^1$ in chart $\kappa_{2}$ by following the trajectory of $(x_2^*,y_2^*,0)\defeq \psi_{12}(k(0,\nu),0,\nu)$ (see Figure \ref{fig00}).

\subsection{$\textbf{Chart }\kappa_2$}
The differential system associated with the vector field $F_2$ writes
\begin{equation}\left\lbrace\begin{array}{rl}\label{eq8}
x_{2}'  = & 1,\\
y_{2}' =  & x_{2}^{2k-1}(\alpha+\widetilde{g}(r_{2}^nx_{2}))+\dfrac{\phi^{[n]}}{2}(-y_{2})^n(1+r_{2}^{2k-1}y_{2}\Upsilon(r_{2}^{2k-1}y_{2}))\\    
	      & +  (r_{2}^{1-2k+n}+r_{2}^{n}y_{2})\vartheta(r_{2}^{n}x_{2},r_{2}^{2kn}(r_{2}^{1-2k}+y_{2})),\\
	r_{2}'  = & 0.\\
\end{array}\right.\end{equation} 

\begin{lemma}\label{lemma4}The forward orbit of \eqref{eq8} starting at $(x_2^*,y_2^*,0)$ intersects $\{y_{2}=0\}$ at $$(x_{2},y_{2},r_{2})=(\eta,0,0),$$ where $\eta$ is a constant satisfying
 \begin{equation}\label{nsigmank}
	\eta>\sigma_{n,k}\defeq \left\lbrace\begin{array}{lllll}
0  & if & n>2k-1,\\
-\left(\dfrac{\vartheta(0,0)}{\alpha}\right)^{\frac{1}{2k-1}} & if & n=2k-1 \hspace{0.1cm}\text{and}\hspace{0.1cm} k\neq 1.\\      
\end{array}\right.
\end{equation} 
\end{lemma}

\begin{proof} Set $c_{x}=\left(2/(\phi^{[n]}\alpha^{n-1})\right)^{\frac{1}{1+2k(n-1)}}>0$ and $c_{y}=-\alpha c_{x}^{2k}<0.$ Consider system \eqref{eq8} restricted to $r_{2}=0.$ Applying the change of variables $(x_{2},y_{2})=(c_{x}u,c_{y}v),$ $v\leqslant0,$ and a time rescaling by the positive constant $c_x,$ system \eqref{eq8}$|_{r_2=0}$ writes
\begin{equation}\left\lbrace\begin{array}{lllll}\label{eq9}
u'  &= & 1,\\
v' & = &\displaystyle -u^{2k-1}-v^n+s_{n,k},\\  
\end{array}\right.\end{equation}
where
\begin{equation*}\label{sigmakn}
	s_{n,k}=\left\lbrace\begin{array}{lllll}
0  & if & n>2k-1,\\
-\dfrac{\vartheta(0,0)}{\alpha c_{x}^{n}} & if & n=2k-1 \hspace{0.1cm}\text{and}\hspace{0.1cm} k\neq 1,\\      
\end{array}\right.
\end{equation*} 

Take $(u_{0},v_{0})=(x_2^*/c_x,y_2^*/c_y).$ Since $y_2^*=-\nu^{-\frac{2k-1}{1+2k(n-1)}},$ we have $v_0>0.$
Thus, by Lemma \ref{lemma2}, the forward trajectory of \eqref{eq9} starting at $(u_0,v_0)$ intersects $\{v=0\}$ at $(u^*,0)$ with $u^*>s_{n,k}^{\frac{1}{2k-1}}.$ 
Consequently, the forward flow of $(x_2^*,y_2^*,0)$ intersects $\{y_{2}=0\}$ at the point $(x_{2},y_{2},r_{2})=(\eta,0,0),$ where $\eta\defeq c_{x} u^*$ is a constant satisfying $\eta>\sigma_{n,k}.$  \end{proof}

 \begin{figure}[h]
	\begin{center}
		\begin{overpic}[scale=0.4]{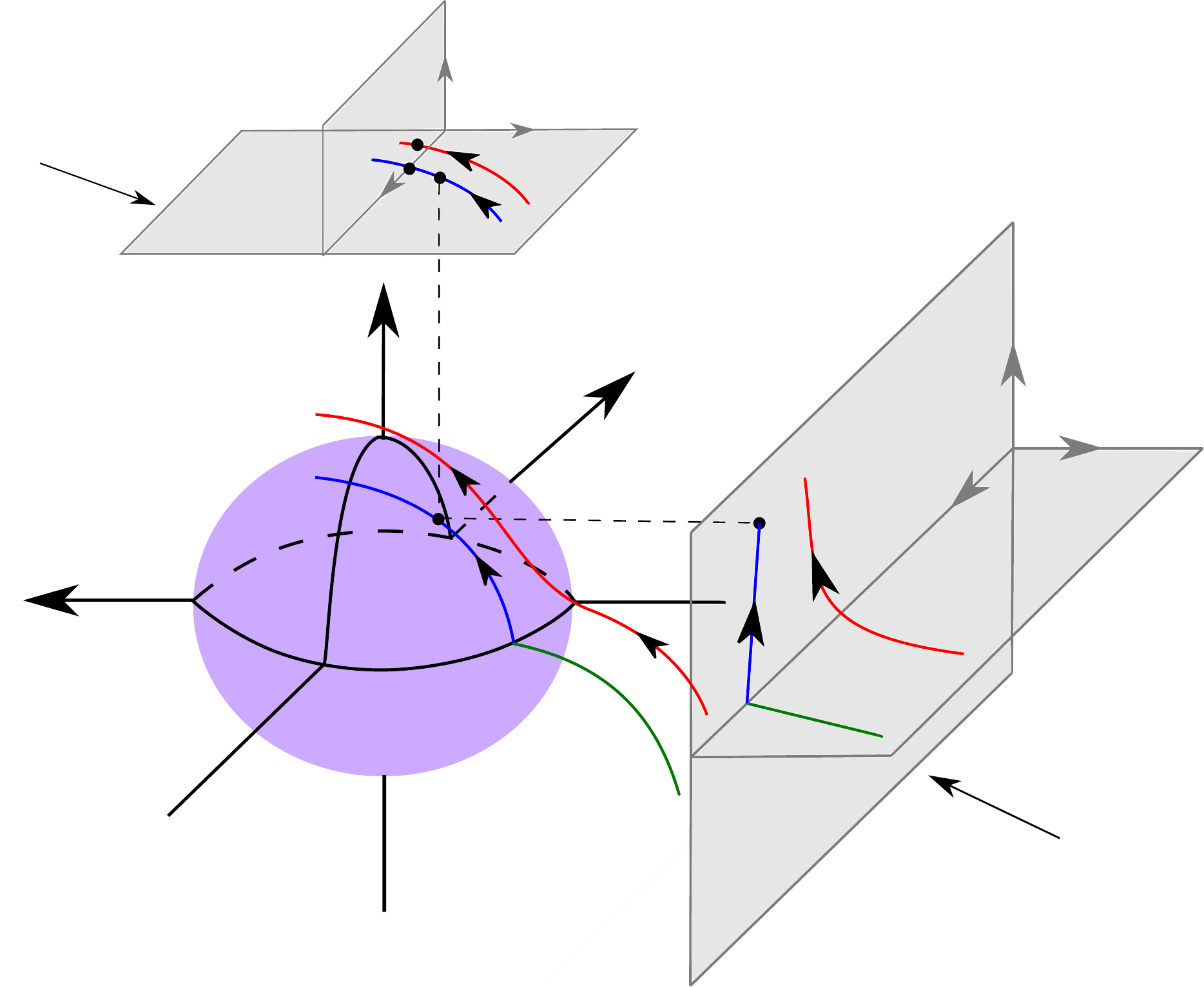}
		\put(52,52){$\ov{x}$}
		\put(-3,32){$\ov{y}$} 
		\put(50,19){{\scriptsize $S_{a}$ \par}}
		\put(42,67.5){{\scriptsize $S^2_{a,\e}$ \par}}
		\put(89,11){$\kappa_{1}$}
		\put(75,30){{\scriptsize $S^1_{a,\e}$ \par}}
		\put(49.6,26.3){{\scriptsize $S_{a,\e}$ \par}}
		\put(71,23){{\scriptsize $S_{a,1}$ \par}}
		\put(28,57){$\ov{\e}$}
		\put(-4,70){$\kappa_{2}$}
		\put(76,42){{\tiny $x_1$ \par}}
		\put(90,47){{\tiny $r_1$ \par}}
		\put(85,53){{\tiny $\e_1$ \par}}
		\put(38,78){{\tiny $r_2$ \par}}
		\put(28,67.5){{\tiny $x_2$ \par}}
		\put(43,74){{\tiny $y_2$ \par}}
		\end{overpic}
		\bigskip
	\end{center}
	\caption{The 2 charts of blow-up, critical manifold $S_a,$ and Fenichel manifold $S_{a,\e}.$}
	\label{fig00}
\end{figure}
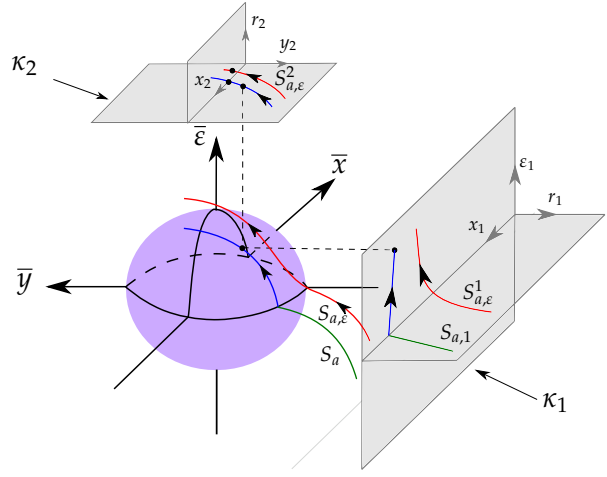

\begin{proposition}\label{varpext}
There exist $a>0$ and $\e^*>0$ such that, for each $\e\in(0,\e^*],$ the forward trajectory of system \eqref{fastsystem}, starting at any point in the set 
\[
\left(\e^{\la^*}(x_2^*-a)\,,\,\e^{\la^*}(x_2^*+a)\right)\times\left(1+\e^{\frac{2k-1}{1+2k(n-1)}}(y_2^*-a)\,,\,1+\e^{\frac{2k-1}{1+2k(n-1)}}(y_2^*+a)\right),
\]
 intersects the line $\{\widehat{y}=1\}.$ In particular, the Fenichel manifold $S_{a,\e}$ intersects $\{\widehat{y}=1\}$ at $(x,\widehat{y},\e)=(x_{\e},1,\e),$ where 
\begin{equation}\label{xeprop}
x_{\e}=\e^{\la^*}\eta+\mathcal{O}\left(\e^{\la^*+\frac{1}{1+2k(n-1)}}\right),
\end{equation}
with $\eta$ satisfying \eqref{nsigmank}. 
\end{proposition}
\begin{proof}
Denote $S_{a,\e}^2\defeq \psi_{12}(S_{a,\e}^1).$ Notice that $S_{a,\e}^2\cap\{y_2=y_2^*\}$ is $\CO(\e^{\frac{1}{1+2k(n-1)}})$ close to $(x_2^*,y_2^*,0).$ From Lemma \ref{lemma4}, the forward orbit of $(x_2^*,y_2^*,0)$ intersects $\{y_{2}=0\}$ transversally at $(x_{2},y_{2},r_{2})=(\eta,0,0).$ Thus,  from the \textit{Implicit Function Theorem}, $S_{a,\e}^2$ also intersect $\{y_2=0\}$ transversally at
$$\left(\eta+\mathcal{O}(\e^{\frac{1}{1+2k(n-1)}}),0,\e^{\frac{1}{1+2k(n-1)}}\right),$$
for $\e>0$ sufficiently small. Furthermore, there exist $a>0$ and $b>0$ sufficiently small such that any forward trajectory of \eqref{eq8}, starting at the set
$$\Xi=(x_2^*-a,x_2^*+a)\times(y_2^*-a,y_2^*+a)\times(0,b],$$ 
also intersects the set $\{y_{2}=0\}.$

Going back to the original coordinates, we conclude that the forward flow of $S_{a,\e}$ intersect $\{\widehat{y}=1\}$ at $(x,\widehat{y},\e)=(x_{\e},1,\e)$ with 
$$x_{\e}  = \e^{\la^*}\left(\eta+\mathcal{O}\left(\e^{\frac{1}{1+2k(n-1)}}\right)\right)=\e^{\la^*}\eta+\mathcal{O}\left(\e^{\la^*+\frac{1}{1+2k(n-1)}}\right).
$$ Moreover, writing the $\Xi$ in the original coordinates we conclude that, for every $\e\in (0,\e^*],$ $\e^*=b^{1+2k(n-1)},$ any trajectory of system \eqref{fastsystem} starting at the set 
\[
\left(\e^{\la^*}(x_2^*-a)\,,\,\e^{\la^*}(x_2^*+a)\right)\times\left(1+\e^{\frac{2k-1}{1+2k(n-1)}}(y_2^*-a)\,,\,1+\e^{\frac{2k-1}{1+2k(n-1)}}(y_2^*+a)\right)
\]
intersects the line $\{\widehat{y}=1\}.$\end{proof}

\section{Upper Transition Map}\label{sec:uflightmap}

This section is devoted to the proof of Theorem \ref{ta}. For this, we need to guarantee that under some conditions the flow of the regularized system $Z_{\e}^{\Phi}$ near a visible regular-tangential singularity defines a map between two vertical sections. Thus, it will be convenient to write this map as the composition of three maps, namely $P^u,$ $Q^u_\e$ and $R^u$ (see Figure \ref{figMAP2}). The map $Q^u_\e$ will be defined through the flow of  $Z_{\e}^{\Phi}$ restricted to the band of regularization, and the maps $P^u$ and $R^u$ will be given by the flow of  $Z_{\e}^{\Phi}$ defined outside the  band of regularization.
In what follows, we shall properly define these maps.

\begin{figure}[h]
	\begin{center}
		\begin{overpic}[scale=0.37]{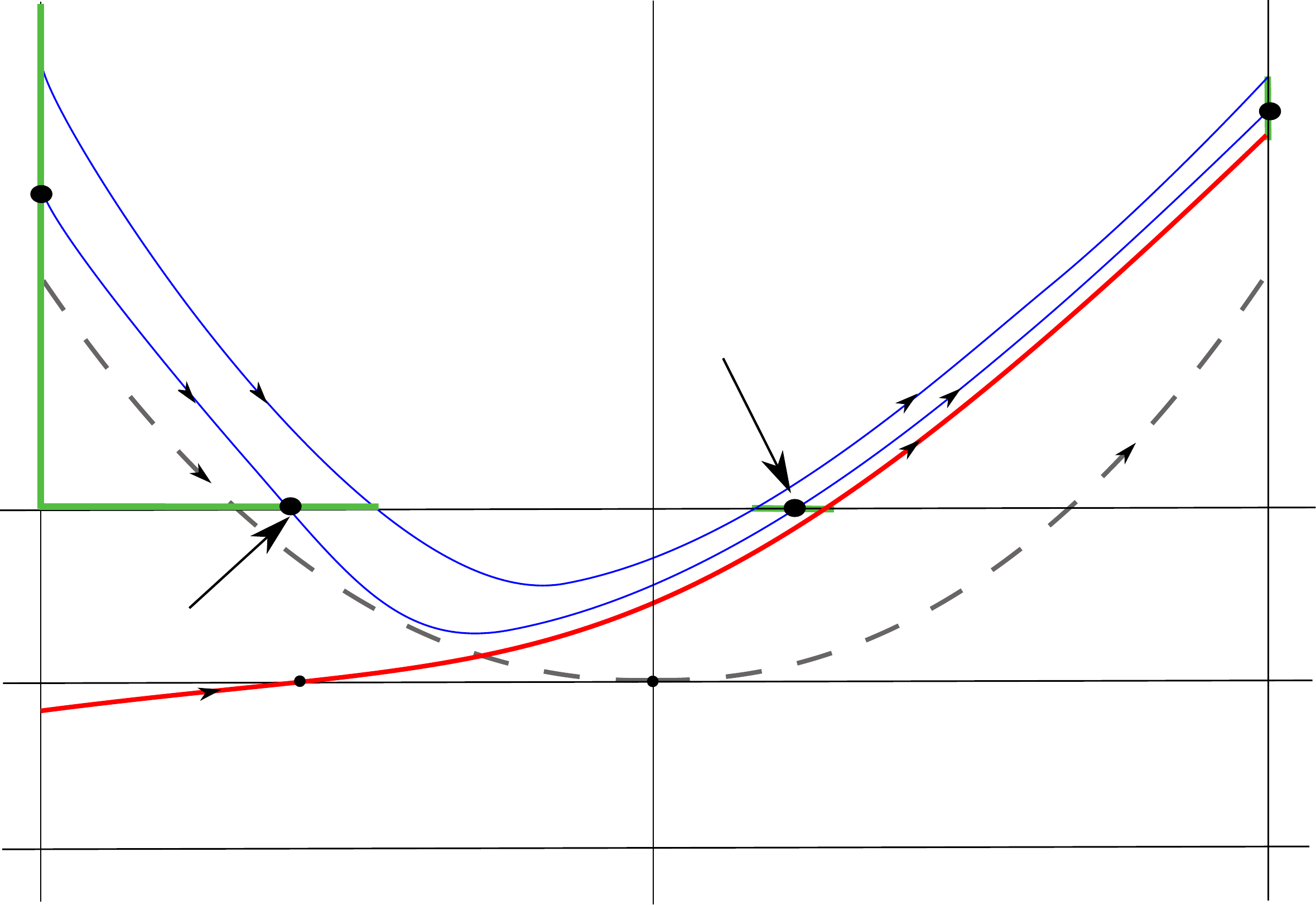}
		
		\put(0,54){$y$}
		\put(6,19.5){\scriptsize $P^u(y)=x$ \par}
		\put(50,44){$Q^u_\e(x)=z$}
		\put(98,59){$R^u(z)$}
		\put(93,-3){$x=\theta$}
		\put(-1,-3){$x=-\rho$}
		\put(98,19){$\Sigma$}
		\put(98,32){$y=\e$}
		\put(98,6){$y=-\e$}
	
		\end{overpic}
		
		\bigskip
		
	\end{center}
	\caption{Dynamics of the maps $P^u,$ $Q_\e^u$ and $R^u.$ The dotted curve is the trajectory of $X^+$ passing through the visible $2k$-multiplicity contact with $\Sigma$ with $(0,0).$}
	\label{figMAP2}
	\end{figure}
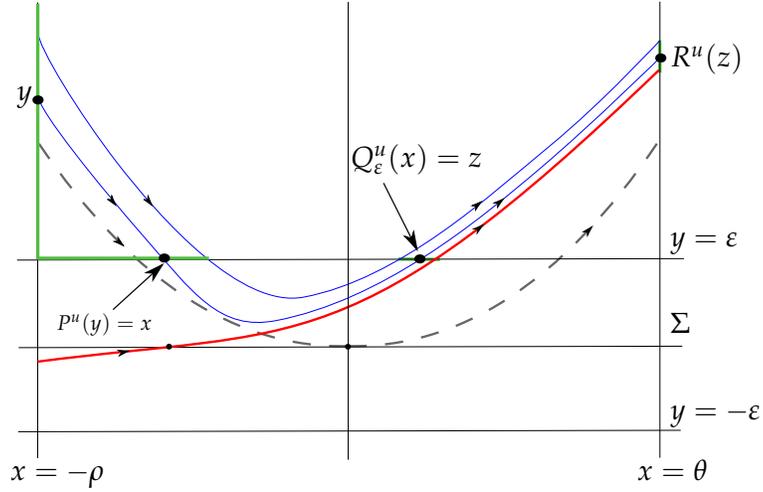

\subsection{Tangential points and transversal sections} \label{tanp}
 
In Proposition \ref{varpext}, we have proven that the Fenichel manifold of system \eqref{regsys} intersects $\{y=\e\}$ at $(x_\e,\e)$ (see \eqref{xeprop}). Now, we shall prove that if $(\psi(\e),\e)$ is a tangential contact of $Z^\Phi_\e$ with the line $\{y=\e\},$ then $x_\e>\psi(\e).$
\begin{lemma} \label{lemmaux} Let $\psi(\e)$ be a tangential contact of the vector field $Z^\Phi_\e$ \eqref{regsys} with $y=\e.$ Then,
\begin{itemize}
\item[(a)] $\psi(\e)=\mathcal{O}(\e^{\frac{1}{2k-1}}),$ and
	\item[(b)] $x_\e>\psi(\e)$ for $\e$ sufficiently small, where $x_{\e}$ is given in Proposition \ref{varpext}.
\end{itemize}

\end{lemma}
\begin{proof}
First, we shall prove statement $(a).$ Let $\psi:[0,\e_0]\rightarrow \mathbb{R}$ be a function, defined for $\e_0>0$ small, satisfying $\pi_{2} Z^{\Phi}_\e(\psi(\e),\e)=0$ for every $\e\in[0,\e_0].$ Here, $\pi_{2} Z^{\Phi}_\e$  denotes the second component of $Z^{\Phi}_\e.$ Then,

$$\begin{array}{lllll} 0  &= & \pi_{2} Z^{\Phi}_\e(\psi(\e),\e)\\
& = &\displaystyle \frac{1}{2}\Big(f(\psi(\e),\e)(1+\Phi(1))+(1-\Phi(1))\Big)\\ 
& = &\displaystyle f(\psi(\e),\e)\\
& = &\displaystyle\alpha \psi(\e)^{2k-1}+\psi(\e)^{2k-1}\widetilde{g}(\psi(\e))+\e \vartheta(\psi(\e),\e),\\  
\end{array}$$ where $\widetilde{g}=\CO(x)$ is a continuous function such that $g(x)=x^{2k-1}\widetilde{g}(x).$ Then,
$$\psi(\e)^{2k-1}=-\frac{\e \vartheta(\psi(\e),\e)}{\alpha+\widetilde{g}(\psi(\e))}\eqdef A(\e).$$ 
Notice that 
\begin{equation*}\label{limA}\left|A(\e)\right|=\left| \frac{\vartheta(\psi(\e),\e)}{\alpha+\widetilde{g}(\psi(\e))}\right|\e\leqslant \max_{\e\in[0,\e_0], x\in \ov{B}}\left| \frac{\vartheta(x,\e)}{\alpha+\widetilde{g}(x)}\right|\e= C\e,
\end{equation*} 
where $B\subset\R$ is a neighbourhood of $0.$ Therefore, $A(\e)=\mathcal{O}(\e),$ i.e. $\psi(\e)=\mathcal{O}(\e^{\frac{1}{2k-1}}).$ 

Now, we shall prove statement $(b).$ From Proposition \ref{varpext},
$$x_{\e}=\e^{\la^*}\eta+\mathcal{O}\left(\e^{\la^*+\frac{1}{1+2k(n-1)}}\right),$$ where $\la^*=\frac{n}{1+2k(n-1)}$ and $\eta>\sigma_{n,k},$ where $\sigma_{n,k}$ satisfies \eqref{nsigmank}. 

First, suppose that $n>2k-1.$ Then, $\frac{1}{2k-1}>\la^*$ and $\eta>0.$ Hence, by statement $(a),$ we conclude that $x_{\e}>\psi(\e).$ 

 Finally, suppose that $n=2k-1,$ with $k\neq 1.$ In this case, $\la^*=1/n.$ Define
$$a(\e)=-\left(\frac{\vartheta(\psi(\e),\e)}{\alpha+\widetilde{g}(\psi(\e))}\right)^{\frac{1}{n}}.$$
Notice that $\psi(\e)=\e^{\frac{1}{n}}a(\e).$ Statement $(a)$ implies that $\psi$ is continuous at $\e=0$ and $\psi(0)=0.$ Therefore, $a(\e)$ is also continuous at $\e=0$ and $a(0)=-\big(\vartheta(0,0)/\alpha\big)^{\frac{1}{n}}.$ Defining $r(\e)=a(\e)-a(0)$ we conclude that 
$$\psi(\e)=\e^{\frac{1}{n}} a(\e)=\e^{\frac{1}{n}} a(0)+\e^{\frac{1}{n}} r(\e).$$ Since $\eta>\sigma_{n,\frac{n+1}{2}}=a(0)$ and $r(0)=0,$ we conclude $x_{\e}>\psi(\e).$\end{proof}

Statement $(i)$ of Theorem \ref{ta} will follows from the next result.

\begin{proposition}\label{trans}
Consider the Filippov system $Z=(X^+,X^-)_{\Sigma}$ given by \eqref{Xnf}, for some $k\geqslant 1,$ and $y_{\rho,\la}^{\e}$ and $y_\T^\e$ given in \eqref{ye}. 
 For $n\geqslant\max\{2,2k-1\},$ let $\Phi\in C_{ST}^{n-1}$  be given as \eqref{Phi} and consider the regularized system $Z_{\e}^{\Phi}$  \eqref{regsys}. Then, there exist $\rho_0,\T_0>0$ such that the vertical segments 
\[
\widehat V_{\rho,\la}^{\e}=\{-\rho\}\times [\e,y_{\rho,\la}^{\e}]\,\,\text{ and }\,\, \widetilde V_{\T}^{\e}=\{\T\}\times[y_\T^\e,y_\T^\e+r e^{-\frac{c}{\e^q}}]
\]
and the horizontal segments 
\[
\widehat H_{\rho,\la}^{\e}=[-\rho,-\e^{\la}]\times\{\e\}\,\,\text{ and }\,\, \overleftarrow{H}_{\e}=[x_{\e}-r e^{-\frac{c}{\e^q}},x_\e]\times\{\e\}
\]
are transversal sections for $Z_{\e}^{\Phi}$ for every $\rho\in(\e^\la,\rho_0],$ $\T\in[x_\e,\T_0],$ $\la\in(0,\la^*),$ with $\la^*= \frac{n}{2k(n-1)+1},$ constants $c,r,q>0,$ and $\e>0$ sufficiently small.
\end{proposition}

\begin{proof}
First of all, we take $\rho_0,\T_0>0$ sufficiently small in order that the points $(\rho_0,0)$ and $(\T_0,0)$ are contained in $U,$ domain of $Z$. Given $(-\rho,y_1)\in\widehat V_{\rho,\la}^{\e}$ and $(\theta,y_2)\in\widetilde V_{\T}^{\e},$ we have
\[
\begin{array}{c}
\Big\langle Z_{\e}^{\Phi}\left(-\rho,y_1\right),(1,0)\Big\rangle=\pi_1 Z_{\e}^{\Phi}\left(-\rho,y_1\right)=X^+_1\left(-\rho,y_1\right)=1\neq 0,\vspace{0.2cm}\\

\Big\langle Z_{\e}^{\Phi}\left(\theta,y_2\right),(1,0)\Big\rangle=\pi_1 Z_{\e}^{\Phi}\left(\theta,y_2\right)=X^+_1\left(\theta,y_2\right)=1\neq 0,
\end{array}
\] respectively, where $\pi_1 Z_{\e}^{\Phi}$ denote the first component of $Z_{\e}^{\Phi}.$ Hence, $V_{\rho,\la}^{\e}$ and $V_{\T}^{\e}$  are transversal sections for $Z_{\e}^{\Phi}.$

From Lemma \ref{lemmaux}, we know that any branch of zeros $\psi(\e)$ of the equation $\pi_{2} Z_{\e}^{\Phi}(x,\e)=0$ satisfies  $\psi(\e)=\mathcal{O}(\e^{\frac{1}{2k-1}}).$ In other words, the zeros of $\pi_{2} Z_{\e}^{\Phi}(x,\e)$ lie in an $\CO(\e^{\frac{1}{2k-1}})$ neighbourhood of $0.$ Since $\rho\in(\e^\la,\rho_0],$ $\T\in[x_\e,\T_0],$ $\la\in(0,\la^*),$ the intervals $\widehat H_{\rho,\la}^{\e}$ and $\overleftarrow{H}_{\e}$ are always away from any $\CO(\e^{\frac{1}{2k-1}})$ neighbourhood of $0$ and, then, $\pi_{2} Z_{\e}^{\Phi}(x,\e)$ does not admit zeros inside these sections. Consequently, given $(x_1,\e)\in \widehat H_{\rho,\la}^{\e}$ and $(x_2,\e)\in\overleftarrow{H}_{\e}$ we have
\[
\begin{array}{l}
\Big\langle Z_{\e}^{\Phi}\left(x_1,\e\right),(0,1)\Big\rangle=\pi_{2} Z_{\e}^{\Phi}\left(x_1,\e\right)\neq 0,\vspace{0.2cm}\\
\Big\langle Z_{\e}^{\Phi}\left(x_2,\e\right),(0,1)\Big\rangle=\pi_2 Z_{\e}^{\Phi}\left(x_2,\e\right)\neq 0.
\end{array}
\]
Hence, $\widehat H_{\rho,\la}^{\e}$ and $\overleftarrow{H}_{\e}$ are transversal sections for $Z_{\e}^{\Phi}.$\end{proof}

\subsection{Construction of the map $P^u$}\label{P}

First, we shall see that the backward trajectory of $X^+$ \eqref{Xnf} starting at $(-\e^{\lambda},\e)$ reaches the straight line $\{x=-\rho\}$ at $(-\rho,y_{\rho,\lambda}^{\e})$ (see \eqref{ye}). After that, the map will be obtained through a Poincar\'{e}-Bendixson argument.

Accordingly, define $\mu:I_{(x,y)}\times U\times[0,\rho_0]\rightarrow\mathbb{R}$ by $$\mu(t,x,y,\rho)=\varphi^{1}_{X^+}(t,x,y)+\rho,$$ where $\varphi_{X^+}=(\varphi^1_{X^+},\varphi^2_{X^+})$ is the flow of $X^+,$  $I_{(x,y)}$ is the interval of definition of $t\mapsto\varphi_{X^+}(t,x,y),$ and $U\subset\R^2$ is a neighbourhood of $(0,0).$ Since $\mu(0,0,0,0)=0$ and $\frac{\partial}{\partial t}\mu(0,0,0,0)=1,$ by the \textit{Implicit Function Theorem} there exists a unique smooth function $(x,y,\rho)\mapsto t_{\rho}(x,y),$ defined in a neighbourhood of $(x,y,\rho)=(0,0,0),$ such that $t_0(0,0)=0$ and $\mu(t_{\rho}(x,y),x,y,\rho)=0,$ i.e. $\varphi^{1}_{X^+}(t_{\rho}(x,y),x,y)=-\rho.$ Therefore, for $\rho>0$ and $\e>0$ sufficiently small, the backward trajectory of $X^+$ starting at $(-\e^{\lambda},\e)$ reaches the straight line $\{x=-\rho\}$ at $$\Big(-\rho,\varphi^{2}_{X^+}(t_{\rho}(-\e^{\lambda},\e),-\e^{\lambda},\e)\Big).$$

In order to prove that $\varphi^{2}_{X^+}(t_{\rho}(-\e^{\lambda},\e),-\e^{\lambda},\e)=y_{\rho,\lambda}^{\e},$ we shall compute the Taylor series expansion of the function $\varphi^{2}_{X^+}(t_{\rho}(x,y),x,y)$ around $(x,y,\rho)=(0,0,0).$ Notice that
\begin{equation}\label{phi220}\begin{array}{rl}
 \varphi^{2}_{X^+}(t_{\rho}(x,y),x,y)  = & \varphi^{2}_{X^+}(t_\rho(x,0),x,0)+y\frac{\partial}{\partial y}(\varphi^{2}_{X^+}(t_\rho(x,y),x,y))\Big|_{y=0}\\&+\mathcal{O}(y^2)\\ 
= & \varphi^{2}_{X^+}(t_\rho(x,0),x,0)+y\left[\frac{\partial\varphi^{2}_{X^+}}{\partial t}(t_\rho(x,0),x,0))\frac{\partial t_\rho}{\partial y}(x,0)\right.\\
& +  \left.\frac{\partial\varphi^{2}_{X^+}}{\partial y}(t_\rho(x,0),x,0)\right]+\mathcal{O}(y^2)\\
= & \varphi^{2}_{X^+}(t_\rho(x,0),x,0)+y\left[\frac{\partial\varphi^{2}_{X^+}}{\partial t}(t_\rho(0,0),0,0))\frac{\partial t_\rho}{\partial y}(0,0)\right.\\
& +  \left.\frac{\partial\varphi^{2}_{X^+}}{\partial y}(t_\rho(0,0),0,0)\right]+\mathcal{O}(xy)+\mathcal{O}(y^2).\\
\end{array}\end{equation}
It is easy to see that $$\frac{\partial\varphi^{2}_{X^+}}{\partial t}(t_\rho(0,0),0,0))=f(-\rho,\ov{y}_{-\rho}) \text{ and } \frac{\partial t_\rho}{\partial y}(0,0)=-\frac{\partial\varphi^{1}_{X^+}}{\partial y}(t_\rho(0,0),0,0).$$ This last equality is obtained implicitly from $\varphi^{1}_{X^+}(t(0,y,\rho),0,y)=-\rho$ and using that $\frac{\partial\varphi^{1}_{X^+}}{\partial t}(t_\rho(0,0),0,0)=1.$ Thus, substituting the above relations into \eqref{phi220}, we get
\begin{equation}\label{phi22}\begin{array}{rl} \varphi^{2}_{X^+}(t_{\rho}(x,y),x,y)  = & \varphi^{2}_{X^+}(t_\rho(x,0),x,0)+y\left[-f(-\rho,\ov{y}_{-\rho})\frac{\partial\varphi^{1}_{X^+}}{\partial y}(t_\rho(0,0),0,0)\right.\\
& +  \left.\frac{\partial\varphi^{2}_{X^+}}{\partial y}(t_\rho(0,0),0,0)\right]+\mathcal{O}(xy)+\mathcal{O}(y^2).\\
\end{array}\end{equation}
Expanding the coefficient of $y$ in \eqref{phi22} around $\rho=0,$  we have
 $$-f(-\rho,\ov{y}_{-\rho})\frac{\partial\varphi^{1}_{X^+}}{\partial y}(t_\rho(0,0),0,0)+\frac{\partial\varphi^{2}_{X^+}}{\partial y}(t_\rho(0,0),0,0)=1+\mathcal{O}(\rho).$$
Thus, substituting the above equality into \eqref{phi22}, we obtain that 
\begin{equation*}\label{psim}
\varphi^{2}_{X^+}(t_{\rho}(x,y),x,y)=\varphi^{2}_{X^+}(t_\rho(x,0),x,0)+y(1+\mathcal{O}(\rho))+\mathcal{O}(xy)+\mathcal{O}(y^2).
\end{equation*} 
Furthermore, from \cite[Theorem $A$]{AndGomNov19}, we know that
$$\varphi^{2}_{X^+}(t_\rho(x,0),x,0)=\ov{y}_{-\rho}+\beta x^{2k}+\mathcal{O}(x^{2k+1}),$$ where $\sgn(\beta)=-\sgn((X^+)^{2k}h(0,0)),$ i.e. $\beta<0.$ Thus, we conclude that
\begin{equation*} \label{psif}
\varphi^{2}_{X^+}(t_{\rho}(x,y),x,y)=\ov{y}_{-\rho}+\beta x^{2k}+\mathcal{O}(x^{2k+1})+y(1+\mathcal{O}(\rho))+\mathcal{O}(xy)+\mathcal{O}(y^2).
\end{equation*}
Taking $x=-\e^{\la}$ and $y=\e,$ we obtain
\begin{equation}\label{ey}\varphi^{2}_{X^+}\Big(t(-\e^\la,\e,\rho),-\e^\la,\e\Big)=\overline{y}_{-\rho}+\e\Big(1+\mathcal{O}(\rho)\Big)+\beta \e^{2k\la}+\mathcal{O}(\e^{(2k+1)\la})+\mathcal{O}(\e^{1+\la}),\end{equation}
which we have called by $y^\e_{\rho,\la}.$

Finally, consider the region
$$\mathcal{R}=\left\{(x,y):-\rho\leqslant x\leqslant -\e^\la, \e\leqslant y\leqslant\varphi^{2}_{X^+}\Big(t,-\e^\la,\e\Big), \forall t\in[0,t(-\e^\la,\e,\rho)]\right\},$$ which is delimited by $ \widehat V_{\rho,\la}^{\e},$ $ \widehat H_{\rho,\la}^{\e},$ and the arc-orbit connecting $(-\rho,y^\e_{\rho,\la})$ with $(-\e^{\la},\e).$ Since $X^+$ has no singularities inside $\mathcal{R},$ we conclude that the forward trajectory of $X^
+$ starting at any point of the transversal section $ \widehat V_{\rho,\la}^{\e}$ must leave $\mathcal{R}$ through the transversal section $ \widehat H_{\rho,\la}^{\e}.$ This naturally defines the map 
$ P^u: \widehat V_{\rho,\la}^{\e} \longrightarrow  \widehat H_{\rho,\la}^{\e}.$

\subsection{Exponential attraction and construction of the map $Q^u_{\e}$}\label{expat}

As we saw in Section \ref{sec:fenichelmanifold}, the Fenichel manifold $S_{a,\e}$ of \eqref{slowsystem} is described as a graph $$\widehat{y}=m(x,\e), \quad -L\leqslant x\leqslant-N, \quad 0\leqslant \e\leqslant \e_0,$$ where $m(x,\e)$ is a smooth function, and $L>N>0$ and $\e_0>0$ are small parameters. Notice that 
\begin{equation}\label{eqm0}
m(x,0)=m_{0}(x)=\phi^{-1}\left(\frac{1+\alpha x^{2k-1}+g(x)}{1-\alpha x^{2k-1}-g(x)}\right),\end{equation}
which is the critical manifold of the system \eqref{fastsystem}$_{\e=0}.$ Thus, we write
\[
m(x,\e)=m_{0}(x)+\e m_{1}(x)+\mathcal{O}(\e^2),
\]
for every $-L\leqslant x\leqslant -N$ and $0\leqslant \e\leqslant \e_0.$ Since  $S_{a,\e}$ is an invariant manifold for \eqref{fastsystem}, the function $m(x,\e)$ satisfies
\begin{equation*}\label{mxe}
\e \dfrac{\p m}{\p x}(x,\e)=\frac{1+f(x,\e m(x,\e))+\phi(m(x,\e))(f(x,\e m(x,\e))-1)}{1+\phi(m(x,\e))}.
\end{equation*}
Hence, using that
 \begin{equation}\label{eq22}\phi'(m_0(x))=\frac{2\alpha(2k-1)x^{2k-2}+2g'(x)}{m'_0(x)(-1+\alpha x^{2k-1}+g(x))^2},\end{equation}
we can compute
\begin{equation}\label{eq21}m_{1}(x)=\frac{-m'_{0}(x)(m'_{0}(x)-m_0(x)\vartheta(x,0))}{\alpha(2k-1)x^{2k-2}+g'(x)}.\end{equation}

The next result provides some estimations for $m_0(x).$
 
\begin{proposition}
For $-L\leqslant x<0$ there exist positive constants $C_1,$ $C_2$ such that
\begin{equation}\label{eq28}
\begin{array}{lllll} C_1 \sqrt[n]{|x|^{2k-1}}  &\leqslant & 1-m_0(x)\leqslant C_2 \sqrt[n]{|x|^{2k-1}},\\
\dfrac{C_1}{\sqrt[n]{|x|^{n-2k+1}}} &\leqslant & m'_0(x)\leqslant \dfrac{C_2}{\sqrt[n]{|x|^{n-2k+1}}}.\\   
\end{array}
\end{equation} 
\end{proposition}
\begin{proof}
In order to obtain the above estimations,  we consider the equation $\phi(\hat y)=\phi(m_0(x))$ for $-1<\hat y<1$ and  $-L\leqslant x<0.$ Of course, $\hat y=m_0(x).$

On the other hand, from \eqref{eqm0},
\begin{equation}\label{phim0}
\phi(m_0(x))=1+2\al x^{2k-1}+\CO(x^{2k}).
\end{equation} 
In addition, expanding $\phi(\hat y)$ around $\hat y=1$ we get
\begin{equation}\label{rf}
\phi(\widehat{y})=1+\frac{\phi^{(n)}(1)}{n!}(\widehat{y}-1)^n+\mathcal{O}((\widehat{y}-1)^{n+1}).
\end{equation}
Subtracting \eqref{rf} from \eqref{phim0}  we get that the equation $\phi(\hat y)=\phi(m_0(x))$ is equivalent to the system
\[
\left\{\begin{array}{l}
s=(\hat y-1)^n,\vspace{0.1cm}\\
u=x^{2k-1},\\
H(s,u)\defeq  \frac{\phi^{(n)}(1)}{n!}s-2\al u+\mathcal{O}(s^\frac{n+1}{n})+\CO(u^\frac{2k}{2k-1})=0.
\end{array}\right.
\]
Since $H(0,0)=0$ and $\frac{\p H}{\p s}(0,0)=\frac{\phi^{(n)}(1)}{n!}\neq0,$ the \textit{Implicit Function Theorem} implies the existence of a unique function $s(u),$ defined in a small neighborhood of $u=0,$ such that $s(0)=0$ and $H(s(u),u)=0.$ Moreover,
\[
s(u)=\dfrac{2\al n!}{\phi^{(n)}(1)}u +\CO(u^2).
\]
Therefore, the equation  $\phi(\hat y)=\phi(m_0(x))$ is solved as $\hat y=1-((-1)^{n}s(x^{2k-1}))^{\frac{1}{n}}.$ Recall that $\phi=\Phi\big|_{[-1,1]},$ where $\Phi\in C^{n-1}_{ST}.$ Thus, from Definition \ref{Cn-1ST},
$\sgn\big(\phi^{(n)}(1)\big)=(-1)^{n+1}.$ 
Consequently,
\begin{equation}\label{eq23}m_{0}(x)=\hat y=1-\sqrt[n]{\frac{2\alpha n!}{|\phi^{(n)}(1)|}}\sqrt[n]{|x|^{2k-1}}+\mathcal{O}(|x|^{1+\frac{2k-1}{n}}),\,\,-L\leqslant x\leqslant 0.\end{equation} 

Finally, the inequalities \eqref{eq28} are obtained directly from \eqref{eq23}.\end{proof}

The next proposition is a technical result, which is proven in Appendix. 

\begin{proposition}\label{prop:aux}
Consider $-L<-N<0$ and $0<\la\leqslant\la^*=\frac{n}{2k(n-1)+1}.$ Then, there exist $K>0$ and $\e_0>0,$ such that, if $0\leqslant\e\leqslant\e_0$ the invariant manifold $\widehat{y}=m(x,\e)$ satisfies \begin{equation}\label{eq24}m_{0}(x)-\frac{\e K}{\sqrt[n]{x^{2k(n-2)+2}}}\leqslant m(x,\e)\leqslant m_{0}(x),\end{equation}
for $-L\leqslant x\leqslant-\e^{\la}.$
\end{proposition}

From Theorem \ref{thm:fenichel} (Fenichel Theorem), we know that, for $\e>0$ sufficiently small, the Fenichel manifold $S_{a,\e}$ exponentially attracts all the solutions with initial conditions $(x_0,1),$ with $-L\leqslant x_0\leqslant -N,$ for any small positive real numbers $L>N.$ In the next result, we show that this exponential attraction holds  for any $(x_0,1)$ with $-L\leqslant x_0\leqslant -\e^\la.$ Consider the equation for the orbits of system \eqref{fastsystem}
\begin{equation}\label{eq29}
\e\frac{d\widehat{y}}{dx}=\frac{1+f(x,\e \widehat{y})+\phi(\widehat{y})(f(x,\e \widehat{y})-1)}{1+\phi(\widehat{y})}.
\end{equation}

\begin{proposition}\label{propexp}
Fix $0<\la<\la^*=\frac{n}{2k(n-1)+1}.$ Let $x_0\in[-L,-\e^{\la}]$ and consider the solution $\widehat y(x,\e)$ of the differential \eqref{eq29} satisfying $\widehat{y}(x_{0},\e)=1.$ Then, there exist positive numbers $c$ and $r$ such that
\begin{equation*}\label{diffym}
|\widehat y(x,\e)-m(x,\e)| \leqslant  r e^{-\frac{c}{\e}\left(|x_0|^{\frac{1}{\la^*}}-|x|^{\frac{1}{\la^*}}\right)},\\ 
\end{equation*}
for $x_0\leqslant x\leqslant-\e^{\la^*}.$
\end{proposition}

\begin{proof}
Performing the change of variables $\omega=\widehat{y}-m(x,\e)$ in equation \eqref{eq29}, we get 
\begin{equation}\label{eq31}\e\frac{d\omega}{dx}=-\xi(x,\e)\phi'(m(x,\e))\omega-\xi(x,\e)F(x,\omega,\e),\end{equation}
where, 
$$F(x,\omega,\e)=\phi(m(x,\e)+\omega)-\phi(m(x,\e))-\phi'(m(x,\e))\omega$$ and
$$\begin{array}{rl} \displaystyle\xi(x,\e)
=&  \dfrac{2}{\Big(1+\phi(m(x,\e))\Big)\Big(1+\phi(m(x,\e)+\omega(x,\e))\Big)}\\ 
& +  \dfrac{\e\Big(m(x,\e)\vartheta(x,\e m(x,\e))-(\omega(x,\e)+m(x,\e))\vartheta(x,\e(\omega(x,\e)+m(x,\e)))\Big)}{\phi(m(x,\e)+\omega(x,\e))-\phi(m(x,\e))}.\\ 
\end{array}$$
Here, we are denoting $\omega(x,\e)=\widehat{y}(x,\e)-m(x,\e),$ which is the solution of \eqref{eq31} with initial condition $\omega(x_0,\e)=1-m(x_0,\e).$ 
Therefore, we also have that 
\begin{equation*}\label{omega1}
\omega(x,\e)=e^{-\frac{1}{\e}\int_{x_0}^x\xi(s,\e)\phi'(m(s,\e))ds}\widetilde{\omega}(x,\e),
\end{equation*}
 where $$\widetilde{\omega}(x,\e)=\omega(x_0,\e)-\frac{1}{\e}\int_{x_0}^x e^{\frac{1}{\e}\int_{x_0}^\nu \xi(s,\e)\phi'(m(s,\e))ds}\xi(\nu,\e)F(\nu,\omega(\nu,\e),\e)d\nu.$$ 

In what follows we shall estimate $|\omega(x,\e)|.$ First, notice that $F$ writes
\begin{equation}\label{eq32}
F(x,\omega,\e)=A(x,\e)\omega,
\end{equation} where $$A(x,\e)=\displaystyle\int_0^1\phi'(m(x,\e)+s\omega(x,\e))-\phi'(m(x,\e))ds.$$ We claim that $A(x,\e)$ is negative for $-L\leqslant x\leqslant 0$ and $L,\e>0$ small enough. Indeed,
from \eqref{rf}, we obtain 
\begin{equation}\label{d2phi}
\phi''(\widehat{y})=\frac{\phi^{(n)}(1)}{(n-2)!}(\widehat{y}-1)^{n-2}+\mathcal{O}((\widehat{y}-1)^{n-1}),\hspace{0.1cm} \widehat{y}\leqslant 1.
\end{equation}
Again, recall that $\phi=\Phi\big|_{[-1,1]},$ where $\Phi\in C^{n-1}_{ST}.$ Hence, from Definition \ref{Cn-1ST},
$\sgn\big(\phi^{(n)}(1)\big)=(-1)^{n+1}.$ 
Thus, from \eqref{d2phi}, we get the existence of $\eta>0$ such that $\phi''(\widehat{y})<0$ for all $1-\eta<\widehat{y}<1.$ This means that $\phi'$ is decreasing for $1-\eta<\widehat{y}\leqslant1.$ Notice that
\begin{equation}\label{d1}
m(x,\e)\leqslant m(x,\e)+s\omega(x,\e)\leqslant(1-s)m(x,\e)+s\leqslant 1, \hspace{0.1cm}\text{for all} \hspace{0.1cm}0\leqslant s\leqslant 1.\end{equation} 
Thus, it remains to show that $m(x,\e)+s\omega(x,\e),\,m(x,\e)>1-\eta$ for $-L\leqslant x< 0$ and $L,\e>0$ small enough. 
From Proposition \ref{prop:aux} and \eqref{eq28}, we have that
\begin{equation}\label{d2}m(x,\e)\geqslant m_{0}(x)-\frac{\e K}{\sqrt[n]{x^{2k(n-2)+2}}}\geqslant 1-C_2\sqrt[n]{L^{2k-1}}-\e^{1-\la^*\left(\frac{2k(n-2)+2}{n}\right)} K,
\end{equation} for $\e,L>0$ small enough. Therefore, $L$ and $\e$ can be  taken smaller, if necessary, in order that $C_2\sqrt[n]{L^{2k-1}}+\e^{1-\la^*\left(\frac{2k(n-2)+2}{n}\right)} K<\eta.$ This implies that
\[
 m(x,\e)+s\omega(x,\e)\geqslant m(x,\e)>1-\eta.
\]
Consequently, $A(x,\e)$ is negative.

Hence, by \eqref{eq32}, we have that 
$$\begin{array}{rcl} |\widetilde{\omega}(x,\e)|&= &\displaystyle |\omega(x_0)|+\frac{1}{\e}\int_{x_0}^x |\xi(\nu,\e)A(\nu,\e)\widetilde{\omega}(\nu,\e)|d\nu\vspace{0.2cm}\\ 
&\leqslant & \displaystyle |\omega(x_0)|-\frac{1}{\e}\int_{x_0}^x \xi(\nu,\e)A(\nu,\e)|\widetilde{\omega}(\nu,\e)|d\nu.\\ 
\end{array}$$ Using Gronwall's Lemma, we get that $$|\widetilde{\omega}(x,\e)|\leqslant|\omega(x_0)|e^{-\frac{1}{\e}\int_{x_0}^x\xi(\nu,\e)A(\nu,\e)d\nu}$$ and, therefore,
$$\begin{array}{rcl} |\omega(x,\e)| &\leqslant & |\omega(x_0)|e^{-\frac{1}{\e}\int_{x_0}^x\xi(\nu,\e)(A(\nu,\e)+\phi'(m(\nu,\e)))d\nu}\\ 
&\leqslant & |\omega(x_0)|e^{-\frac{1}{\e}\int_{x_0}^x\xi(\nu,\e)(\int_0^1\phi'(m(\nu,\e)+s\omega(\nu,\e))ds)d\nu}.\\ 
\end{array}$$

To conclude this proof, notice that 
$$\displaystyle\xi(x,\e)
= \dfrac{2}{\Big(1+\phi(m_0(x))\Big)\Big(1+\phi(\omega(x,0)+m_0(x))\Big)}+\mathcal{O}(\e).$$ 
Thus, $L,\e>0$ can be taken small enough in order that $\xi(x,\e)\geqslant l>0,$ for every $x\in[-L,0].$ 
Moreover, from  \eqref{rf},  given $0<\eta<1,$ there exist positive constants $c_1,c_2>0$ such that
	$$c_1(1-\widehat{y})^{n-1}\leqslant\phi'(\widehat{y})\leqslant c_2(1-\widehat{y})^{n-1},\quad \text{for} \quad |\widehat{y}-1|<\eta. $$ 
Finally, using \eqref{d1} and \eqref{d2}, we obtain that $|m(\nu,\e)+s\omega(\nu,\e)-1|<\eta$. Hence, for $x\leqslant-\e^{\la^*},$ we have that
$$\begin{array}{rcl} |\omega(x,\e)| &\leqslant & |\omega(x_0)|e^{-\frac{c_1 l}{\e}\int_{x_0}^x(\int_0^1(1-m(\nu,\e)-s\omega(\nu,\e))^{n-1}ds)d\nu}\\ 
&\leqslant & |\omega(x_0)|e^{-\frac{c_1 l}{\e}\int_{x_0}^x(\int_0^1((1-m(\nu,\e))(1-s))^{n-1}ds)d\nu}\\ 
&\leqslant & |\omega(x_0)|e^{-\frac{c_1 l}{n\e}\int_{x_0}^x(1-m(\nu,\e))^{n-1}d\nu}\\
&\leqslant & |\omega(x_0)|e^{-\frac{c_1 l}{n\e}\int_{x_0}^x(1-m_0(\nu))^{n-1}d\nu}\\ 
&\leqslant & |\omega(x_0)|e^{-\frac{c_1 l}{n\e}\int_{x_0}^x(C_1|\nu|^{\frac{2k-1}{n}})^{n-1}d\nu}\\ 
&\leqslant & |\omega(x_0)|e^{-\frac{c}{\e}(|x_0|^{\frac{1}{\la^*}}-|x|^{\frac{1}{\la^*}})},\\ 
\end{array}$$ where $c=\frac{c_1 l C_1^{n-1}}{2k(n-1)+1}$ is a positive constant.  The inequality \eqref{eq28} has also been used. 
\end{proof}

\begin{figure}[h]
	\begin{center}
		\begin{overpic}[scale=0.57]{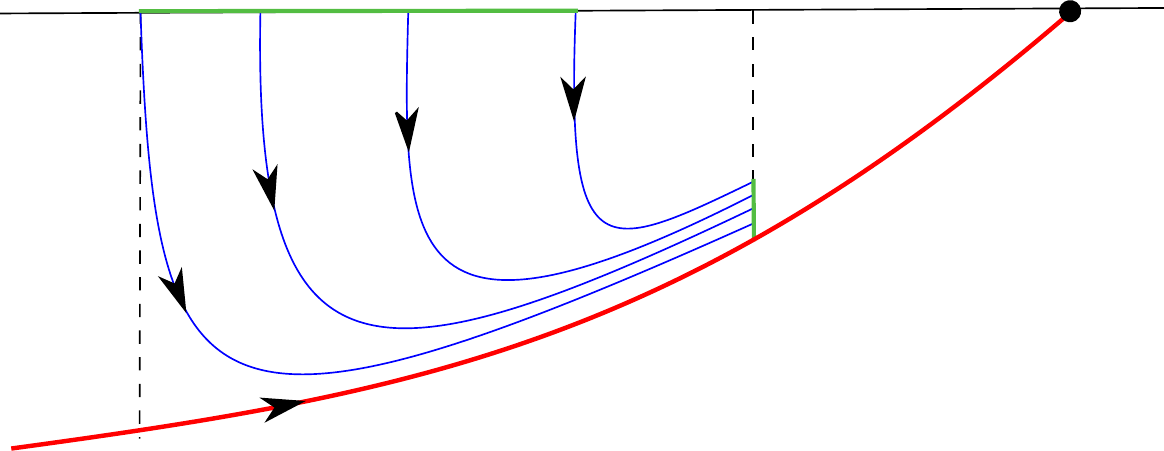}

		\put(14,5){\scriptsize $S_{a,\e}$ \par}
		\put(90,40){$x_\e$}
		\put(46,39){$-\e^\la$}
		\put(8.5,40){$-\rho$}
		\put(60,39){$-\e^{\la^*}$}
		\put(101,38){$\widehat{y}=1$}

		\end{overpic}
		\bigskip
	\end{center}
	\caption{The exponential attraction of $S_{a,\e}.$}
	\label{figMAPQU}
	\end{figure}

Fix $0<\la<\la^*.$ From Proposition \ref{propexp}, applied to $x_0=-\e^\la$ and $x=-\e^{\la^*},$ we know that there exist positive numbers $\widetilde{r}$ and $c$ such that
$$ \begin{array}{rcl} |\widehat y(-\e^{\la^*},\e)-m(-\e^{\la^*},\e)| &\leqslant&  \widetilde{r} e^{-\frac{c}{\e}\left(|-\e^\la|^{\frac{1}{\la^*}}-|-\e^{\la^*}|^{\frac{1}{\la^*}}\right)}\\ 
& = &\displaystyle r e^{-\frac{c}{\e^{q}}}, \\
\end{array}$$
where $r=\widetilde{r}e^c$ and $q=1-\frac{\la}{\la^*}$ are positive constants. Thus, 
$$\widehat{y}(-\e^{\la^*},\e)=m(-\e^{\la^*},\e)+\mathcal{O}(e^{-c/\e^q}).$$  

Hence, arguing analogously to the construction of map $P^u$ (see Section \ref{P}), any solution of the system \eqref{fastsystem} with initial condition in the interval $[-\rho,-\e^\la],$ $\e$ sufficiently small, reaches the section $x=-\e^{\la^*}$ exponentially close to the Fenichel manifold (see Figure \ref{figMAPQU}). 
From Proposition \ref{varpext}, these solutions can be continued until the section $\hat y=1.$ Going back through the rescaling $y=\e \hat y,$ we get defined the following map through the flow of \eqref{regsys},
$$\begin{array}{rcl} Q^u_{\e}:  \widehat H_{\rho,\la}^{\e}&\longrightarrow & \overleftarrow{H}_{\e}\\ 
(x,\e)&\longmapsto & \Big(x_\e+\mathcal{O}(e^{-c/\e^q}),\e\Big),
\end{array}$$ where $\widehat H_{\rho,\la}^{\e}=[-\rho,-\e^{\la}]\times\{\e\}$ and $\overleftarrow{H}_{\e}=[x_\e-re^{-\frac{c}{\e^q}},x_\e]\times\{\e\},$ for $\e>0$ small enough. 

\subsection{Construction of the map $R^u$} \label{R}

In order to define the map $R^u,$ we first  prove the following result.
\begin{proposition}\label{theorem1}
Consider the Filippov system $Z=(X^+,X^-)_{\Sigma}$ given by \eqref{Xnf}, for some $k\geqslant 1,$ and $y_{\rho,\la}^{\e}$ and $y_\T^\e$ given in \eqref{ye}. For $n\geqslant\max\{2,2k-1\},$ let $\Phi\in C_{ST}^{n-1}$  be given as \eqref{Phi} and consider the regularized system $Z_{\e}^{\Phi}$  \eqref{regsys}. Then, there exists $\T_0>0$ such that, for each $\T\in[x_{\e},\T_0],$ an extension of the Fenichel manifold $S_{a,\e}$ intersects $\{x=\theta\}$ at $(\theta,y_{\theta}^\e).$
\end{proposition}

\begin{proof}
By Proposition \ref{varpext} we know that the Fenichel manifold $S_{a,\e}$ intersects $\{y=\e\}$ at $(x_{\e},\e).$ In order to continue $S_{a,\e}$ into  $x=\theta,$ consider the solutions $(x(t),y(t))$ of the differential system \eqref{cs}
with initial condition $x(0)=x_{\e}$ and $y(0)=\e.$  Thus, $x(t)=t+x_{\e}$  and
$$y(t)=\e+\int_{0}^{t}\alpha(s+x_\e)^{2k-1}+g(s+x_\e)+y(s,\e)\vartheta(s+x_\e,y(s,\e))ds.$$
Therefore, the trajectory $(x(t),y(t))$ intersects $\{x=\theta\}$ at $(\theta,y_{\theta}^\e),$ with

$$y_{\theta}^\e\defeq y(\theta-x_{\e})=\frac{\alpha\theta^{2k}}{2k}-\frac{\alpha x_{\e}^{2k}}{2k}+\e+G_{\e}(x_{\e},\theta),$$ where $$G_{\e}(x,\theta)=\int_{x}^{\theta}{[g(s)+y(s-x,\e)\vartheta(s,y(s-x,\e))]ds}.$$

In what follows, we shall develop  $G_{\e}(x_{\e},\theta)$ in Taylor series around $(x,\theta,\e)=(0,0,0).$ First, notice that
\begin{equation}\label{eq13} G_{\e}(x,\theta)=G_{\e}(0,\theta)+\sum_{i=1}^{2k-1}\frac{\partial^i G_{\e}}{\partial x^i}(0,\theta)x^i+\mathcal{O}(x^{2k}),\end{equation} 
and
\begin{equation}\label{eq14} G_{\e}(0,\theta)=G_{0}(0,\theta)+\e\frac{\partial }{\partial \e}G_{\e}(0,\theta)\Big|_{\e=0}+\mathcal{O}(\e^2).\end{equation}
Thus, substituting \eqref{eq14} into \eqref{eq13} and taking $x=x_\e,$ we have 
\begin{equation} \label{eq17} G_{\e}(x_{\e},\theta)=G_{0}(0,\theta)+\e\frac{\partial}{\partial \e}G_{\e}(0,\theta)\Big|_{\e=0}+\sum_{i=1}^{2k-1}\frac{\partial^i G_{0}}{\partial x^i}(0,\theta)x_{\e}^i+\mathcal{O}(\e^2)+\mathcal{O}(\e x_{\e})+\mathcal{O}(x_{\e}^{2k}).\end{equation}

Now, in order to estimate $G_{0}(0,\theta)$ and $\frac{\partial}{\partial \e}G_{\e}(0,\theta)\Big|_{\e=0}$ in \eqref{eq17}, we compute
\begin{equation}\label{eq15}
G_{\e}(0,\theta)  =  G_{\e}(0,0)+\theta\frac{\partial G_{\e}}{\partial \theta}(0,0)+\mathcal{O}(\theta^2) = \displaystyle  \theta\frac{\partial G_{\e}}{\partial \theta}(0,0)+\mathcal{O}(\theta^2).
\end{equation} 
We know that
$$G_0(0,\theta)=\int_{0}^{\theta}{[g(s)+y_{0}(s)\vartheta(s,y_{0}(s))]ds},$$ where $y_{0}$ satisfies the following Cauchy problem
	$$\left\lbrace\begin{array}{lll}\label{eq0}
y'_{0} &= & \alpha t^{2k-1}+g(t)+y_{0}\vartheta(t,y_{0}),\\
y_{0}(0) & =  &\displaystyle  0.\\
\end{array}\right.$$ 
Notice that $y_{0}^{(i)}(0)=0$ for $i=0,1,\ldots,2k-1$ and $y_{0}^{(2k)}(0)=(2k-1)!\alpha.$ Thus, 
\begin{equation}\label{y0def}
y_{0}(t)=\frac{\alpha }{2k}t^{2k}+\mathcal{O}(t^{2k+1})
\end{equation} and
$$\begin{array}{lllll} \displaystyle\frac{\partial G_0}{\partial\theta}(0,\theta)  &= & g(\theta)+y_{0}(\theta)\vartheta(s,y_{0}(\theta))\\
& = &\displaystyle g(\theta)+ \displaystyle\frac{\alpha \theta^{2k}}{2k}\vartheta(s,y_{0}(\theta))+\mathcal{O}(\theta^{2k+1})\\  
& = &\displaystyle \mathcal{O}(\theta^{2k}).\\ 
\end{array}$$ 
Hence, we conclude that 
\begin{equation}\label{eqG0}
G_{0}(0,\theta)=\mathcal{O}(\theta^{2k+1}).
\end{equation}
 Analogously, 
$$G_\e(0,\theta)=\int_{0}^{\theta}{[g(s)+y(s,\e)\vartheta(s,y(s,\e))]ds}$$ and, then,  $\frac{\partial G_{\e}}{\partial\theta}(0,0)=\e \vartheta(0,\e).$ Therefore, by \eqref{eq15}, $G_{\e}(0,\theta)=\theta\e \vartheta(0,\e)+\mathcal{O}(\theta^2).$ Hence, \begin{equation}\label{eq18}\frac{\partial G_{\e}}{\partial\e}(0,\theta)\Big|_{\e=0}=\mathcal{O}(\theta).\end{equation}

Finally, in order to estimate the remainder terms in \eqref{eq17}, we compute
\begin{equation}\label{eq16} G_{0}(x,\theta)=G_{0}(x,0)+\theta\frac{\partial G_{0}}{\partial \theta}(x,0)+\ldots+\theta^{2k-1}\frac{\partial^{2k-1} G_{0}}{\partial \theta^{2k-1}}(x,0)+\mathcal{O}(\theta^{2k}).\end{equation} 
Using the definition of $G_0(x,\T)$ and \eqref{y0def}, we get that 
\begin{equation}\label{partialG0}
\frac{\partial^i}{\partial x^i}\frac{\partial^j }{\partial \theta^j}G_{0}(0,0)=0,
\end{equation}
for all $j\in\{0,\ldots,2k-1\}$ and $i\in\{1 ,\ldots,2k-j\}.$ So, by \eqref{eq16} and \eqref{partialG0}, we obtain that 
\begin{equation}\label{eq19}\frac{\partial^i G_{0}}{\partial x^i}(0,\theta)=\mathcal{O}(\theta^{2k+1-i}),
\end{equation} for all $i\in\{1,\ldots,2k-1\}.$

Substituting  \eqref{eqG0}, \eqref{eq18}, and \eqref{eq19} into \eqref{eq17}, we get 
$$\begin{array}{rl} G_{\e}(x_{\e},\theta)  = & G_{0}(0,\theta)+\mathcal{O}(\e\theta)+\mathcal{O}(\e^2)+\displaystyle\sum_{i=1}^{2k-1}\mathcal{O}(\theta^{2k+1-i} x_\e^i)
 +\mathcal{O}(\e x_{\e})+\mathcal{O}(x_{\e}^{2k})\\
= &   \displaystyle \mathcal{O}(\theta^{2k+1})+\mathcal{O}(\e\theta)+\sum_{i=1}^{2k-1}\mathcal{O}(\theta^{2k+1-i} x_\e^i)+\mathcal{O}(x_{\e}^{2k}).\\
\end{array}$$
Consequently, 
$$y_{\theta}^\e=\displaystyle\frac{\alpha\theta^{2k}}{2k}-\frac{\alpha x_{\e}^{2k}}{2k}+\e+ \mathcal{O}(\theta^{2k+1})+\mathcal{O}(\e\theta)+\sum_{i=1}^{2k-1}\mathcal{O}(\theta^{2k+1-i} x_\e^i)+\mathcal{O}(x_{\e}^{2k}).$$
Therefore, by Lemma \ref{y0} we can conclude that $y_{\theta}^0=y(\T)=\ov{y}_\T,$ i.e.
$$y_{\theta}^\e=\displaystyle\overline{y}_{\theta}+\e+\mathcal{O}(\e\theta)+\sum_{i=1}^{2k-1}\mathcal{O}(\theta^{2k+1-i} x_\e^i)+\mathcal{O}(x_{\e}^{2k}).$$
\end{proof} 

Finally, from Proposition \ref{theorem1} and arguing analogously to the construction of map $P^u$ (see Section \ref{P}), we may define the map 
$$\begin{array}{rcl} R^u: \overleftarrow{H}_{\e}&\longrightarrow & \widetilde V_{\theta}^{\e}\\ 
(x,\e)&\longmapsto & \Big(\theta,y^\e_\theta+\mathcal{O}(e^{-c/\e^q})\Big),\\ 

\end{array}$$ where $\overleftarrow{H}_{\e}=[x_\e-re^{-\frac{c}{\e^q}},x_\e]\times\{\e\}$ and $\widetilde V_{\theta}^{\e}=\{\theta\}\times[y^\e_\theta,y^\e_{\theta}+re^{-\frac{c}{\e^q}}],$ for all $\theta\in[x_\e,\T_0]$ and $\e>0$ small enough.

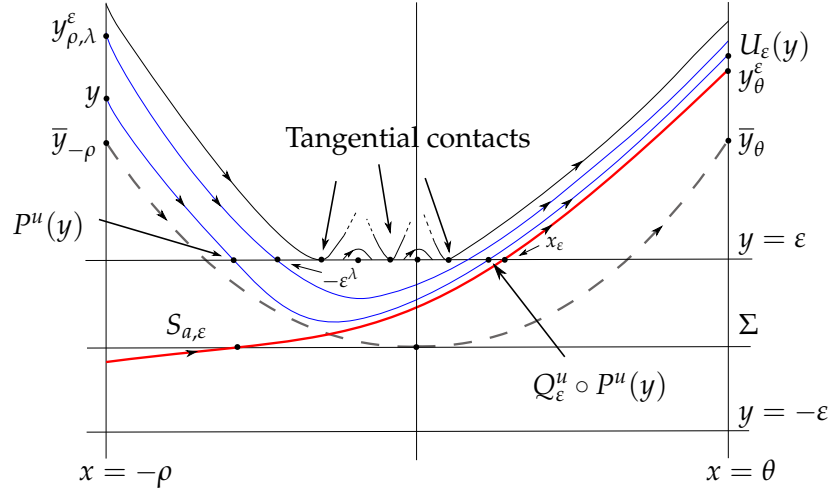
\begin{figure}[h]
	\begin{center}
		\begin{overpic}[scale=0.37]{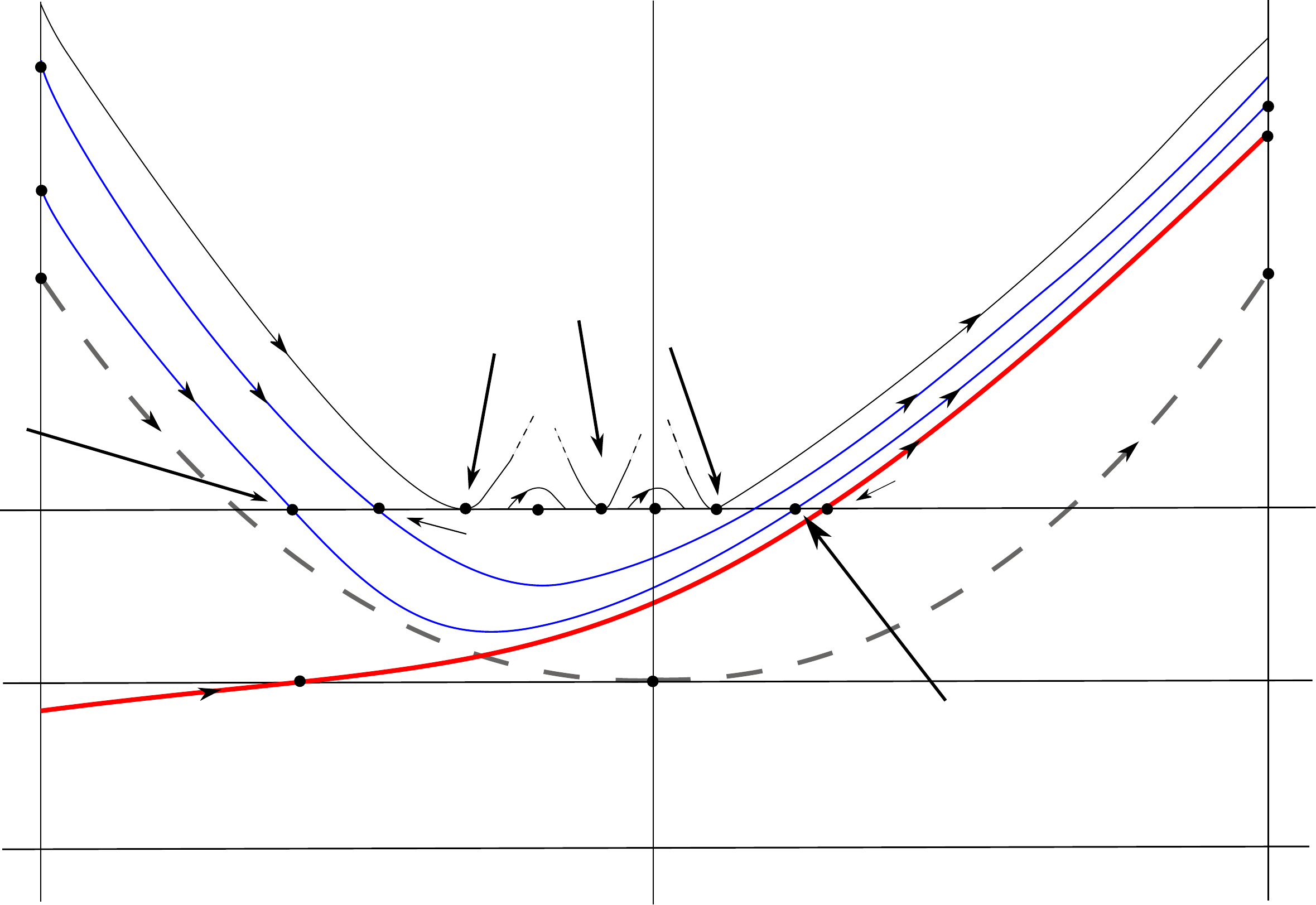}

		\put(-5,47){$\ov{y}_{-\rho}$}
		\put(98,47){$\ov{y}_{\theta}$}
		\put(98,61){$U_{\e}(y)$}
		\put(-5,64){$y^\e_{\rho,\la}$}
		\put(-11,34){$P^u(y)$}
		\put(12,19){$S_{a,\e}$}
		\put(69,32){\scriptsize $x_\e$ \par}
		\put(35.5,26){\scriptsize $-\e^\la$ \par}
		\put(30,47){Tangential contacts}
		\put(98,56.5){$y^\e_{\theta}$}
		\put(0,54){$y$}
		\put(67,10){$Q^u_\e\circ P^u(y)$}
		\put(93,-3){$x=\theta$}
		\put(-1,-3){$x=-\rho$}
		\put(98,19){$\Sigma$}
		\put(98,32){$y=\e$}
		\put(98,6){$y=-\e$}

		\end{overpic}
		\bigskip
	\end{center}
	\caption{The map $U_{\e}=R^u\circ Q^u_\e\circ P^u$ for the regularized system $Z^\Phi_\e.$ The dotted curve is the trajectory of $X^+$ passing through the visible $2k$-multiplicity contact with $\Sigma$ with $(0,0).$ One can see the exponential attraction of the Fenichel manifold $S_{a,\e}.$}
	\label{figMAP}
	\end{figure}

\subsection{Proof of Theorem \ref{ta}}
Consider a Filippov system $Z=(X^+,X^-)_{\Sigma}$ satisfying hypothesis {\bf (A)} for some $k\geqslant 1.$ For $n\geqslant 2k-1,$ let $\Phi\in C^{n-1}_{ST}$ be given as \eqref{Phi} and consider the regularized system $Z_{\e}^{\Phi}$ \eqref{regula}. As noted in Remark \ref{assumption}, we shall assume that $n\geqslant \max\{2,2k-1\}.$ 

From the comments of Section \ref{sec:canprel}, we can assume that $Z\big|_{U}$ can be written as \eqref{Xnf}, which has its regularization given by \eqref{regsys}. Thus, 
statement (a) of Theorem \ref{ta} follows from Proposition \ref{trans}. Finally, statement (b) follows by taking the composition
\[
\begin{array}{cccl}
U_{\e}:& \widehat V_{\rho,\la}^{\e}& \longrightarrow& \widetilde V_{\T}^{\e}\\
&(-\rho,y)&\longmapsto&R^u\circ Q^u_\e\circ P^u(-\rho,y),
\end{array}
\]
where $P^u,$ $Q^u_{\e},$ and $R^u$ are defined in Sections \ref{P}, \ref{expat}, and \ref{R}, respectively (see Figure \ref{figMAP}). Indeed, 
 the existence of  $\rho_0$ and $\T_0>0$ are guaranteed by the construction of the map $P^u$ (see Section \ref{P}) and Proposition \ref{theorem1}, respectively. The existence of constants $c,r,q>0,$ for which $U_{\e}(-\rho,y)=y_{\T}^{\e}+\CO(e^{-\frac{c}{\e^q}})$ is guaranteed by the construction of the map $Q_\e^u$ (see Section \ref{expat}).
Furthermore, by construction of the maps $P^u,$ $Q^u_\e$ and $R^u,$ we have that the trajectories of  $Z_{\e}^{\Phi}$ starting at the section $\widehat V_{\rho,\la}^{\e}$ intersect the line $y=\e$ only in two points before reaching the section $\widetilde V_{\T,\la}^{\e}.$ Moreover, these intersections take place at $\widehat H_{\rho,\la}^{\e}\cup \overleftarrow{H}_{\e}.$

\section{Lower Transition Map}\label{sec:lflightmap}

This section is devoted to prove Theorem \ref{tb}. Analogously to the previous section, we need to guarantee that under some conditions the flow of the regularized system $Z_{\e}^{\Phi}$ near a visible regular-tangential singularity defines a map between two sections, in this case, a horizontal section and a vertical section. Again, it will be convenient to write this map as the composition of three maps, namely $P^l,$ $Q^l_\e$ and $R^l.$ The maps $P^l$ and $Q^l_\e$ will be defined through the flow of  $Z_{\e}^{\Phi}$ restricted to the band of regularization, and the map $R^l$ will be given by the flow of  $Z_{\e}^{\Phi}$ defined outside the  band of regularization.
In what follows, we shall define the maps $P^l,Q^l_\e$ and $R^l$ (see Figure \ref{figMAP22}).

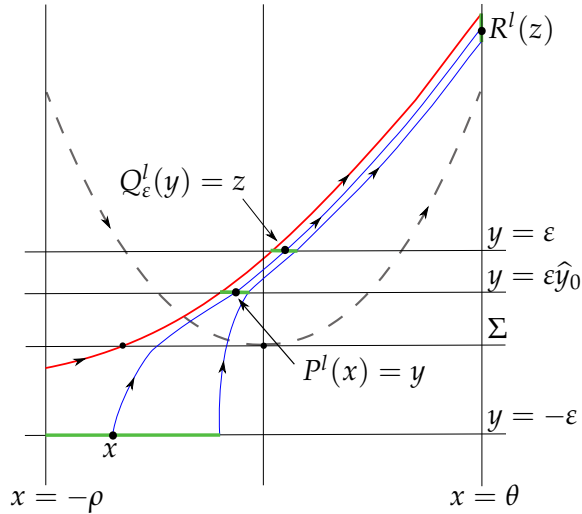
\begin{figure}[h]
	\begin{center}
		\begin{overpic}[scale=0.4]{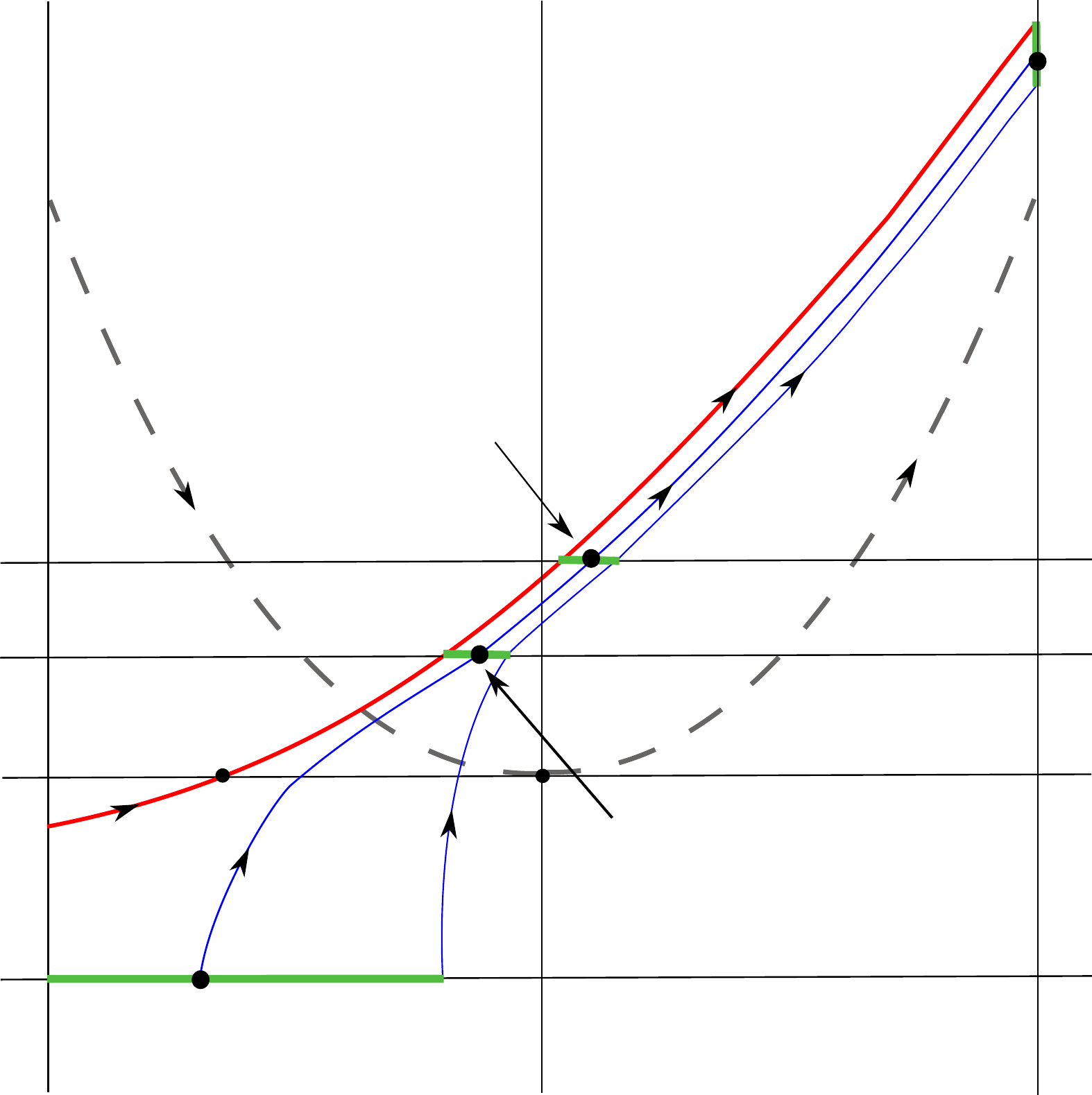}

	  \put(16,6){$x$}
		\put(96,93){$R^l(z)$}
		\put(57.5,22){$P^l(x)=y$}
		\put(19.5,61.5){$Q_\e^l(y)=z$}
		\put(88,-4){$x=\theta$}
		\put(-3,-4){$x=-\rho$}
		\put(96,30.7){$\Sigma$}
		\put(96,50.9){$y=\e$}
		\put(96,12.9){$y=-\e$}
		\put(96,41.9){$y=\e\widehat{y}_0$}
				
		\end{overpic}
	\end{center}
	
	\bigskip
	
	\caption{Dynamics of the maps $P^l,$ $Q_\e^l$ and $R^l.$ The dotted curve is the trajectory of $X^+$ passing through the visible $2k$-multiplicity contact with $\Sigma$ with $(0,0).$}
	\label{figMAP22}
	\end{figure}

First of all, the next result is obtained following the same argument than Proposition \ref{trans}.
\begin{proposition}\label{transL}
Consider the Filippov system $Z=(X^+,X^-)_{\Sigma}$ given by \eqref{Xnf}, for some $k\geqslant 1,$ and  $y_\T^\e$ given in \eqref{ye}. 
 For $n\geqslant\max\{2,2k-1\},$ let $\Phi\in C_{ST}^{n-1}$  be given as \eqref{Phi} and consider the regularized system $Z_{\e}^{\Phi}$  \eqref{regsys}. Then, there exist $\rho_0,\T_0>0,$ such that the vertical segment
\[
\widecheck V_{\T}^{\e}=\{\T\}\times[y_\T^\e-r e^{-\frac{c}{\e^q}},y_\T^\e]
\]
and the horizontal segments 
\[
\widecheck H_{\rho,\la}^{\e}=[-\rho,-\e^{\la}]\times\{-\e\}\,\,\text{ and }\,\, \overrightarrow{H}_{\e}=[x_{\e},x_{\e}+r e^{-\frac{c}{\e^q}}]\times\{\e\}
\]
are transversal sections for every $\rho\in(\e^\la,\rho_0],$ $\T\in[x_\e,\T_0],$ $\la\in(0,\la^*),$ with $\la^*= \frac{n}{2k(n-1)+1},$ constants $c,r,q>0,$ and $\e>0$ sufficiently small.
\end{proposition}

As before, statement $(i)$ of Theorem \ref{tb} will follows from Proposition \ref{transL}.

\subsection{Construction of the map $P^l$}\label{P^l}
First, we shall see that the forward trajectory of $\ov{Z}_{\e}^\Phi$ \eqref{fastsystem} starting at $(-\e^{\lambda},-1)$ reaches the straight line $\{\hat{y}=\hat{y}_0\},$ with $\hat{y}_0\in(1-\eta,1),$ for some $\eta>0$ small enough. After that, the map will be obtained through Poincar\'{e}-Bendixson argument.

Accordingly, consider a function $\widetilde\mu:I_{(x,\widehat{y})}\times \widehat{U}\times[0,\e_0]\rightarrow\mathbb{R}$ given by 
$$\widetilde\mu(\tau,x,-1,\e)=\varphi^{2}_{\ov{Z}_{\e}^\Phi}(\tau,x,-1)-\widehat{y}_0,$$ where $\varphi_{\ov{Z}_{\e}^\Phi}=(\varphi_{\ov{Z}_{\e}^\Phi}^1,\varphi_{\ov{Z}_{\e}^\Phi}^2)$ denotes the flow of $\ov{Z}_{\e}^\Phi,$ $I_{(x,\widehat{y})}$ is the maximal interval of definition of $\tau\mapsto\varphi_{\ov{Z}_{\e}^\Phi}(\tau,x,\hat y),$ $\e_0>0$ is sufficiently small, and $\widehat{U}$ is the domain of the vector field $Z$ in the $(x,\widehat{y})$-coordinates.

Now, for each $\widehat{y}\in[-1,\widehat{y}_0]$ and $\e=0,$ we have 
$$\varphi_{\ov{Z}_{0}^\Phi}(0,0,\widehat{y})=(0,\widehat{y}) \,\text{and}\,
\dfrac{\partial\varphi^2_{\ov{Z}_{0}^\Phi}}{\partial \tau}(0,0,\widehat{y})=\dfrac{1-\Phi(\widehat{y})}{2}>0.$$
Then, there exists $\tau_0>0$ such that $\varphi_{\ov{Z}_{0}^\Phi}(\tau_0,0,-1)=(0,\widehat{y}_0) .$ In this way, 
$$\widetilde\mu(\tau_0,0,-1,0)=0 \,\text{and}\, 
\dfrac{\partial\widetilde\mu}{\partial \tau}(\tau_0,0,-1,0)=\dfrac{1-\Phi(\widehat{y}_0)}{2}\neq 0.$$
Thus, from \textit{Implicit Function Theorem} there exists a unique smooth function $\tau(x,\e),$ such that, $\varphi^{2}_{\ov{Z}_{\e}^\Phi}(\tau(x,\e),x,-1)=\widehat{y}_0$ and $\tau(0,0)=\tau_0.$ Therefore, for $\e>0$ sufficiently small, the forward trajectory of $\ov{Z}_{\e}^\Phi$ starting at $(-\e^{\lambda},-1)$ reaches the straight line $\{\hat{y}=\hat{y}_0\}$ at $$\Big(\varphi^{1}_{\ov{Z}_{\e}^\Phi}(\tau(-\e^{\lambda},-1),-\e^{\lambda},-1),\hat{y}_0\Big).$$

In what follows we shall compute the Taylor expansion of $\varphi^{1}_{\ov{Z}_{\e}^\Phi}(\tau(x,\e),x,-1)$ around $(x,\e)=(0,0).$   Notice that 

\begin{equation}\label{phi2}\begin{array}{rl} \varphi^{1}_{\ov{Z}_{\e}^\Phi}(\tau(x,\e),x,-1)  = & \varphi^{1}_{\ov{Z}_{0}^\Phi}(\tau(x,0),x,-1)+\mathcal{O}(\e)\\
= & \varphi^{1}_{\ov{Z}_{0}^\Phi}(\tau(0,0),0,-1)+x\dfrac{\partial}{\partial x}\left(\varphi^{1}_{\ov{Z}_{0}^\Phi}(\tau(x,0),x,-1)\right)\Big|_{x=0}\vspace{0.2cm}\\ 
& + \mathcal{O}(x^2)+\mathcal{O}(\e)\\
= &  \varphi^{1}_{\ov{Z}_{0}^\Phi}(\tau_0,0,-1)+x\left[\dfrac{\partial\varphi^{1}_{\ov{Z}_{0}^\Phi}}{\partial \tau}(\tau(x,0),x,-1)\dfrac{\partial \tau}{\partial x}(x,0)\right.\\
&+ \left.\dfrac{\partial \varphi^{1}_{\ov{Z}_{0}^\Phi}}{\partial x}(\tau(x,0),x,-1)\right]\Big|_{x=0}+\mathcal{O}(x^2)+\mathcal{O}(\e).\\
\end{array}\end{equation} 
Substituting
$$\varphi^{1}_{\ov{Z}_{\e}^\Phi}(\tau_0,0,-1)=0\quad\text{and}\quad
\dfrac{\partial\varphi^{1}_{\ov{Z}_{0}^\Phi}}{\partial \tau}(\tau_0,0,-1)=0$$
into \eqref{phi2}, we have
\begin{equation}\label{phi2222}\begin{array}{rl} \varphi^{1}_{\ov{Z}_{\e}^\Phi}(\tau(x,\e),x,-1)
= &  x\left[\dfrac{\partial\varphi^{1}_{\ov{Z}_{0}^\Phi}}{\partial \tau}(\tau_0,0,-1)\dfrac{\partial \tau}{\partial x}(0,0)+\dfrac{\partial\varphi^{1}_{\ov{Z}_{0}^\Phi}}{\partial x}(\tau_0,0,-1)\right]\\
& +\mathcal{O}(x^2)+\mathcal{O}(\e)\\
= &  x\dfrac{\partial\varphi^{1}_{\ov{Z}_{0}^\Phi}}{\partial x}(\tau_0,0,-1)+\mathcal{O}(x^2)+\mathcal{O}(\e).\\
\end{array}\end{equation} 
Now, notice that $\frac{\partial\varphi_{\ov{Z}_{0}^\Phi}}{\partial x}(\tau,0,-1)$  is solution of the differential equation $$u'=D\ov{Z}_0^\Phi(0,\varphi^2_{\ov{Z}_{0}^\Phi}(\tau,0,-1))u,$$ with 
$$D\ov{Z}_0^\Phi(0,\varphi^2_{\ov{Z}_{0}^\Phi}(\tau,0,-1))=\begin{bmatrix}
0 & 0 \\
* & -\dfrac{\Phi'\left(\varphi^2_{\ov{Z}_{0}^\Phi}(\tau,0,-1)\right)}{2}\\
\end{bmatrix}.$$ Consequently,  $$\begin{array}{lllll}\begin{bmatrix}
u_1'(\tau)  \\
 u_2'(\tau)\\
\end{bmatrix}&= &\begin{bmatrix}
0 & 0 \\
\ast & -\dfrac{\Phi'\left(\varphi^2_{\ov{Z}_{0}^\Phi}(\tau,0,-1)\right)}{2}\\
\end{bmatrix}\begin{bmatrix}
u_1(\tau)  \\
 u_2(\tau)\\
\end{bmatrix}\vspace{0.2cm}\\&= &\begin{bmatrix}
0  \\
 **\\
\end{bmatrix},\end{array}$$ which implies that $u_1(\tau)$ is constant. Since
$$\frac{\partial\varphi^1_{\ov{Z}_{0}^\Phi}}{\partial x}(\tau_0,0,-1)=\frac{\partial\varphi^1_{\ov{Z}_{0}^\Phi}}{\partial x}(0,0,-1)=1,$$
we conclude, by \eqref{phi2222},  that 
$$\varphi^{1}_{\ov{Z}_{\e}^\Phi}(\tau(x,\e),x,-1)=x+\mathcal{O}(x^2)+\mathcal{O}(\e).$$

Taking $x=-\e^\la,$ we get
$$\varphi^{1}_{\ov{Z}_{\e}^\Phi}(\tau(-\e^\la,\e),-\e^\la,-1)=-\e^\la+\mathcal{O}(\e^{2\la})+\mathcal{O}(\e)\eqdef x^\e_\la.$$

Finally, consider the region $\mathcal{K}$ delimited by the curves $y=-\e,$ $y=\e\widehat{y}_0,$ $y=m(x,\e),$ $y=-\frac{x}{\e}-(\frac{\rho}{\e}+\e)$ and the arc-orbit connecting $(-\e^{\la},-\e)$ and $(x_{\la}^{\e},\e \,\hat y_0).$ Since $\ov{Z}_{\e}^\Phi$ has no singularities inside $\mathcal{K},$ one can easily see that the forward trajectory of $\ov{Z}_{\e}^\Phi$ starting at any point of the transversal section $ \widecheck H_{\rho,\la}^{\e}$ must leave $\mathcal{K}$ through the transversal section $\{(x,y)\in U:y=\e\widehat{y}_0\}.$ This naturally defines a map 
$$ P^l: \widecheck H_{\rho,\la}^{\e} \longrightarrow \{(x,y)\in U:y=\e\widehat{y}_0\}.$$ 

\subsection{Exponential attraction and construction of the map $Q^l_{\e}$}\label{expata}

As we saw in Section \ref{expat}, for $L,N>0$ and $\e_0>0$ small enough, the Fenichel manifold $S_{a,\e}$ is described as $$m(x,\e)=m_0(x)+\e m_1(x)+\mathcal{O}(\e^2),$$ for $-L\leqslant x\leqslant-N$ and $0\leqslant  \e \leqslant \e_0,$ where $m_0$ and $m_1$ were defined in \eqref{eqm0} and \eqref{eq21}.

Now, we shall compute the intersection of $m(x,\e)$ with the straight line $\{\hat{y}=\hat{y}_0\}$ with $1-\eta<\widehat{y}_0< 1,$ for some $\eta>0$ small enough. Indeed, since $m_0(0)=1$ and $$\lim_{x\to-\infty}m_0(x)=-1,$$ then there exists a negative number $\widehat{x}_0$ such that $m_0(\widehat{x}_0)=\widehat{y}_0.$  Moreover, $\widehat{x}_0$ is close to zero because $\widehat{y}_0$ is near to $1$ and $m_0(0)=1.$
After that, consider the function $$\widehat{\mu}(x,\e)=m(x,\e)-\widehat{y}_0,$$ and notice that $\widehat{\mu}(\widehat{x}_0,0)=m_0(\widehat{x}_0)-\widehat{y}_0=0$ and $\dfrac{\partial \hat{\mu}}{\partial x}(\widehat{x}_0,0)=\dfrac{\partial m}{\partial x}(\widehat{x}_0,0)=m'_0(\widehat{x}_0)\neq 0,$ where we have used equation \eqref{eq23}. Thus, there exists a smooth function $\widehat{x}(\e),$ such that $\widehat{x}(0)=\widehat{x}_0$ and $m(\widehat{x}(\e),\e)=\widehat{y}_0.$ Accordingly, from \eqref{eq21}, we have
$$\widehat{x}'(0)=-\dfrac{\dfrac{\partial m}{\partial \e}(\widehat{x}_0,0)}{\dfrac{\partial m}{\partial x}(\widehat{x}_0,0)}=-\dfrac{m_1(\widehat{x}_0)}{m_0'(\widehat{x}_0)}=\frac{m_0'(\widehat{x}_0)-m_0(\widehat{x}_0)\vartheta(\widehat{x}_0,0)}{\alpha(2k-1)\widehat{x}_0^{2k-2}+g'(\widehat{x}_0)}.$$
The last expression is positive, because $m_0'(x)\to\infty$ when $x\to 0$ and $m_0(x)\vartheta(x,0)$ is bounded in the interval $[-L,0],$ with $L$ sufficiently small.
Therefore,  the Taylor expansion of $\widehat{x}(\e)$ around $\e=0$ writes $$\widehat{x}(\e)=\widehat{x}_0+\e\widehat{x}'(0)+\mathcal{O}(\e^2)$$ and, consequently, $\widehat{x}_0<\widehat{x}(\e)<0$ for $\e$ sufficiently small.

\begin{proposition}\label{propexp1}
Fix $0<\la<\la^*=\frac{n}{2k(n-1)+1}.$ Let $x_0\in[\widehat{x}(\e), -\kappa\e^\la],$ with $0<\kappa<1,$ and consider the solution $\widehat y(x,\e)$ of system \eqref{eq29} satisfying $\widehat{y}(x_{0},\e)=\hat{y}_0.$ Then, there exist positive numbers $C$ and $\widetilde{r}$ such that
\begin{equation*}\label{diffym1}
|m(x,\e)-\widehat y(x,\e)| \leqslant  \widetilde{r} e^{-\frac{C}{\e}\left(|x_0|^{\frac{1}{\la^*}}-|x|^{\frac{1}{\la^*}}\right)},\\ 
\end{equation*}
for $x_0\leqslant x\leqslant-\e^{\la^*}.$

\end{proposition}
\begin{proof}
Performing the change of variables $\omega=m(x,\e)-\widehat{y}$ in equation \eqref{eq29}, we have
\begin{equation}\label{eq311}\e\frac{d\omega}{dx}=\xi(x,\e)\phi'(m(x,\e))\omega+\xi(x,\e)F(x,\omega,\e),\end{equation}
where
$$F(x,\omega,\e)=\phi(m(x,\e)-\omega)-\phi(m(x,\e))-\phi'(m(x,\e))\omega$$ and
$$\begin{array}{rl} \displaystyle\xi(x,\e)
=&  \dfrac{2}{\Big(1+\phi(m(x,\e))\Big)\Big(1+\phi(m(x,\e)-\omega(x,\e))\Big)}\\ 
& +  \dfrac{\e\Big(m(x,\e)\vartheta(x,\e m(x,\e))-(m(x,\e)-\omega(x,\e))\vartheta(x,\e(\omega(x,\e)-m(x,\e)))\Big)}{\phi(\omega(x,\e)-m(x,\e))-\phi(m(x,\e))}.\\ 
\end{array}$$
Here, we are denoting $\omega(x,\e)=m(x,\e)-\widehat{y}(x,\e)$ which is the solution of \eqref{eq311} with initial condition $\omega(x_0,\e)=m(x_0,\e)-\hat{y}_0.$

Notice that $F$ writes
 \begin{equation}\label{eq321}
F(x,\omega,\e)=A(x,\e)\omega,
\end{equation}
where 
 $$A(x,\e)=-\int_0^1\phi'(m(x,\e)+(s-1)\omega(x,\e))+\phi'(m(x,\e))ds.$$ Here, as in the proof of Proposition \ref{propexp}, we also claim that $A(x,\e)$ is negative for $-L\leqslant x\leqslant 0$ and $\e$ sufficiently small. Indeed, we know that $\phi'>0$ on the interval $(-1,1).$ In addition, since for $\e>0$ small enough we have $m(x,\e)>\widehat{y}(x,\e)$  and $\dfrac{d\widehat{y}}{dx}(x)>0$ for $x\geqslant x_0$, then the solution $\omega(x,\e)$ satisfies 
$$0\leqslant \omega(x,\e)\leqslant m(x,\e)-\widehat{y}_0.$$ Hence, from Proposition \ref{prop:aux} and \eqref{eq28} we get
 \begin{equation}\label{d11} m(x,\e)+(s-1)\omega(x,\e)\leqslant m(x,\e)\leqslant m_0(x,\e)\leqslant 1-C_1\sqrt[n]{|x|^{2k-1}}\leqslant 1-C_1\sqrt[n]{\epsilon^{\la^*(2k-1)}}<1,\end{equation} and 
\begin{equation}\label{d22}m(x,\e)+(s-1)\omega(x,\e)\geqslant m(x,\epsilon)s-(s-1)\widehat{y}_0\geqslant \widehat{y}_0>1-\eta,\hspace{0.1cm}
\end{equation} for $0\leqslant s\leqslant 1$ and $\eta,\e>0$ small enough. Therefore, we conclude that $A(x,\e)$ is negative.

In this way, by \eqref{eq311} and \eqref{eq321}, we obtain
\begin{equation*}\label{rd11}\begin{array}{rcl} \epsilon\dfrac{d\omega}{dx}&= & \xi(x,\epsilon)(\phi'(m(x,\epsilon))\omega+F(x,\omega,\epsilon))\\
&= & \xi(x,\epsilon)(\phi'(m(x,\epsilon))+A(x,\epsilon))\omega\\
&= & -\xi(x,\epsilon)\Big(\displaystyle\int_0^1\phi'(m(x,\e)+(s-1)\omega(x,\e))ds\Big)\omega,\\ 
\end{array}\end{equation*}
which has its solution with initial condition $\omega(x_0)$ given by
$$\begin{array}{rcl} \omega(x,\e) &= & \omega(x_0)e^{-\frac{1}{\e}\int_{x_0}^x\xi(\nu,\e)(\int_0^1\phi'(m(\nu,\e)+(s-1)\omega(\nu,\e))ds)d\nu}.\\ 
\end{array}$$ Thus,
$$\begin{array}{rcl} |\omega(x,\e)| &= & |\omega(x_0)|e^{-\frac{1}{\e}\int_{x_0}^x\xi(\nu,\e)(\int_0^1\phi'(m(\nu,\e)+(s-1)\omega(\nu,\e))ds)d\nu}.\\ 
\end{array}$$

To conclude this proof, we shall estimate $|\omega(x,\e)|.$ For this, notice that 
$$\displaystyle\xi(x,\e)=\dfrac{2}{\Big(1+\phi(m_0(x))\Big)\Big(1+\phi(m_0(x)-\omega(x,0))\Big)}+\mathcal{O}(\e).$$
 Hence, $L,\e>0$ can be taken sufficiently small in order that $\xi(x,\e)\geqslant l>0,$ for all $-L\leqslant x\leqslant 0.$ Moreover, given $0<\eta<1,$ there exist positive constants $c_1,$ $c_2$ such that for $|\widehat{y}-1|<\eta$ one has  
	$$c_1(1-\widehat{y})^{n-1}\leqslant\phi'(\widehat{y})\leqslant c_2(1-\widehat{y})^{n-1}.$$
	Finally, using \eqref{d11} and \eqref{d22}, we obtain that $|m(\nu,\e)+(s-1)\omega(\nu,\e)-1|<\eta.$ Therefore, for $x\leqslant-e^{\la^*},$ we get that
	
$$\begin{array}{rcl} |\omega(x,\e)| &\leqslant & |\omega(x_0)|e^{-\frac{c_1}{\e}\int_{x_0}^x\xi(\nu,\epsilon)(\int_0^1(1-m(\nu,\e)-(s-1)\omega(\nu,\e))^{n-1}ds)d\nu}\\ 
&\leqslant & |\omega(x_0)|e^{-\frac{l c_1}{\e}\int_{x_0}^x(\int_0^1(1-m(\nu,\e))^{n-1}ds)d\nu}\\ 
&\leqslant & |\omega(x_0)|e^{-\frac{l c_1}{\e}\int_{x_0}^x(1-m(\nu,\e))^{n-1}d\nu}\\
&\leqslant & |\omega(x_0)|e^{-\frac{l c_1}{\e}\int_{x_0}^x(1-m_0(\nu))^{n-1}d\nu}\\ 
&\leqslant & |\omega(x_0)|e^{-\frac{l c_1}{\e}\int_{x_0}^x(C_1|\nu|^{\frac{2k-1}{n}})^{n-1}d\nu}\\ 
&\leqslant & |\omega(x_0)|e^{-\frac{C}{\e}(|x_0|^{\frac{1}{\la^*}}-|x|^{\frac{1}{\la^*}})},\\ 
\end{array}$$
where $C=\frac{n l c_1 C_1^{n-1}}{2k(n-1)+1}$ is a positive constant. The inequality \eqref{eq28} has also been used. 
\end{proof}

\begin{figure}[h]
	\begin{center}
		\begin{overpic}[scale=0.55]{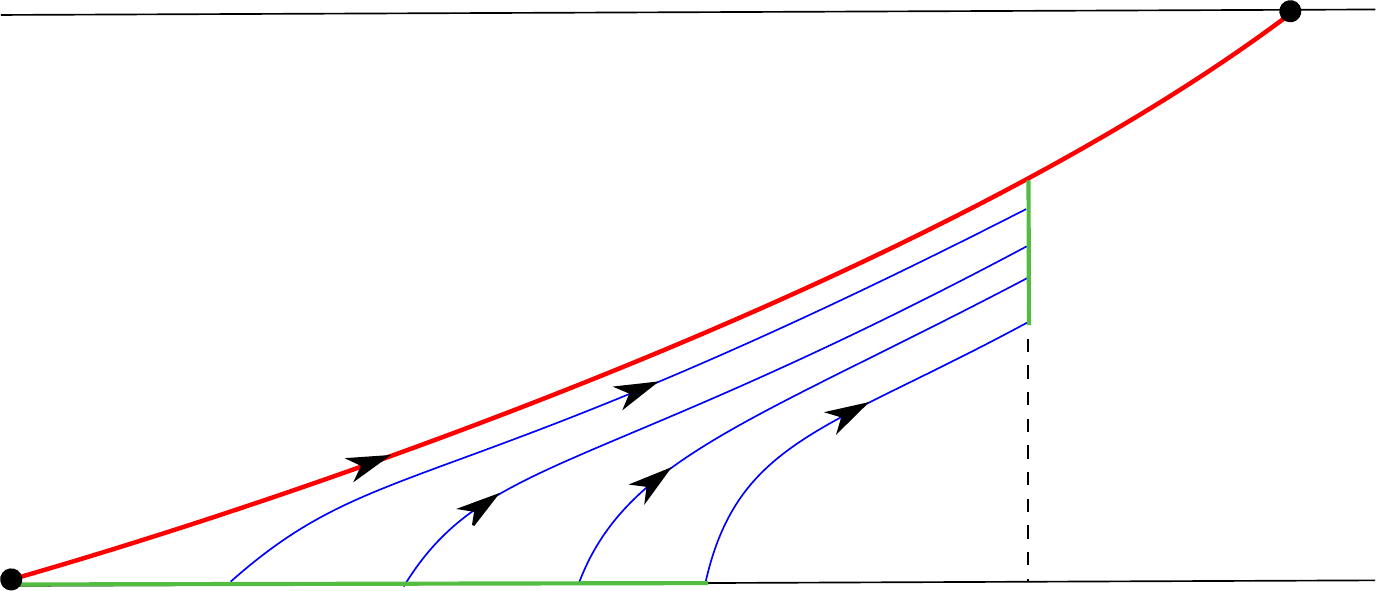}

		\put(12,10){$S_{a,\e}$}
		\put(48,-4){$x_\la^\e$}
		\put(-2,-4.5){$\widehat{x}(\e)$}
		\put(70,-4.5){$-\e^{\la^*}$}
		\put(101,1){$\widehat{y}=\widehat{y}_0$}
		\put(101,41.5){$\widehat{y}=1$}
	  \put(93,44){$x_\e$}
		\end{overpic}
		\bigskip
	\end{center}
	\caption{The exponential attraction of $S_{a,\e}.$}
	\label{figMAPQL}
	\end{figure}

Fix $0<\la<\la^*.$ From Proposition \ref{propexp1}, applied to $x_0=x_\la^\e$ and $x=-\e^{\la^*},$ where $x_0\leqslant-\kappa\e^\la$ for some $\kappa\in(0,1),$  we know that there exist positive numbers $\widetilde{r}$ and $C$ such that
$$ \begin{array}{rcl} |m(-\e^{\la^*},\e)-\widehat y(-\e^{\la^*},\e)| &\leqslant&  \widetilde{r} e^{-\frac{C}{\e}\left(|x_\e^\la|^{\frac{1}{\la^*}}-|-\e^{\la^*}|^{\frac{1}{\la^*}}\right)}\\ 
& \leqslant & r e^{-\frac{c}{\e^{q}}}, \\
\end{array}$$
where $c=C\kappa^\frac{1}{\la^*},$ $r=\widetilde{r}e^C$ and $q=1-\frac{\la}{\la^*}$ are positive constants. Hence, 
$$\widehat{y}(-\e^{\la^*},\e)=m(-\e^{\la^*},\e)+\mathcal{O}(e^{-c/\e^q}).$$   
Thus, arguing analogously to the construction of map $P^l$ (see Section \ref{P^l}), any solution of the system \eqref{fastsystem} with initial condition in the interval $[\widehat{x}(\e),x_\la^\e],$ $\e$ sufficiently small, reaches the section $x=-\e^{\la^*}$ exponentially close to the Fenichel manifold (see Figure \ref{figMAPQL}). 
From Proposition \ref{varpext}, these solutions can be continued until the section $\hat y=1.$ Going back through the rescaling $y=\e \hat y,$ we may define the following map through the flow of \eqref{regsys},
$$\begin{array}{rcl} Q^l_{\e}:  [\widehat{x}(\e),x^\e_\la]\times\{y=\e\widehat{y}_0\}&\longrightarrow & \overrightarrow{H}_{\e}\\ 
(x,\e)&\longmapsto & \Big(x_\e+\mathcal{O}(e^{-c/\e^q}),\e\Big),\\ 

\end{array}$$ where $\overrightarrow{H}_{\e}=[x_\e,x_\e+re^{-\frac{c}{\e^q}}]\times\{\e\},$ for $\e>0$ small enough.

\subsection{Construction of the map $R^l$}\label{R^l}

Finally, from Proposition \ref{theorem1} and arguing analogously to the construction of map $P^l$ (see Section \ref{P^l}),  we may define the map 

$$\begin{array}{rcl} R^l: \overrightarrow{H}_{\e}&\longrightarrow & \widecheck V_{\T}^{\e}\\ 
(x,\e)&\longmapsto & \Big(\theta,y^\e_\theta+\mathcal{O}(e^{-c/\e^q})\Big),\\ 
\end{array}$$
where $\overrightarrow{H}_{\e}=[x_\e,x_\e+re^{-\frac{c}{\e^q}}]\times\{\e\}$ and $\widecheck V_{\T}^{\e}=\{\theta\}\times[y^\e_\theta-re^{-\frac{c}{\e^q}},y^\e_{\theta}],$ for every $\theta\in[x_\e+re^{-\frac{c}{\e^q}},\T_0],$ and $\e>0$ small enough.

\begin{figure}[h]
	\begin{center}
		\begin{overpic}[scale=0.4]{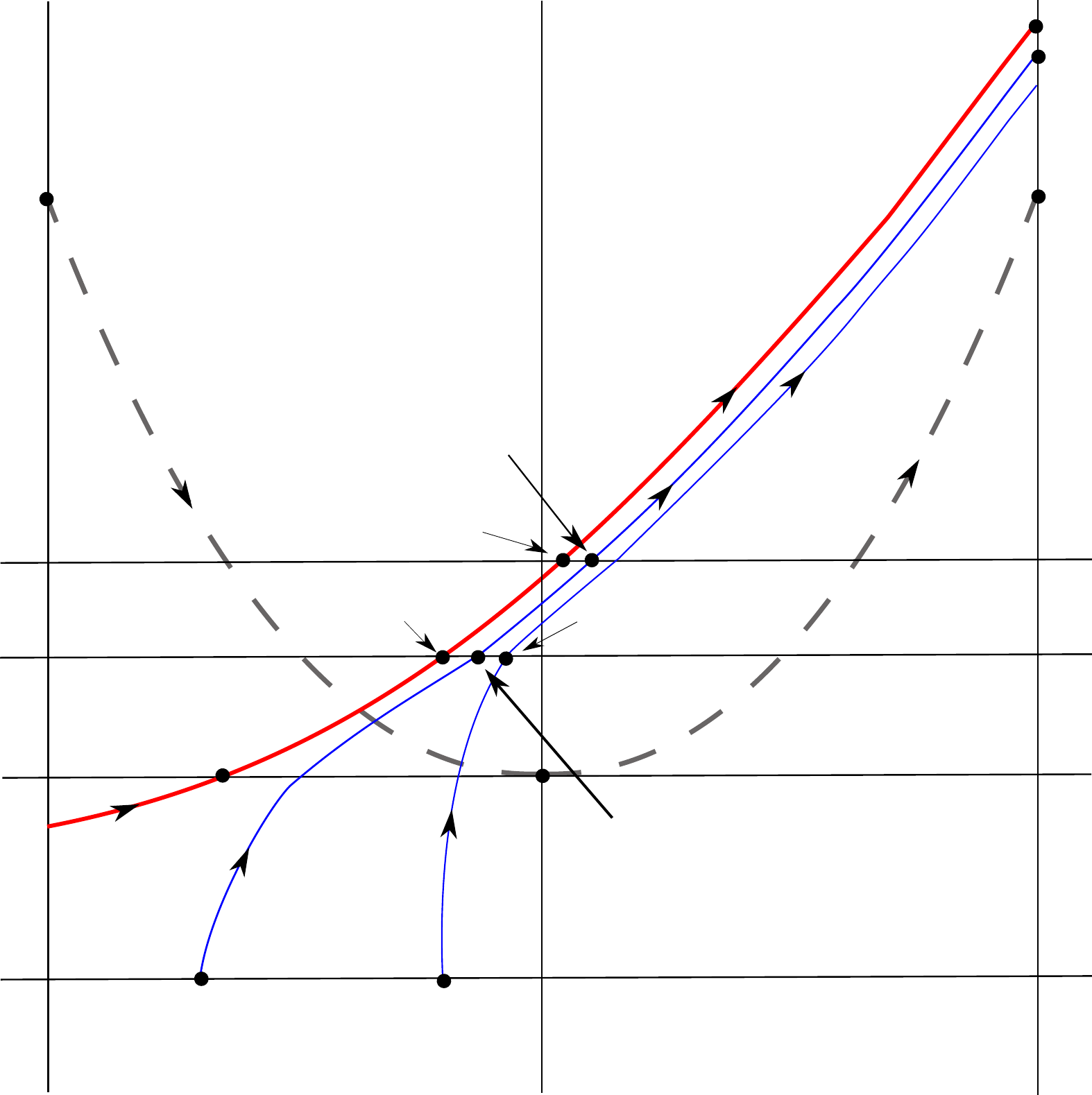}

	  \put(36,4){$-\e^\la$}
	  \put(8,18){$S_{a,\e}$}
		\put(96,99.5){$y^\e_{\theta}$}
		\put(-9,81){$\ov{y}_{-\rho}$}
		\put(96,81){$\ov{y}_{\theta}$}
		\put(37,51){$x_\e$}
		\put(28,44){\scriptsize $\widehat{x}(\e)$ \par}
		\put(54,43){\scriptsize $x^\e_\la$ \par}
	  \put(16,5){$x$}
		\put(96,91){$L_\e(x)=R^l\circ Q_\e^l\circ P^l(x)$}
		\put(56.5,22){$P^l(x)$}
		\put(17,61){$Q_\e^l\circ P^l(x)$}
		\put(88,-4){$x=\theta$}
		\put(-3,-4){$x=-\rho$}
		\put(96,30.2){$\Sigma$}
		\put(96,51){$y=\e$}
		\put(96,13){$y=-\e$}
		\put(96,42){$y=\e\widehat{y}_0$}

		\end{overpic}
	\end{center}
	
	\bigskip
	
	\caption{The map $L_{\e}=R^l\circ Q_\e^l\circ P^l$ for the regularized system $Z^\Phi_\e.$ The dotted curve is the trajectory of $X^+$ passing through the visible $2k$-multiplicity contact with $\Sigma$ with $(0,0).$ One can see the exponential attraction of the Fenichel manifold $S_{a,\e}.$}
	\label{figMAP33}
	\end{figure}
\subsection{Proof of Theorem \ref{tb}}

Consider a Filippov system $Z=(X^+,X^-)_{\Sigma}$ satisfying hypothesis {\bf (A)} for some $k\geqslant 1.$ For $n\geqslant 2k-1,$ let $\Phi\in C^{n-1}_{ST}$ be given as \eqref{Phi} and consider the regularized system $Z_{\e}^{\Phi}$ \eqref{regula}. As noted in Remark \ref{assumption}, we shall assume that $n\geqslant \max\{2,2k-1\}.$ 

From the comments of Section \ref{sec:canprel}, we can assume that $Z\big|_{U}$ can be written as \eqref{Xnf}, which has its regularization given by \eqref{regsys}. Thus, statement (a) of Theorem \ref{tb} follows from Proposition \ref{transL}. Finally, statement (b) follows by taking the composition
\[
\begin{array}{cccl}
L_{\e}:& \widecheck H_{\rho,\la}^{\e}& \longrightarrow& \widecheck V_{\T}^{\e}\\
&(x,-\e)&\longmapsto&R^l\circ Q^l_\e\circ P^l(x,-\e),
\end{array}
\]
where $P^l,$ $Q^l_{\e},$ and $R^l$ are defined in Sections \ref{P^l}, \ref{expata}, and \ref{R^l}, respectively (see Figure \ref{figMAP33}). Indeed, the existence of  $\rho_0$ and $\T_0>0$ are guaranteed by the construction of the map $P^l$ (see Section \ref{P^l}) and Proposition \ref{theorem1}, respectively. The existence of constants $c,r,q>0,$ for which $L_{\e}(x,-\e)=y_{\T}^{\e}+\CO(e^{-\frac{c}{\e^q}})$ is guaranteed by the construction o the map $Q_\e^l$ (see Section \ref{expata}).
 
\section{Regularization of boundary limit cycles}\label{sec:limitcycle}

Assume that the Filippov system $Z=(X^+,X^-)_{\Sigma}$ satisfies hypothesis {\bf (B)}  for some $k\geqslant 1$ (see Section \ref{H}). Therefore, 
from the comments of Section \ref{sec:canprel}, we can assume that, for some neighborhood $U\subset\R^2$ of the origin,  $Z\big|_{U}$ is written as \eqref{Xnf}, which has its regularization given by \eqref{regsys}. 
 
Consider the transversal section $S=\{(x,y)\in U: x=0\}.$ From hypothesis {\bf (B)}, the flow of $Z$ defines a Poincar\'{e} map $\pi:  S' \longrightarrow S$ around the limit cycle $\Gamma.$ Here, $S'\subset S$ is an open set (in the topology induced by $S$) containing $(0,0).$ Accordingly, $\pi(0)=0$ and, since $\Gamma$ is hyperbolic, $\pi'(0)=K\neq0.$ Moreover, taking into account that $X^+$ is planar, the uniqueness of
solutions implies that $K>0.$ 

Denote by $F$ the saturation of $S'$ through the flow of $X^+$ until $S.$ 
For each $\theta>0$ and $\rho>0$ small enough, we know from \eqref{Xnf}  that $\Sigma_{\theta}\defeq \{x=\theta\}\cap F$ and $\Sigma_{-\rho}\defeq \{x=-\rho\}\cap F$ are transversal to $X^+.$ 
Thus, the flow of $X^+$  induces an exterior map $P^{e}:\Sigma_{\theta}\longrightarrow\Sigma_{-\rho},$
which is $C^{2k}$ diffeomorphism.
Accordingly, from Lemma \ref{y0} and hypothesis {\bf (B)}, $P^{e}(\ov{y}_{\theta})=\overline{y}_{-\rho}$ and $K_{\theta,\rho}\defeq \frac{dP^e}{dy}(\ov{y}_{\theta})\neq0.$  Again, taking into account that $X^+$ is planar,  the uniqueness of solutions implies that $K_{\theta,\rho}>0.$ Thus, expanding $P^{e}$ around $y=\ov{y}_\T$, we get 
 \begin{equation}\label{EXTD}
  P^{e}(y)=\ov{y}_{-\rho}+K_{\T,\rho}(y-\ov{y}_\T)+\CO((y-\ov{y}_\T)^2).
 \end{equation}

In order to prove Theorem \ref{tc}, we shall need the following result.

\begin{lemma}\label{lemma6}
$\lim\limits_{\theta,\rho \to 0}K_{\theta,\rho}=K.$
\end{lemma}
\begin{proof}
Notice that, for $\rho>0$ and $\theta>0$ small enough, the flow of $X^+$ induces the following $C^{2k}$ maps,
$$\la_{\theta}:S'\rightarrow\{x=\theta\}\cap F\quad\text{and}\quad \la_{\rho}:\{x=-\rho\}\cap F\rightarrow S\cap F,$$
which satisfies $\la_{\rho}(\overline{y}_{-\rho})=0$ and $\la_{\theta}(0)=\ov{y}_{\theta}.$ Indeed, 
 consider the functions $$\mu_{1}(t,y,\theta)=\varphi^1_{X^+}(t,0,y)-\theta, \quad\text{for} \quad (0,y)\in S',$$ and $$\mu_{2}(t,y,\rho)=\varphi^1_{X^+}(t,-\rho,y),\quad\text{for} \quad (-\rho,y)\in \{x=-\rho\}\cap F.$$  
 Since, $\mu_{1}(0,0,0)=0=\mu_{2}(0,0,0),$
$$\frac{\partial \mu_{1}}{\partial t}(0,0,0)=\frac{\partial \varphi^1_{X^+}}{\partial t}(0,0,0)=1\neq 0, \quad \text{and} \quad \frac{\partial \mu_{2}}{\partial t}(0,0,0)=\frac{\partial \varphi^1_{X^+}}{\partial t}(0,0,0)=1\neq 0,$$ we get, by the \textit{Implicit Function Theorem}, the existence of unique smooth functions $t_{1}(y,\theta)$ and $t_{2}(y,\rho)$ such that 
$t_1(0,0)=0=t_2(0,0),$
$$\mu_1(t_1(y,\theta),y,\theta)=0, \quad \text{and}\quad \mu_2(t_2(y,\rho),y,\rho)=0,$$ i.e. $\varphi^{1}_{X^+}(t_1(y,\theta),0,y)=\theta$ and $\varphi^{1}_{X^+}(t_2(y,\rho),-\rho,y)=0.$ Thus, 
$$\la_{\theta}(y)=\varphi^2_{X^+}(t_{1}(y,\theta),0,y)\quad \text{and}\quad\la_{\rho}(y)=\varphi^2_{X^+}(t_{2}(y,\rho),-\rho,y).$$ 
Notice that
$$\frac{d\la_{\theta}}{dy}(0)=\frac{\partial \varphi^{2}_{X^+}}{\partial t}(t_{1}(0,\theta),0,0)\frac{\partial t_1}{\partial y}(0,\theta)+\frac{\partial \varphi^{2}_{X^+}}{\partial y}(t_{1}(0,\theta),0,0)$$ and 
$$\frac{d\la_{\rho}}{dy}(\ov{y}_{-\rho})=\frac{\partial \varphi^{2}_{X^+}}{\partial t}(t_{2}(\ov{y}_{-\rho},\rho),-\rho,\ov{y}_{-\rho})\frac{\partial t_2}{\partial y}(\ov{y}_{-\rho},\rho)+\frac{\partial \varphi^{2}_{X^+}}{\partial y}(t_{2}(\ov{y}_{-\rho},\rho),-\rho,\ov{y}_{-\rho}).$$
Since
$$\frac{\partial t_1}{\partial y}(0,0)=-\frac{\frac{\partial \varphi^{1}_{X^+}}{\partial y}(0,0,0)}{\frac{\partial \varphi^{1}_{X^+}}{\partial t}(0,0,0)}=-\frac{\partial \varphi^{1}_{X^+}}{\partial y}(0,0,0)=0,$$ 
$$\frac{\partial t_2}{\partial y}(0,0)=-\frac{\frac{\partial \varphi^{1}_{X^+}}{\partial y}(0,0,0)}{\frac{\partial \varphi^{1}_{X^+}}{\partial t}(0,0,0)}=-\frac{\partial \varphi^{1}_{X^+}}{\partial y}(0,0,0)=0,$$
and
$$ \frac{\partial \varphi^{2}_{X^+}}{\partial y}(0,0,0)=1,$$
we get that
\begin{equation}\label{LT}\begin{array}{rcl} \lim\limits_{\theta \to 0}\frac{d\la_{\theta}}{dy}(0)& = &\displaystyle\frac{\partial \varphi^{2}_{X^+}}{\partial t}(t_{1}(0,0),0,0)\frac{\partial t_1}{\partial y}(0,0)+\frac{\partial \varphi^{2}_{X^+}}{\partial y}(t_{1}(0,0),0,0)\\ 
& = &\displaystyle \frac{\partial \varphi^{2}_{X^+}}{\partial t}(0,0,0)\left[-\frac{\partial \varphi^{1}_{X^+}}{\partial y}(0,0,0)\right]+\frac{\partial \varphi^{2}_{X^+}}{\partial y}(0,0,0)\\ 
& = &\displaystyle 1
\end{array}\end{equation} and 
\begin{equation}\label{LR}\begin{array}{rcl} \lim\limits_{\rho \to 0}\frac{d\la_{\rho}}{dy}(\ov{y}_{-\rho}) & =  &\displaystyle\frac{\partial \varphi^{2}_{X^+}}{\partial t}(t_{2}(0,0),0,0)\frac{\partial t_2}{\partial y}(0,0)+\frac{\partial \varphi^{2}_{X^+}}{\partial y}(t_2(0,0),0,0)\\ 
& = &\displaystyle \frac{\partial \varphi^{2}_{X^+}}{\partial t}(0,0,0)\left[-\frac{\partial \varphi^{1}_{X^+}}{\partial y}(0,0,0)\right]+\frac{\partial \varphi^{2}_{X^+}}{\partial y}(0,0,0)\\ 
& = &\displaystyle 1.
\end{array}\end{equation} 

Finally, since $\pi=\la_{\rho}\circ P^e\circ\la_{\theta},$ we conclude that
$$\begin{array}{rcl} \displaystyle\frac{d\pi}{dy}(0)& = &\displaystyle\frac{d\la_{\rho}}{dy}(P^e\circ\la_{\theta}(0))\frac{dP^e}{dy}(\la_{\theta}(0))\frac{d\la_{\theta}}{dy}(0)\\ 
& = &\displaystyle \frac{d\la_{\rho}}{dy}(\overline{y}_{-\rho})K_{\theta,\rho}\frac{d\la_{\theta}}{dy}(0).\\ 
\end{array}$$ Therefore, 
\begin{equation} 
\label{limK}\frac{K}{K_{\theta,\rho}}=\frac{d\la_{\rho}}{dy}(\overline{y}_{-\rho})\frac{d\la_{\theta}}{dy}(0).
\end{equation}
The result follows by taking the limit of  \eqref{limK} and using \eqref{LT} and \eqref{LR}.\end{proof}

\subsection{Proof of Theorem \ref{tc}}
By Theorem \ref{ta} for $\T=x_\e$, there exist $\rho_0>0,$ and constants $\beta<0$ and $c,r,q>0$ such that for every $\rho\in(\e^\la,\rho_0],$ $\la\in(0,\la^*),$ and $\e>0$ small enough, the flow of $Z_{\e}^{\Phi}$ defines a map $U_{\e}$ between the transversal sections $\widehat V_{\rho,\la}^{\e}=\{-\rho\}\times [\e,y_{\rho,\la}^{\e}]\,\,\text{ and }\,\, \widetilde V_{x_\e}^{\e}=\{x_\e\}\times[y_{x_\e}^\e,y_{x_\e}^\e+r e^{-\frac{c}{\e^q}}]$ satisfying
\begin{equation}\label{Ueee}
\begin{array}{cccl}
U_{\e}:& \widehat V_{\rho,\la}^{\e}& \longrightarrow& \widetilde V_{x_\e}^{\e}\\
&y&\longmapsto&y_{x_\e}^{\e}+\CO(e^{-\frac{c}{\e^q}}),
\end{array}
\end{equation}
where $y_{x_\e}^{\e} =\ov y_{x_\e}+\e+\mathcal{O}(\e^{2k\la^*})$ (see Figure \ref{figMAP111}).
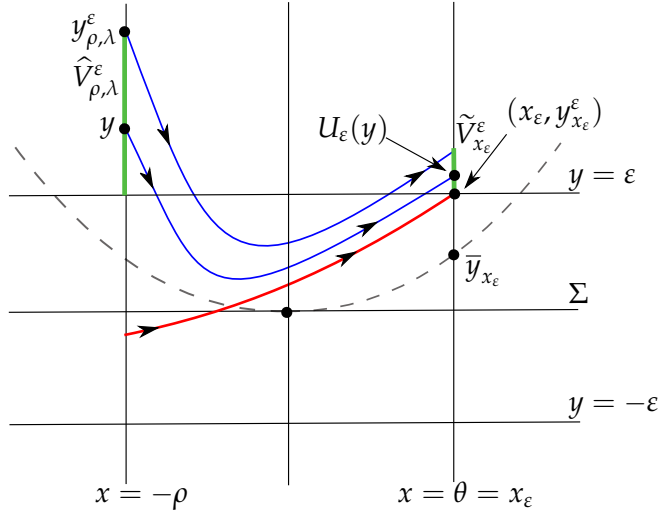
\begin{figure}[h]
	\begin{center}
		\begin{overpic}[scale=0.7]{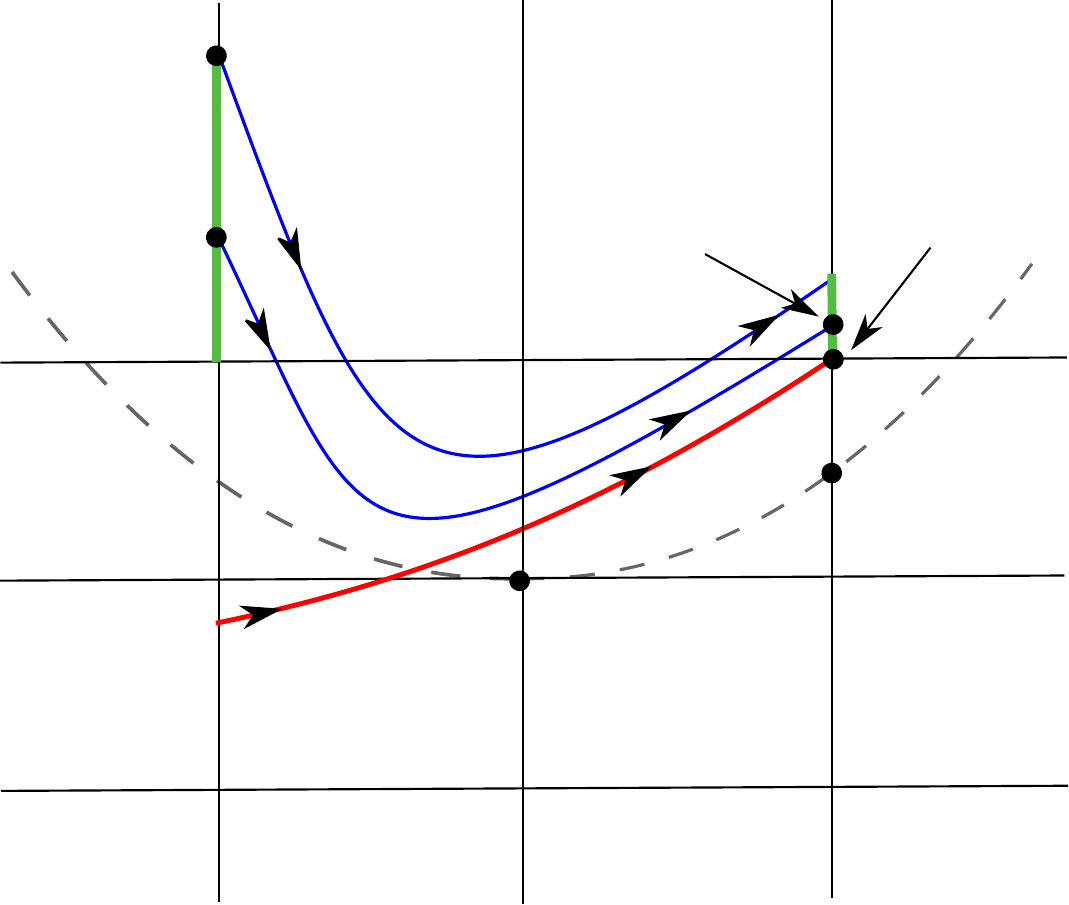}
		\put(11,70){$\widehat V_{\rho,\la}^{\e}$}
		\put(78,60){$\widetilde V_{x_\e}^{\e}$}
		\put(68,-3){$x=\theta=x_\e$}
		\put(15,-3){$x=-\rho$}
		\put(98,32){$\Sigma$}
		\put(98,53){$y=\e$}
		\put(98,13){$y=-\e$}
		\put(87,64){$(x_\e,y_{x_\e}^{\e})$}
		\put(80,37){$\ov{y}_{x_\e}$}
		\put(54,61){$U_{\e}(y)$}
		\put(16,62){$y$}
		\put(11,79){$y^\e_{\rho,\la}$}
		\end{overpic}
	\end{center}	
	\bigskip	
	\caption{Upper Transition Map $U_{\e}$ of the regularized system $Z^\Phi_\e.$The dotted curve is the trajectory of $X^+$ passing through the visible regular-tangential singularity of multiplicity $2k.$ The red curve is the Fenichel manifold.}
	\label{figMAP111}
	\end{figure}
		
		Now, notice that there exists $\e_0>0$ such that $$\widetilde V_{x_\e}^{\e}=\{ x_{\e}\} \times[y_{x_\e}^\e,y_{x_\e}^\e+re^{-\frac{c}{\e^q}}]\subset  \{x=x_\e\}\cap F,$$ for all $\e\in[0,\e_0].$ In this way, we define the function $\pi_\e(y)=P^{e}\circ U_{\e}(y).$ Thus, from \eqref{EXTD} and \eqref{Ueee}, we have
		
		\begin{equation}\label{Pie}
		\begin{array}{lllll} \pi_{\e}(y) 
		& = &\displaystyle P^{e}\Big(\overline{y}_{x_\e}+\e+\mathcal{O}\Big(\e^{2k\la^*}\Big)+\mathcal{O}(e^{-c/\e^q})\Big)\\
		& = &\displaystyle P^{e}\Big(\overline{y}_{x_\e}+\e+\mathcal{O}\Big(\e^{2k\la^*}\Big)\Big)\\
		& = &\displaystyle \overline{y}_{-\rho}+K_{x_\e,\rho}\Big(\e+\mathcal{O}\Big(\e^{2k\la^*}\Big)\Big)+\mathcal{O}\Big(\e+\mathcal{O}\Big(\e^{2k\la^*}\Big)\Big)^2\\
		& = &\displaystyle \overline{y}_{-\rho}+K_{x_\e,\rho}\e+\mathcal{O}\Big(\e^{2k\la^*}\Big).\\
		\end{array}
		\end{equation}
Using \eqref{ey}	and \eqref{Pie}, we get
		 $$\pi_\e(y)-y^\e_{\rho,\la}=(K_{x_\e,\rho}-1)\e+\mathcal{O}(\e\rho)-\beta \e^{2k\la}+\mathcal{O}(\e^{(2k+1)\la})+\mathcal{O}(\e^{1+\la})+\mathcal{O}\Big(\e^{2k\la^*}\Big),$$ where $\beta<0.$ Recall that $0<\la<\la^*.$ Thus, we shall study the limit $\lim\limits_{\rho,\e\rightarrow 0}\frac{\pi_\e(y)-y^\e_{\rho,\la}}{\e}$ in three distinct cases.

		 First, suppose that $\la>\frac{1}{2k}.$ Then,
		$$\frac{\pi_\e(y)-y^\e_{\rho,\la}}{\e}=K_{x_\e,\rho}-1+\mathcal{O}(\rho)+\mathcal{O}(\e^{2k\la-1}).$$ Hence, by Lemma \eqref{lemma6},
		\begin{equation}\label{eql1}	
		\lim\limits_{\rho,\e\rightarrow 0}\frac{\pi_\e(y)-y^\e_{\rho,\la}}{\e}=K-1.\end{equation}
		
		Now, suppose that $\la<\frac{1}{2k}.$ Then,
		$$\frac{\pi_\e(y)-y^\e_{\rho,\la}}{\e^{2k\la}}=(K_{x_\e,\rho}-1)\e^{1-2k\la}+\mathcal{O}(\e^{1-2k\la}\rho)-\beta+\mathcal{O}(\e^{\la})+\mathcal{O}\Big(\e^{2k\la^*-2k\la}\Big).$$ Hence, by Lemma \eqref{lemma6},
		\begin{equation}\label{eql2}	
		\lim\limits_{\rho,\e\rightarrow 0}\frac{\pi_\e(y)-y^\e_{\rho,\la}}{\e^{2k\la}}=-\beta>0.\end{equation}
		
		Finally,  suppose that $\la=\frac{1}{2k}.$ Then, 
		$$\frac{\pi_\e(y)-y^\e_{\rho,\la}}{\e}=K_{x_\e,\rho}-1-\beta+\mathcal{O}(\rho)+\mathcal{O}(\e^{\la})+\mathcal{O}\Big(\e^{2k\la^*-1}\Big).$$ Hence, by Lemma \eqref{lemma6},
		\begin{equation}\label{eql3}	
		\lim\limits_{\rho,\e\rightarrow 0}\frac{\pi_\e(y)-y^\e_{\rho,\la}}{\e}=K-1-\beta.\end{equation}
			 
		Now, we prove statement $(a)$ of Theorem \ref{tc}. Since $\Gamma$ is an unstable hyperbolic limit cycle, we know that $K>1.$ Consequently, all the above limits,\eqref{eql1}, \eqref{eql2} and \eqref{eql3}, are strictly positive and,
		since $\e>0,$ there exists $\delta_0>0$ such that 
		$$0<\rho,\e<\delta_0\hspace{0.1cm}\Rightarrow\hspace{0.1cm}\pi_\e(y)-y^\e_{\rho,\la}>0.$$		
		Hence, $\pi_\e([\e,y^\e_{\rho,\la}])\cap[\e,y^\e_{\rho,\la}]=\emptyset,$ for all $\e>0$ small enough. This means that $\pi_\e$ has no fixed points in $[\e,y^\e_{\rho,\la}]$ and, equivalently, the regularized system $Z_{\e}^{\Phi}$ does not admit limit cycles passing through the section $\widehat H_{\rho,\la}^{\e}.$
		
		 Now, we prove statement $(b)$ of Theorem \ref{tc}. In this case, $\la>\frac{1}{2k}.$  Since $\Gamma$ is an asymptotically stable hyperbolic limit cycle, we know that $K<1.$ Thus, the limit \eqref{eql1} is strictly negative and, since $\e>0,$  there exists $\delta_0>0$ such that 
		$$0<\rho,\e<\delta_0\hspace{0.1cm}\Rightarrow\hspace{0.1cm}\pi_\e(y)-y^\e_{\rho,\la}<0.$$ Hence, $\pi_\e(y)<y^\e_{\rho,\la}.$ Moreover, from \eqref{Pie}, we get
		\begin{equation*}\label{eql4}	
		\lim\limits_{\rho,\e\rightarrow 0}\frac{\pi_\e(y)-\ov{y}_{-\rho}}{\e}=K>0.\end{equation*}
		Since $\e>0,$  there exists $\delta_1>0$ such that 
		$$0<\rho,\e<\delta_1\hspace{0.1cm}\Rightarrow\hspace{0.1cm}\pi_\e(y)-\ov{y}_{-\rho}>0.$$ Hence, $\pi_{\e}(y)>\ov{y}_{-\rho},$ for all $\e>0$ sufficiently small. This means that $\pi_\e([\e,y^\e_{\rho,\la}])\subset[\e,y^\e_{\rho,\la}].$  From the {\it Brouwer Fixed Point Theorem}, we conclude that $\pi_\e$ admits fixed points in $[\e,y^\e_{\rho,\la}]$ and, equivalently, the regularized system $Z_{\e}^{\Phi}$ admits limit cycles passing through the section $\widehat H_{\rho,\la}^{\e}.$	
		
		In what follows, we prove the uniqueness of the fixed point in  $[\e,y^\e_{\rho,\la}].$ Indeed, expanding $P^e$ in Taylor series around $y=y_{x_\e}^\e,$ we have that
		$$P^e(y)=P^e(y_{x_\e}^\e)+\frac{dP^e}{dy}(y_{x_\e}^\e)(y-y_{x_\e}^\e)+\mathcal{O}((y-y_{x_\e}^\e)^2).$$ Thus,
		$$\begin{array}{lllll} \pi_{\e}(y) &= & P^e(y_{x_\e}^\e+\mathcal{O}(e^{-c/\e^q}))\\
		& = &\displaystyle P^e(y_{x_\e}^\e)+\frac{dP^e}{dy}(y_{x_\e}^\e)\mathcal{O}(e^{-c/\e^q})+\mathcal{O}(e^{-2c/\e^q})\\
		& = &\displaystyle P^e(y_{x_\e}^\e)+\mathcal{O}(e^{-c/\e^q}),\\
		\end{array}$$ and, consequently, $|\pi_\e(y_1)-\pi_\e(y_2)|=\mathcal{O}(e^{-c/\e^q}),$ for all $y_1,y_2\in[\e,y^\e_{\rho,\la}].$ Now, consider the following function
		$$\begin{array}{rcl} \nu_{\e}:  [\e,y^\e_{\rho,\la}]&\longrightarrow & [0,1]\\ 
		y &\longmapsto & \displaystyle\frac{y}{y^\e_{\rho,\la}-\e}+\frac{\e}{\e-y^\e_{\rho,\la}}.\\ 
	\end{array}$$ Notice that $\nu_{\e}^{-1}(u)=(y^\e_{\rho,\la}-\e)u+\e.$ Hence, if $\widetilde{\pi}_\e(u)=\pi_\e\circ\nu_\e^{-1}(u),$ then $$|\widetilde{\pi}_\e(u_1)-\widetilde{\pi}_\e(u_2)|=\mathcal{O}(e^{-c/\e^q}),$$ for all $u_1,u_2\in[0,1].$ Fix $l\in(0,1),$ take $u_1,u_2\in[0,1],$ and define the function $\ell(\e)=(y^\e_{\rho,\la}-\e)l.$  There exists $\e(u_1,u_2)>0$ and a neighborhood $U(u_1,u_2)\subset[0,1]^2$ of $(u_1,u_2)$  such that 
$$|\widetilde{\pi}_\e(x)-\widetilde{\pi}_\e(y)|<\ell(\e)|x-y|,$$ for all $(x,y)\in U(u_1,u_2)$ and $\e\in(0,\e(u_1,u_2)).$ Since $\{U(u_1,u_2):\,(u_1,u_2)\in [0,1]^2\}$ is an open cover  of the compact set $[0,1]^2,$ there exists a finite sequence  $(u^i_1,u^i_2)\in[0,1]^2,$ $i=1,\ldots,s,$ for which $\{U^i\defeq U(u_1^i,u_2^i):\,i=1,\ldots,s\}$ still covers $[0,1]^2.$
Taking $\breve{\e}=\min\{\e(u^i_1,u^i_2):i=1,\ldots,s\},$ we obtain that
$$|\widetilde{\pi}_\e(x)-\widetilde{\pi}_\e(y)|<\ell(\e)|x-y|,$$ for all $\e\in(0,\breve{\e})$ and $(x,y)\in[0,1]^2.$
Finally, since $\pi_{\e}(z)=\widetilde{\pi}_\e\circ\nu(z),$ we get
$$\begin{array}{lllll} |\pi_\e(x)-\pi_\e(y)| &= & |\widetilde{\pi}_\e\circ\nu_\e(x)-\widetilde{\pi}_\e\circ\nu_\e(y)|\\
& < & \ell(\e)|\nu_\e(x)-\nu_\e(y)|\\
& = &\displaystyle \displaystyle\frac{\ell(\e)}{y^\e_{\rho,\la}-\e}|x-y|\\
& = &\displaystyle l|x-y|,\\
\end{array}$$ for all $\e\in(0,\breve{\e})$ and $x,y\in[\e,y_{\rho,\la}^{\e}].$ Thus, we have concluded that $\pi_\e$ is a contraction for $\e>0$ small enough. By the \textit{Banach Fixed Point Theorem}, $\pi_\e$ admits a unique asymptotically stable fixed point for $\e>0$ small enough. Therefore, the regularized system $Z_{\e}^{\Phi}$ admits a unique asymptotically stable limit cycle $\Gamma_{\e}$ passing through the section $\widehat H_{\rho,\la}^{\e},$ for $\e>0$ sufficiently small. 

Finally, to conclude this proof, we show that the Hausdorff distance between $\Gamma$ and $\Gamma_{\e}$ is of order $\e$, that is $d_H(\Gamma,\Gamma_{\e})=\CO(\e).$ Recall that the Hausdorff distance $d_H$ between nonempty subsets $A$ and $B$ is given by  
\[
d_H(A,B)=\max\left\{\sup_{x\in A}\inf_{y\in B}|x-y|\,,\,\sup_{x\in B}\inf_{y\in A}|x-y|\right\}.
\]
Hence, $d_H(\Gamma,\Gamma_{\e})=\CO(\e)$ provided that the following statements hold:
\begin{itemize}
\item[i.] for each $p\in\Gamma_{\e}$ there exists $q\in\Gamma$ such that $|p-q|=\CO(\e);$
\item[ii.] for each $q\in\Gamma$ there exists $p\in\Gamma_{\e}$ such that $|p-q|=\CO(\e).$
\end{itemize}

In the sequel, we shall verify item i. Item ii can be verified analogously.

Consider the points $\gamma_{\e}^{\e}\in \Gamma_{\e}$ and $\delta_0^{\e}\in\Gamma$ given by
\[
\{\gamma_{\e}^{\e}\}=\Gamma_{\e}\cap\{(x,\e):\,x>0\}\,\text{ and }\,\{\delta_0^{\e}\}=\Gamma\cap\{(\pi_x \gamma_{\e}^{\e},y):\,0<y<\e\}.
\] 
Here, $\pi_x$ denotes the projection onto the first coordinate. Notice that $|\gamma_{\e}^{\e}-\delta_0^{\e}|<\e.$  Moreover, there exists $t_\e>0$ such that $\varphi_{X^+}^2(t_\e,\gamma_{\e}^{\e})=\e$ and $\varphi_{X^+}^2(t,\gamma_{\e}^{\e})>\e,$ for all $t\in(0,t_\e)$. Thus, $\Gamma_\e$ can be decomposed as the union $\Gamma_\e=\Gamma_\e^+\cup\Gamma_\e^r,$ where
\[
\Gamma_\e^+=\{\varphi_{X^+}(t,\gamma_{\e}^{\e}):t\in[0,t_\e]\}\,\text{ and }\,\Gamma_\e^r=\Gamma_\e\cap\{(x,y):y\in[-\e,\e]\}.
\] 

Since $\Gamma_\e^r$ is delimited from bellow by $\Gamma$ and from above by the horizontal line $\{(x,\e):\,x\in\R\}$, we have that, for each $p\in\Gamma_\e^r,$ there exists $q=(\pi_x p,y_p)\in\Gamma\cap\{(x,y):y\in(0,\e)\}.$ Thus, $d(p,q)=\CO(\e).$

Finally,  we check  property i. for points $p\in\Gamma_{\e}^+$. Since  $\varphi_{X^+}$ is Lipschitz on compacts, there exists a constant $L>0$ such that 
\begin{equation}\label{egammae}
|\varphi_{X^+}(t,\gamma_{\e}^{\e})-\varphi_{X^+}(t,\delta_0^{\e})|\leqslant L|(t,\gamma_{\e}^{\e})-(t,\delta_0^{\e})|
=L|(0,\gamma_{\e}^{\e}-\delta_0^{\e})|
\leqslant L \e,
\end{equation}
for all $t\in[0,t_\e]$. Now, for each $p\in\Gamma_\e^+$, there exists $t^p\in[0,t_\e],$ such that $p=\varphi_{X^+}(t^p,\gamma_{\e}^{\e}).$ Thus, taking $q=\varphi_{X^+}(t^p,\delta_0^{\e})\in\Gamma$ and using \eqref{egammae}, we conclude that $|p-q|=\CO(\e).$

\section{Cases of uniqueness and nonexistence of limit cycles} \label{sec:nonexistence}
Consider the Filippov system $Z=(X^+,X^-)_{\Sigma}$ and assume that
\begin{itemize}
\item[{\bf (H)}] $X^+$ has locally a unique isocline $x=\psi(y)$ of $2k-$multiplicity contacts with the straight lines $y=\e,$ $\e>0$ small enough.
\end{itemize}
From the comments of Remark \ref{assumption2}, we shall prove the following proposition.
\begin{proposition}\label{propd}
Consider a Filippov system $Z=(X^+,X^-)_{\Sigma}$ and assume that $X^+$ satisfies hypotheses {\bf (B)} and {\bf (H)} for some $k\geqslant 1.$ For $ n\geqslant 2k-1,$ let $\Phi\in C^{n-1}_{ST}$ be given as $\eqref{Phi}.$ Then, the following statements hold.
\begin{enumerate}
	\item[(a)] If the limit cycle $\Gamma$ is unstable, then for $\e>0$ sufficiently small the regularized system $Z_{\e}^{\Phi}$ \eqref{regula} does not admit limit cycles converging to $\Gamma.$
	
		\item[(b)] If the limit cycle $\Gamma$ is asymptotically stable, then for $\e>0$ sufficiently small the regularized system $Z_{\e}^{\Phi}$ \eqref{regula} admits a unique limit cycle $\Gamma_{\e}$ converging to $\Gamma.$ Moreover, $\Gamma_{\e}$ is hyperbolic and asymptotically stable.
\end{enumerate}
\end{proposition}
\subsection{Mirror maps in the regularized system}
Consider the nonsmooth vector field $Z=(X^+,X^-)$ and assume that $X^+$ satisfies hypotheses {\bf (A)} and {\bf (H)} for some $k\geqslant 1.$ For $n\geqslant\max\{2, 2k-1\},$ let $\Phi\in C^{n-1}_{ST}$ be given as \eqref{Phi} and consider the regularized system $Z_{\e}^{\Phi}$ \eqref{regula}. In what follows, we shall see that, for each $(x,\e)\in\{y=\e\}$ near to $(\psi(\e),\e)$ there exists a unique small time $t(x,\e)$ satisfying $t(x,\e)=0$ if, and only if, $x=\psi(\e)$ and $\varphi_{Z_{\e}^{\Phi}}(t(x,\e),x,\e)\in\{y=\e\}.$ In this case, we can define the following map
\[\begin{array}{ccc}
\rho_\e:V^-_{\psi(\e)}\subset\{y=\e\}&\longrightarrow&V^+_{\psi(\e)}\subset\{y=\e\}\\
(x,\e)&\longmapsto&\varphi_{Z_{\e}^{\Phi}}(t(x,\e),x,\e),
\end{array}\]
where $V^-_{\psi(\e)}=(\psi(\e)-\delta_\e^-,\psi(\e)]\times\{\e\}$ and $V^+_{\psi(\e)}=[\psi(\e),\psi(\e)+\delta_\e^+)\times\{\e\},$ for some positive real numbers $\delta_\e^-,\delta_\e^+.$ Notice that $\rho_\e(\psi(\e),\e)=(\psi(\e),\e).$ The map $\rho_\e$ is called the {\it Mirror Map} associated with $Z_{\e}^{\Phi}$ at $\psi(\e)$ (see Figure \ref{mirrormap1}).

\begin{figure}[h]
	\begin{center}
    \begin{overpic}[scale=0.8]{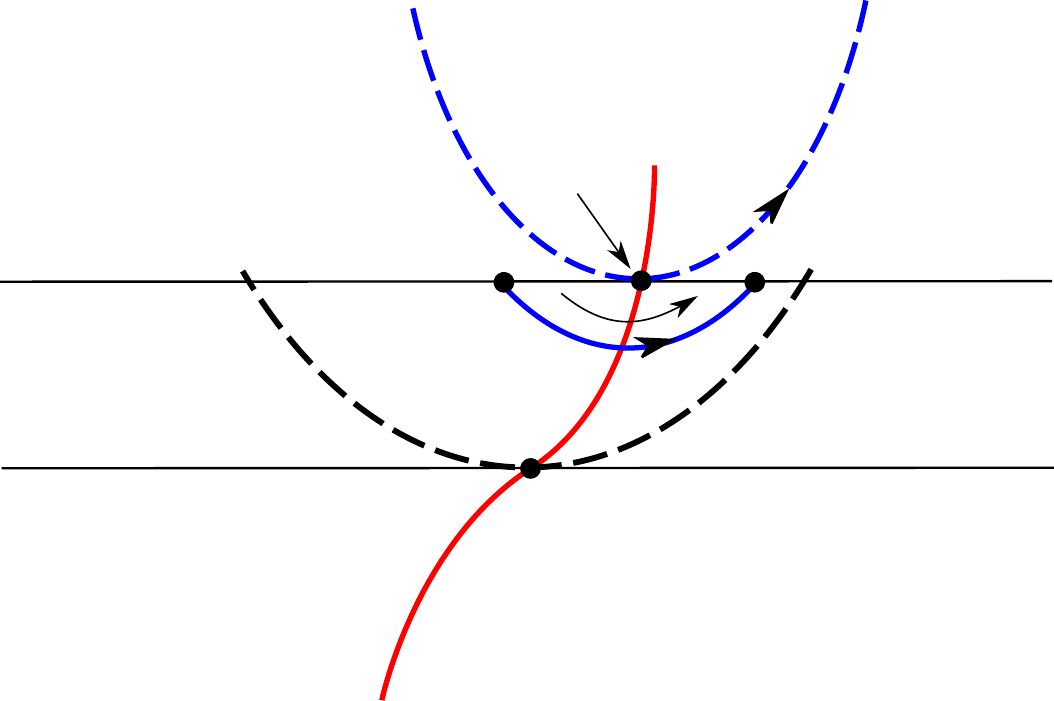}
		\put(51,50){$\psi(\e)$}
		\put(41,42){{\scriptsize $(x,\e)$ \par}}
		\put(71,42){{\scriptsize $\rho_\e(x,\e)$ \par}}
		\put(101,22){$\Sigma$}
		\put(101,40){$y=\e$}
		\put(41,7){$x=\psi(y)$}
		\end{overpic}
		\caption{Mirror Map $\rho_\e$ of $Z_{\e}^{\Phi}$ at $\psi(\e)$.}
	\label{mirrormap1}
	\end{center}
	\end{figure}

First, consider the horizontal and vertical translations $u=x-\psi(\e)$ and $v=y-\e,$ respectively. Notice that $(u,v)=(0,0)$ is a point on the isocline $u=\psi(v+\e)-\psi(\e)$ in the $(u,v)-$coordinates. Define the vector fields $X^+_\e(u,v):=X^+(u+\psi_{\e}(\e),v+\e)$ and $\widetilde{Z}_\e^\Phi(u,v):=Z_\e^\Phi(u+\psi(\e),v+\e).$ Expanding $\pi_2\circ\varphi_{\widetilde{Z}_{\e}^{\Phi}}(t,u,0)$ in Taylor series around $t=0,$ we get
\begin{equation}\label{stpsi}
\pi_2\circ\varphi_{\widetilde{Z}_{\e}^{\Phi}}(t,u,0)=\sum_{i=1}^{2k}\frac{(X^+_\e)^ih(u,0)}{i!}t^i+\CO(t^{2k+1}).\end{equation} 
From the construction of Section \ref{sec:canprel}, it is easy to see that
\begin{equation}\label{derlie}
(X^+_\e)^ih(u,0)=\frac{\al_\e(2k-1)!}{(2k-i)!}u^{2k-i}+\CO(u^{2k-i+1}),
\end{equation}
for each $i\in\{1,\cdots,2k\},$ where 
\[
\al_\e=\frac{1}{(2k-1)!}\dfrac{\partial^{2k-1}f_\e}{\partial u^{2k-1}}(0,0)>0 \text{ and } f_\e(u,v)=\frac{\pi_2\circ X^+(u+\psi(\e),v+\e)}{\pi_1\circ X^+(u+\psi(\e),v+\e)}.\]
Notice that $\al_0=\al>0,$ which is given in \eqref{Xnf}. Now, we define the map
\[S(s,u,\e)=\frac{2k}{\al_\e u^{2k}}\pi_2\circ\varphi_{\widetilde{Z}_{\e}^{\Phi}}(su,u,0).\]
Using \eqref{stpsi} and \eqref{derlie} we can rewrite $S$ as
\[S(s,u,\e)=-1+(1+s)^{2k}+\CO(u,\e).\]
Since $S(-2,0,0)=0$ and $\frac{\partial S}{\partial s}(-2,0,0)=-2k<0$, by \textit{Implicit Function Theorem} we get the existence of a smooth function $s(u,\e)$ such that $s(0,0)=-2$ and $S(s(u,\e),u,\e)=0$. From the definition of $S,$ for $t(u,\e)=us(u,\e)$ we have that $\pi_2\circ\varphi_{\widetilde{Z}_{\e}^{\Phi}}(t(u,\e),u,0)=0.$ Finally, expanding $s$ around $(u,\e)=(0,0)$ we get that $s(u,\e)=-2+\CO(u,\e).$ Consequently, we can define the map $\widetilde{\rho}_\e$ in a neighborhood $V_0\subset\Sigma$ of $(0,0)$ by
$$\widetilde{\rho}_\e(u,0)=u+t(u,\e)=-u+\CO(u^2,\e u).$$ 
Therefore, going back to the original coordinates, we conclude that 
$$\rho_\e(x,\e)=-x+2\psi(\e)+\CO\left((x-\psi(\e))^2,\e (x-\psi(\e))\right).$$  

\subsection{The first return map $\pi_\e$}\label{sec:firstreturnmap}

To prove Proposition \ref{propd} we need to define the first return map $\pi_\e$ of $Z_{\e}^{\Phi},$ for $\e>0$ sufficiently small.

First of all, take $\rho,\e>0$ small enough in order that the intersections of the trajectory of $Z_{\e}^{\Phi}$ starting at $(\psi(\e),\e)$ with the sections $\{x=-\rho\}$ and $\{x=x_\e\}$ are contained in $U,$ namely $(-\rho,\ov{y}^\e_{-\rho})$ and $(x_\e,\ov{y}^\e_{x_\e}),$ respectively. Since $\pi_1\circ X^+(-\rho,\ov{y}^\e_{-\rho})\neq 0$ and $\pi_1\circ X^+(x_\e,\ov{y}^\e_{x_\e})\neq 0$, then $\{x=-\rho\}$ and $\{x=x_\e\}$ are transversal sections of $X^+$ at the points $(-\rho,\ov{y}^\e_{-\rho})$ and $(x_\e,\ov{y}^\e_{x_\e}),$ respectively. 
Hence, by \cite[Theorem A]{AndGomNov19} we know that there exist the transition maps $T_\e^{u}:[\psi(\e),x_\e]\times\{\e\}\longrightarrow\{x=x_\e\}$ and $T_\e^{s}:[-\rho,\psi(\e)]\times\{\e\}\longrightarrow\{x=-\rho\}$ satisfying
\begin{equation*}\label{Tsu2}
\begin{split}
T_\e^{u}(x)&=\ov{y}^\e_{x_\e}+\kappa_{x_\e,\e}^{u}(x-\psi(\e))^{2k}+\CO\left((x-\psi(\e))^{2k+1}\right),\\
T_\e^{s}(x)&=\ov{y}^\e_{-\rho}+\kappa_{\rho,\e}^{s}(x-\psi(\e))^{2k}+\CO\left((x-\psi(\e))^{2k+1}\right),\\
\end{split}
\end{equation*}
where $\sgn(\kappa_{x_\e,\e}^{u})=-\sgn((X^{+})^{2k}h(\psi(\e)))=\sgn(\kappa_{\rho,\e}^{s})$, i.e. $\kappa_{x_\e,\e}^u,\kappa_{\rho,\e}^s<0$ (see Figure \ref{mirrormap11}). Using the \textit{Implicit Function Theorem}, it is easy to see  that
$$(T_\e^{s})^{-1}(y)=\psi(\e)-\sqrt[2k]{\frac{1}{-\kappa_{\rho,\e}^{s}}}(\ov{y}^\e_{-\rho}-y)^{\frac{1}{2k}}+\CO\left((\ov{y}^\e_{-\rho}-y)^{1+\frac{1}{2k}}\right).$$ 
Now, we know that there exists a diffeomorphism 
$D:\{x=x_\e\}\longrightarrow\{x=-\rho\}$ given by 
$$D(y)=\ov{y}^\e_{-\rho}+K^\e_{x_\e,\rho}(y-\ov{y}^\e_{x_\e})+\CO((y-\ov{y}^\e_{x_\e})^2).$$
Finally, we get the first return map $\pi_\e:\{x=-\rho\}\longrightarrow\{x=-\rho\}$ defined as
\begin{equation}\label{pireg1}\begin{array}{rcl}
\pi_\e(y)&:=&D\circ T_\e^u\circ\rho_\e\circ(T_\e^s)^{-1}(y)\\
&=&\ov{y}^\e_{-\rho}-\dfrac{K^\e_{x_\e,\rho}\kappa_{x_\e,\e}^{u}}{\kappa_{\rho,\e}^{s}}(\ov{y}^\e_{-\rho}-y)+\CO((\ov{y}^\e_{-\rho}-y)^p)+\CO(\e),
\end{array}\end{equation}
for some $p>1.$ 
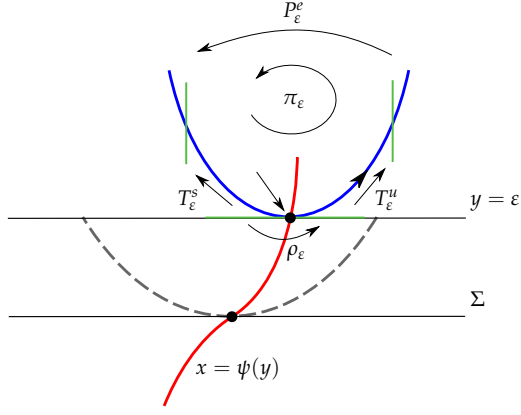
\begin{figure}[h]
	\begin{center}
    \begin{overpic}[scale=0.69]{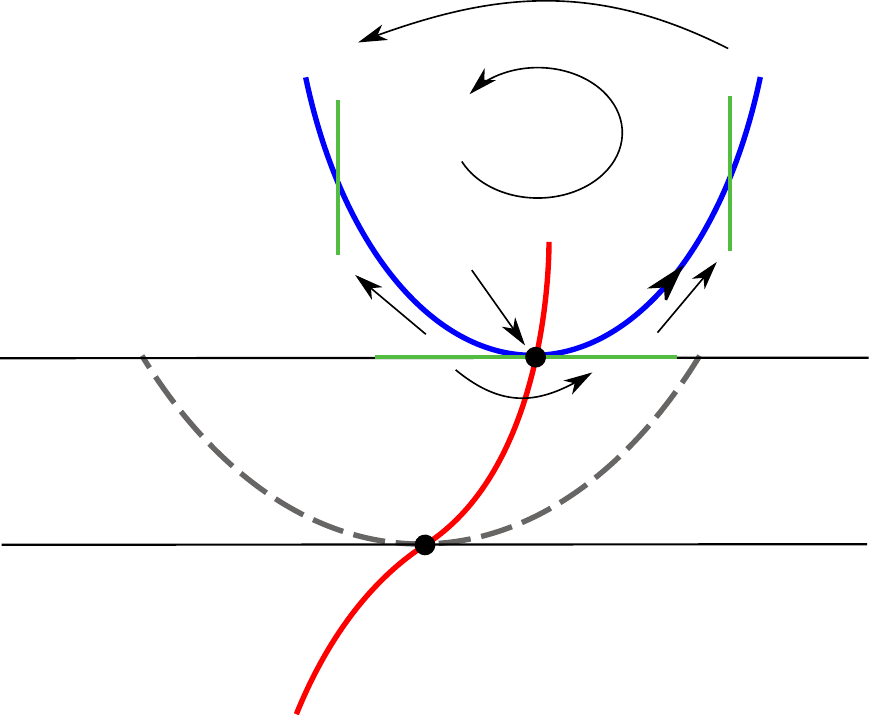}
		\put(37,44){{\scriptsize $T_\e^s$ \par}}
		\put(80,44){{\scriptsize $T_\e^u$ \par}}
		\put(101,22){{\scriptsize $\Sigma$ \par}}
		\put(101,44){{\scriptsize $y=\e$ \par}}
		\put(60,66){{\scriptsize $\pi_\e$ \par}}
		\put(60,85){{\scriptsize $P_\e^e$ \par}}
		\put(61,33){{\scriptsize $\rho_\e$ \par}}
		\put(41,7){{\scriptsize $x=\psi(y)$ \par}}
		\end{overpic}
		\caption{The first return map $\pi_\e$ of $Z_\e^\Phi.$
		}
	\label{mirrormap11}
	\end{center}
	\end{figure}
\subsection{Proof of Proposition \ref{propd}}\label{sec:proofpropd}
First of all, if $\Gamma_\e$ is a limit cycle  of the regularized system $Z_{\e}^{\Phi}$ \eqref{regula}  converging to $\Gamma$ (i.e. there exists a fixed point $(-\rho,y_\e^\rho)\in\{x=-\rho\}$ of $\pi_\e$ such that $\lim\limits_{\e \to 0}y_\e^\rho=\ov{y}_{-\rho}$), then by \eqref{pireg1} we get
$$\frac{d\pi_\e}{dy}(y)=\dfrac{K^\e_{x_\e,\rho}\kappa_{x_\e,\e}^{u}}{\kappa_{\rho,\e}^{s}}+\CO\left((\ov{y}^\e_{-\rho}-y)^{p-1}\right).$$
It is easy to see that $\lim_{\e,\rho \to 0}\kappa_{x_\e,\e}^{u}/\kappa_{\rho,\e}^{s}=1.$ Thus, using Lemma \ref{lemma6} we have that
$$\lim\limits_{\e,\rho \to 0}\frac{d\pi_\e}{dy}(y_\e^\rho)=K.$$
Hence, since $\Gamma$ is hyperbolic provided that $\Gamma$ is unstable (resp. asymptotically stable), then $K>1$ (resp. $K<1$). Consequently, $\Gamma_\e$ is unstable (resp. asymptotically stable), for $\e>0$ sufficiently small.
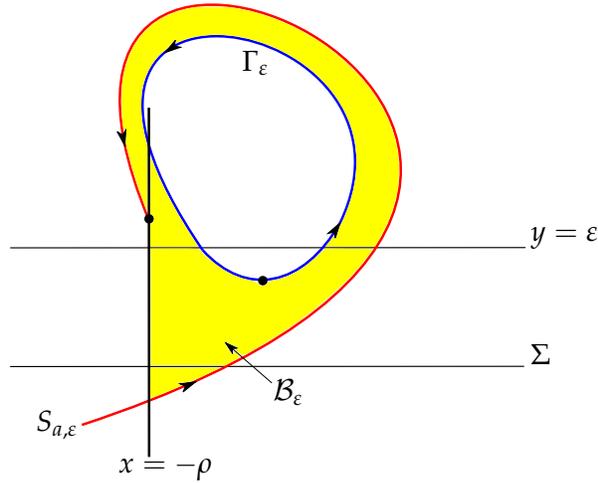
\begin{figure}[h]
	\begin{center}
    \begin{overpic}[scale=0.6]{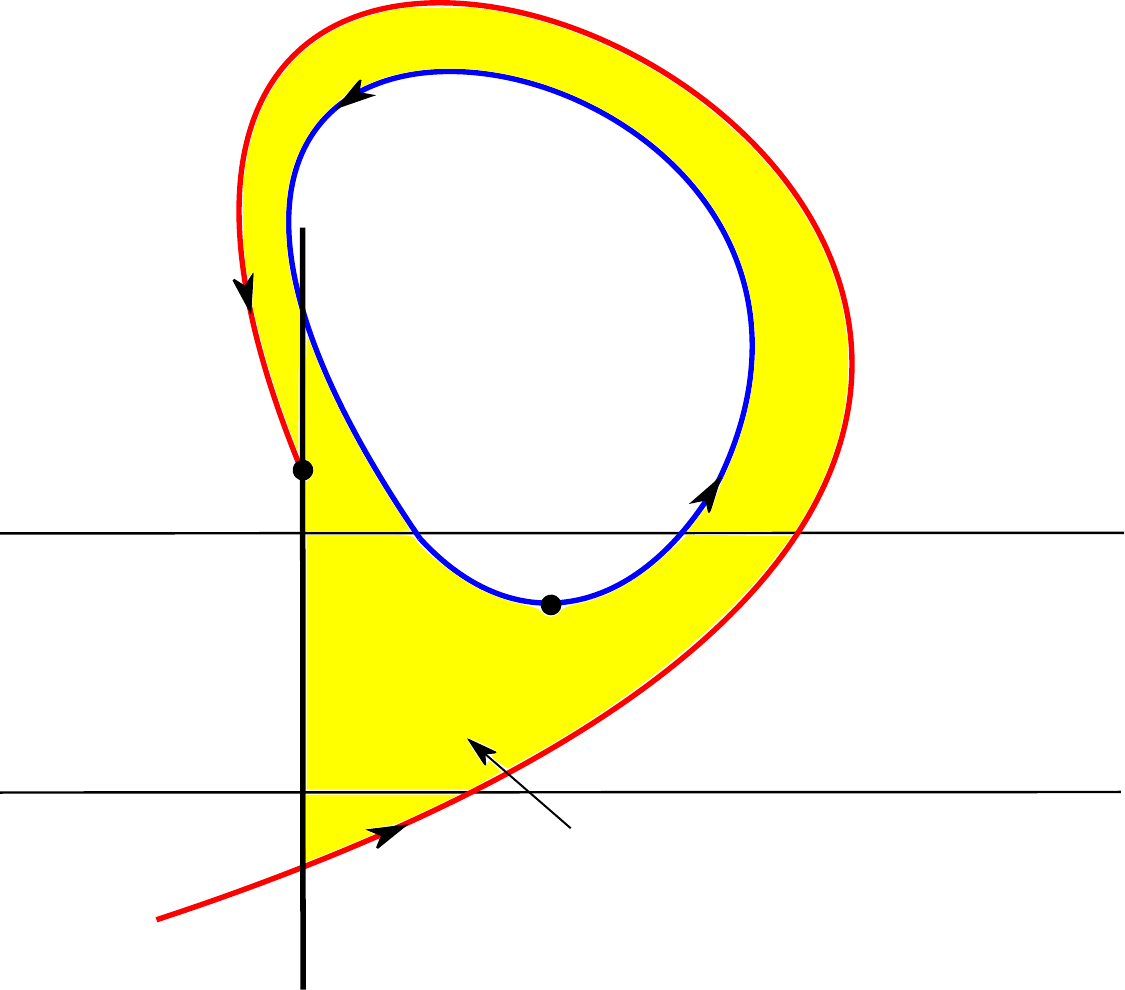}
		\put(5,5){$S_{a,\e}$}
		\put(21,-3){$x=-\rho$}
		\put(45,75){$\Gamma_\e$}
		\put(101,18){$\Sigma$}
		\put(101,42){$y=\e$}
		\put(51,11){$\mathcal{B}_\e$}
		\end{overpic}
		\caption{The region $\mathcal{B}_\e$.}
	\label{regiongsr1}
	\end{center}
	\end{figure}
	
The proof of the first statement is by contradiction. Suppose that there exists a limit cycle $\Gamma_\e$ of $Z_{\e}^{\Phi}$ such that $\Gamma_\e$ converges to $\Gamma,$ for $\e>0$ small enough. Consider the region $\mathcal{B}_\e$ delimited by the curves $x=-\rho,$ the limit cycle $\Gamma_\e$ and the Fenichel manifold $S_{a,\e}$ associated with $Z_{\e}^{\Phi},$ (see Figure \ref{regiongsr1}). It is easy to see that  $\mathcal{B}_\e$ is positively invariant compact set, and has no singular points (because $\Gamma_\e$ converges to the regular orbit $\Gamma$), for $\e>0$ small enough. For $\e>0$ sufficiently small choose $q_\e\in\mathcal{B}_\e$ from the \textit{Poincar\'{e}--Bendixson Theorem} $\omega(q_\e)\subset \mathcal{B}_\e$ is a limit cycle of $Z_{\e}^{\Phi}$ that is not unstable, which is a contradiction.

Now, we shall prove the second statement. Indeed, from Theorem \ref{tc}, for $\e>0$ small enough, we know that $Z_{\e}^{\Phi}$ admits a asymptotically stable limit cycle $\Gamma_\e$ converging to $\Gamma.$ Moreover, from above we have that $\Gamma_\e$ is hyperbolic. Finally, we claim that $\Gamma_\e$ is the unique limit cycle with these properties. Indeed, suppose that there exists another limit cycle $\widetilde{\Gamma_\e}$ converging to $\Gamma$, hyperbolic and asymptotically stable. Now, consider the region $\mathcal{R}_\e$ delimited by the limit cycles $\Gamma_\e$ and $\widetilde{\Gamma_\e}.$  Notice that $\mathcal{R}_\e$ is negatively invariant compact set and has no singular points (because $\Gamma_\e$ and $\widetilde{\Gamma_\e}$ converges to the regular orbit $\Gamma$), for $\e>0$ small enough. For $\e>0$ sufficiently small choose $q_\e\in\mathcal{R}_\e,$ from the \textit{Poincar\'{e}--Bendixson Theorem} we can conclude that $\alpha(q_\e)\subset\mathcal{R}_\e$ is a limit cycle of $Z_{\e}^{\Phi}$ that is not asymptotically stable,  which is a contradiction.

	\section{Piecewise Polynomial Example}\label{sec:example}
	This section is devoted to providing examples of piecewise polynomial transition functions and piecewise polynomial vector fields satisfying the hypotheses of Theorem \ref{tc}. 
	
	\begin{proposition}\label{phi_n} For $n\geqslant 1,$ consider
	\[
	\phi_n(x)=(-1)^n\dfrac{(2n+1)!}{2^{2n}(n!)^2}\int_0^x(s-1)^n(s+1)^n ds.
	\]
	Define $\Phi_n:\R\rightarrow\R$ as $\Phi_n(x)=\phi_n(x)$  for $x\in(-1,1),$ and $\Phi_n(x)=\sgn(x)$ for $|x|\geqslant1$ . Then, $\Phi_n\in C^n_{ST}$ for every positive integer $n.$
	\end{proposition}
\begin{proof}
Notice that $\phi_n(\pm 1)=\pm 1$ and 
$$\phi'_n(x)=(-1)^n\dfrac{(2n+1)!}{2^{2n}(n!)^2}(x-1)^n(x+1)^n.$$ Thus, $\phi'_n(x)>0$ for all $x\in(-1,1),$  $\phi^{(i)}_n(\pm 1)=0$ for $i=1,\ldots,n,$ and $$\phi^{(n+1)}_n(\pm 1)=\prod_{i=1}^{n}(\mp 1)^n(2i+1)\neq 0.$$ 
Consequently, $\Phi_n\in C^n_{ST}.$\end{proof}
	
Now, consider the planar vector field $Z=(X^+,X^-),$ with $X^-(x,y)=(0,1)$ and $X^+=(X_1^+,X_2^+),$ where
	$$X_1^+(x,y)=-x(-1+x^{2 k})+(-1+y)^{2k-1}(-1+x-xy),$$  and
$$X_2^+(x,y)=x^{2k-1}-(-1+x^{2k}+(-1+y)^{2k})(-1+y), \hspace{0.1cm}\text{for}\hspace{0.1cm} k>1.$$
Define $\Sigma=h^{-1}(0),$ with $h(x,y)=y.$ Notice that the vector field $Z$ has a $2k$-multiplicity contact with $\Sigma$ at $(0,0).$ Indeed, $(X^+)^ih(0,0)=0,$ for $i=1,\ldots,2k-1,$ and $(X^+)^{2k}h(0,0)=(2k-1)!.$ 
	Now, let $H(x,y)=1-x^{2k}-(y-1)^{2k}$ and consider  the level curve $\Gamma=H^{-1}(0).$ Notice that
	$$\langle DH(x,y),X^+(x,y)\rangle\Big|_{H^{-1}(0)}=0.$$
	Thus, $\Gamma$ is invariant through the flow of $X^+.$ Moreover, $X^+$ has no singularities in $H^{-1}(0).$ Then, by the \textit{Poincar\'e Bendixson Theorem}, $\Gamma$ is a periodic orbit of $X^+.$ Furthermore, for $(x,y)\in \Gamma,$ we get $$div X^+(x,y)=\frac{\partial X^+_{1}}{\partial x}(x,y)+\frac{\partial X^+_{2}}{\partial y}(x,y)=-2k<0.$$
	Thus, given $\gamma$ any parametrization of $\Gamma,$ $T$ its period, and $S$ a transversal section of $X^+$ at $0\in\gamma,$ we have that the derivative of Poincar\'e map $\pi:S_0\subset S\rightarrow S$ is given by $$d\pi(0)=exp\Big[\int_{0}^Tdiv X^+(\gamma(t))dt\Big]=e^{-2k T}.$$
Consequently, we conclude that $\Gamma$ is an asymptotically stable hyperbolic limit cycle of $X^+.$

Therefore, by Theorem \ref{tc}, we conclude that the regularized system $Z_{\e}^{\Phi}$ with $\Phi\in C_{ST}^{n-1}$ admits a unique asymptotically stable limit cycle $\Gamma_{\e}$ passing through the section $\widehat H_{\rho,\la}^{\e}=[-\rho,-\e^{\la}]\times\{\e\},$ for $\e>0$ sufficiently small (see Figures \ref{figZ3} and \ref{figZ4}). Moreover, $\Gamma_{\e}$ is $\e$-close to $\Gamma.$
	 \begin{figure}[H]
	\begin{center}
		\begin{overpic}[scale=0.36]{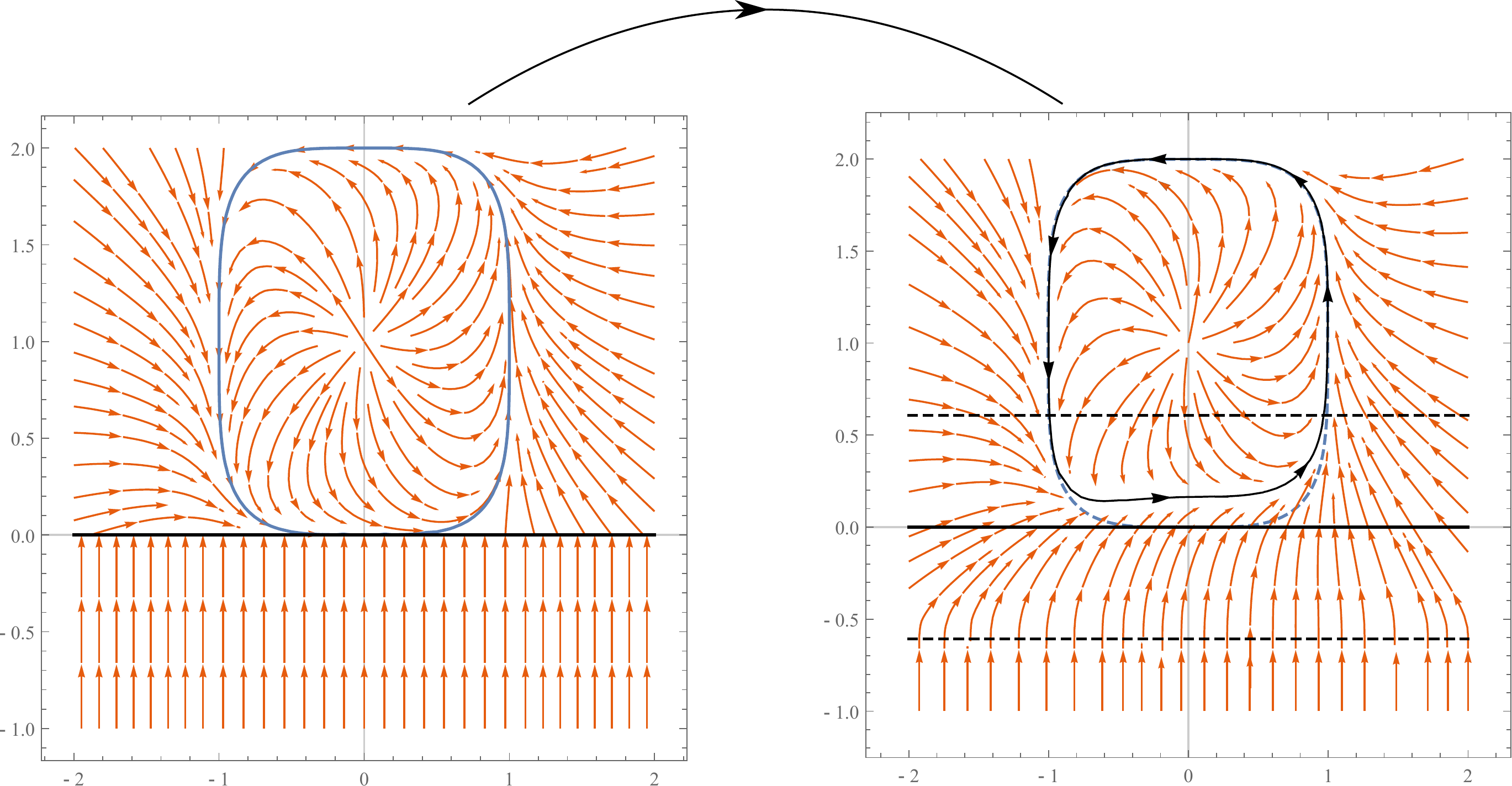}
		\put(40,52.5){$\Phi$-regularization}
		\put(95,18){$\Sigma$}
		\put(40,18){$\Sigma$}
		\put(90,39){$\Gamma_\e$}
		\put(35,39){$\Gamma$}
		\put(101,25){$y=\e$}
		\put(101,10){$y=-\e$}
		\end{overpic}
		\bigskip
	\end{center}
	\caption{ Vector field Z and its regularized system $Z_\e^\Phi$. The figure on the left shows the hyperbolic limit cycle $\Gamma$ passing through the visible $2k$-multiplicity contact with $\Sigma$ at $(0,0)$ and the figure on the right shows the limit cycle $\Gamma_\epsilon,$ for $\e=0.6$, $n=6,$ $k=2$ and $\Phi\in C_{ST}^{5}$ with $\phi(u)=\tfrac{3003}{1024}x-\tfrac{3003 }{512}x^3+\tfrac{9009}{1024}x^5-\tfrac{2145}{256}x^7+\tfrac{
 5005}{1024}x^9-\tfrac{819}{512}x^{11}+\tfrac{231}{1024}x^{13}.$}
	\label{figZ3}
	\end{figure}

		 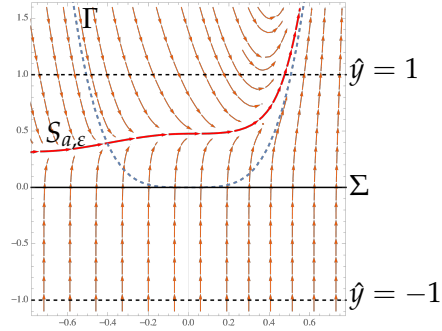
\begin{figure}[h]
	\begin{center}
		\begin{overpic}[scale=0.35]{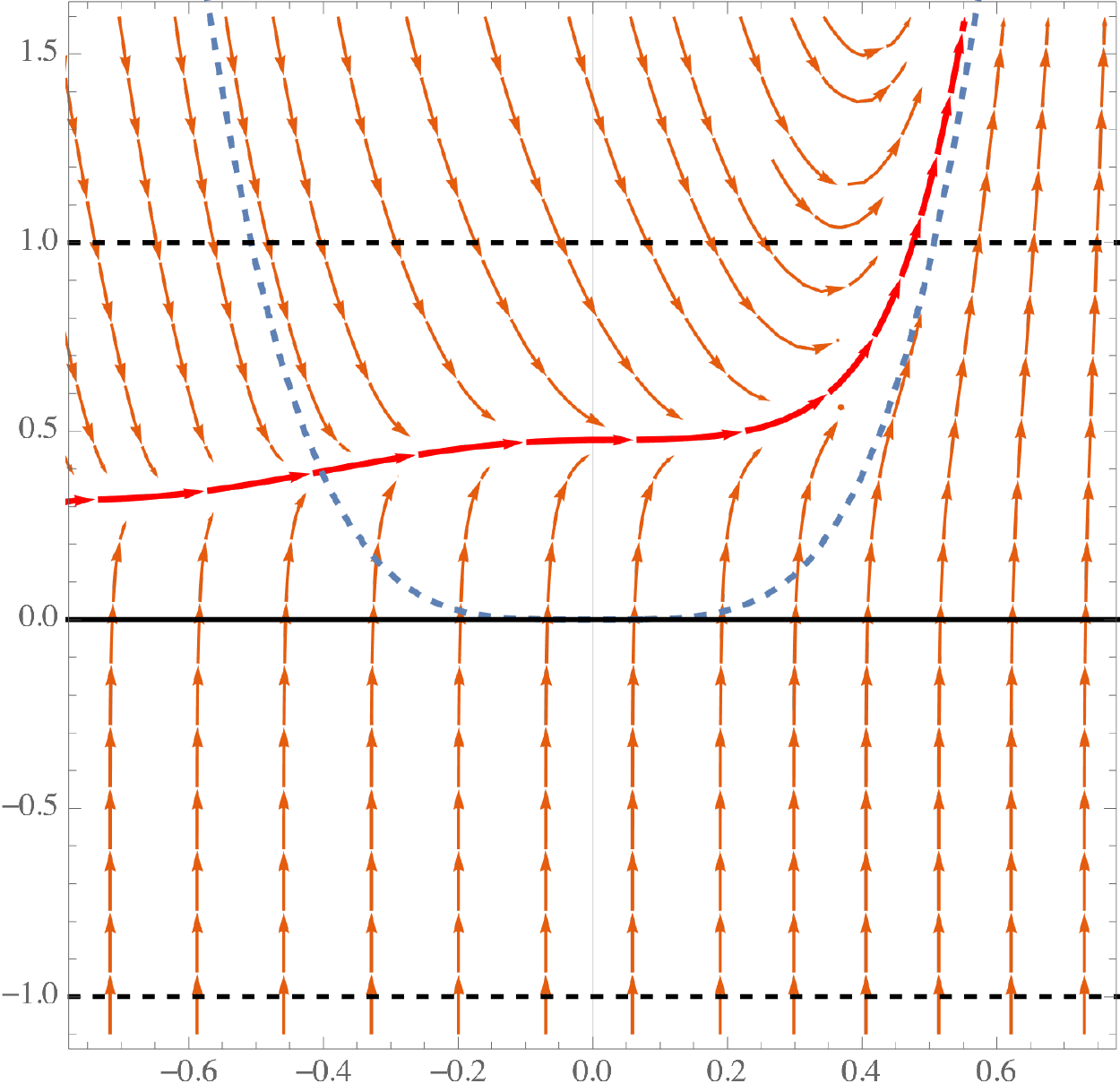}
		\put(101,40){$\Sigma$}
		\put(10,56){$S_{a,\e}$}
		\put(21,90){$\Gamma$}
		\put(101,75){$\hat y=1$}
		\put(101,8){$\hat y=-1$}
		\end{overpic}
		\bigskip
	\end{center}
	\caption{Behaviour of the regularized system $Z_\e^\Phi$ in a neighbourhood of $(x,y)=(0,0)$ in coordinates $(x,\hat y)$, $y=\e\hat y,$ for $\e=0.017$, $n=6,$ $k=2$ and $\Phi\in C_{ST}^{5}$ with $\phi(u)=\tfrac{3003}{1024}x-\tfrac{3003 }{512}x^3+\tfrac{9009}{1024}x^5-\tfrac{2145}{256}x^7+\tfrac{ 5005}{1024}x^9-\tfrac{819}{512}x^{11}+\tfrac{231}{1024}x^{13}.$ The bold arrowed curve illustrates the extension of the Fenichel manifold $S_{a,\e}$. The limit cycle $\Gamma_{\e}$ intersects $\hat y=1$ exponentially close to $S_{a,\e}.$ }
	\label{figZ4}
	\end{figure}

	\section*{Appendix: Proof of Proposition \ref{prop:aux}}\label{apA}
		
Consider the compact region 
\begin{equation*}\label{eq25}\mathcal{B}=\left\{(x,\widehat{y}):-L\leqslant x\leqslant -\e^\la, m_{0}(x)-\frac{\e K}{\sqrt[n]{x^{2k(n-2)+2}}}\leqslant \widehat{y}\leqslant m_{0}(x)\right\}.\end{equation*}
We shall prove that the vector field $\eqref{fastsystem}$ points inwards $\mathcal{B}$ in the following three boundaries of $\mathcal{B},$
\[ 
\begin{array}{ll}
\mathcal{B}^-\!\! \!\!&=\left\{(x,\widehat{y}):-L\leqslant x\leqslant -\e^\la, \widehat{y}=\widehat{y}_{\e}(x)=m_{0}(x)-\frac{\e K}{\sqrt[n]{x^{2k(n-2)+2}}}\right\},\vspace{0.2cm}\\
\mathcal{B}^+\!\!\!\!&=\left\{(x,\widehat{y}):-L\leqslant x\leqslant -\e^\la, \widehat{y}=m_{0}(x)\right\},\quad \text{and}\vspace{0.2cm}\\
\mathcal{B}^l\!\!\!\!&=\left\{(-L,\widehat{y}):m_{0}(-L)-\frac{\e K}{\sqrt[n]{L^{2k(n-2)+2}}}\leqslant \widehat{y}\leqslant m_{0}(-L)\right\}.\\
\end{array}
\]

On the border $\mathcal{B}^-,$ the vector field \eqref{fastsystem} writes 
$$\ov{Z}^\Phi_{\e}(x,\widehat y_{\e}(x))=\Big(\e\left(1+\Phi\left(\widehat{y}_{\e}(x)\right)\right),1+f(x,\e \widehat{y}_{\e}(x))+\Phi(\widehat{y}_{\e}(x))(f(x,\e \widehat{y}_{\e}(x))-1)\Big).$$
A normal vector of $\mathcal{B}^-$ is given by $$n^-_{\e}(x)=\left(m'_0(x)-\tfrac{K\e(2k(n-2)+2)}{n\sqrt[n]{|x|^{(2k+1)(n-2)+4}}},-1\right).$$
Thus, it is enough to see that 
\begin{equation}\label{eq265}
\begin{split}
\langle Z^\Phi_\e(x,\widehat y_{\e}(x)),n^-_{\e}(x)\rangle & = \left[\e(1+\Phi(\widehat{y}_{\e}(x)))\left(m'_0(x)-\tfrac{K\e(2k(n-2)+2)}{n\sqrt[n]{|x|^{(2k+1)(n-2)+4}}}\right)\right]\\
&\quad -\Big[1+f(x,\e \widehat{y}_{\e}(x))+\Phi(\widehat{y}_{\e}(x))(f(x,\e \widehat{y}_{\e}(x))-1)\Big]\\
& <0.\\
\end{split}
\end{equation}

Now, expanding in Taylor series $\Phi(\widehat y_{\e}(x))$ and $\vartheta(x,\e\widehat{y}_{
\e}(x))$ around $\e=0,$ we have 
\begin{equation*}\label{eq26}
\Phi(\widehat y_{\e}(x))=\Phi(m_{0}(x))-\frac{\Phi'(m_0(x))K}{\sqrt[n]{x^{2k(n-2)+2}}}\e+\sum_{l=2}^{n-1}\frac{(-1)^{l}\Phi^{(l)}(m_0(x))K^l}{\sqrt[n]{x^{2lk(n-2)+2l}}}\frac{\e^l}{l!}+s(x,\e),\end{equation*}

\begin{equation*}\label{eq26}
\vartheta(x,\e\widehat{y}_{
\e}(x))=\vartheta(x,0)+r(x,\e),\end{equation*} where $s(x,\e)$ and $r(x,\e)$ are the Lagrange remainders of $\Phi(\widehat y_{\e}(x))$ and $\vartheta(x,\e\widehat{y}_{
\e}(x))$ respectively, i.e. for some  $c,d\in(0,\e),$ we get

\begin{equation}\label{remaiders}\begin{array}{rcl}
s(x,\e)& = &\displaystyle\left[\dfrac{(-1)^{n}\Phi^{(n)}\left(\widehat{y}_{c}(x)\right)K^n}{x^{2k(n-2)+2}}\right]\dfrac{\e^n}{n!}, \quad \text{and}\vspace{0.2cm}\\
r(x,\e)& = &\displaystyle\left[\vartheta_{y}\left(x,d\widehat{y}_d(x)\right)\left(m_0(x)-\dfrac{2dK}{\sqrt[n]{x^{2k(n-2)+2}}}\right)\right]\e.\end{array}\end{equation}
Notice that, the inequality \eqref{eq265} can be written as $L(x,\e)+T(x,\e)+O(x,\e)<0,$ where
\[\begin{array}{rl}\displaystyle L(x,\e)=&\e\Big[m_0'(x)(1+\Phi(m_0(x)))+\frac{\Phi'(m_0(x))K(f(x,0)-1)}{\sqrt[n]{x^{2k(n-2)+2}}}\\ &\displaystyle-m_0(x)(1+\Phi(m_0(x)))\vartheta(x,0)\Big],\vspace{0.2cm}\\
\end{array}\]
\[\begin{array}{rl}
 T(x,\e)  = &\displaystyle -\e^2\frac{m_0'(x)\Phi'(m_0(x))K}{\sqrt[n]{x^{2k(n-2)+2}}}+\e^3\frac{K^2(2k(n-2)+2)\Phi'(m_0(x))}{n\sqrt[n]{|x|^{(4k+1)(n-2)+6}}}\\ 

&\displaystyle +\e^2\frac{m_0(x)\vartheta(x,0)\Phi'(m_0(x))K}{\sqrt[n]{x^{2k(n-2)+2}}}+\e^2\frac{K\vartheta(x,\e\widehat{y}_\e(x))(1+\Phi(m_0(x)))}{\sqrt[n]{x^{2k(n-2)+2}}}\\

&\displaystyle -\frac{\e^3K^2\vartheta(x,\e\widehat{y}_\e(x))\Phi'(m_0(x))}{\sqrt[n]{x^{4k(n-2)+4}}}-\e m_0(x)r(x,\e)(1+\Phi(m_0(x))\\

&\displaystyle +\e^2\frac{m_0(x)r(x,\e)\Phi'(m_0(x))K}{\sqrt[n]{x^{2k(n-2)+2}}}+\e \sum_{l=2}^{n-1}\frac{(-1)^{l}\Phi^{(l)}(m_0(x))K^l m'_0(x)}{\sqrt[n]{x^{2lk(n-2)+2l}}}\frac{\e^l}{l!}\\ 

&\displaystyle -\e^2\sum_{l=2}^{n-1}\frac{(-1)^{l}\Phi^{(l)}(m_0(x))K^{l+1}(2k(n-2)+2)}{n\sqrt[n]{x^{(2k(l+1)+1)(n-2)+2l+4}}}\frac{\e^l}{l!}\\ 

&\displaystyle -\sum_{l=2}^{n-1}\frac{(-1)^{l}\Phi^{(l)}(m_0(x))K^l}{\sqrt[n]{x^{2lk(n-2)+2l}}}\frac{\e^l}{l!}\e\widehat{y}_\e(x)\vartheta(x,\e\widehat{y}_\e(x))\\ 

& +\left(\e m_0'(x)-\e^2\frac{K(2k(n-2)+2)}{n\sqrt[n]{|x|^{(2k+1)(n-2)+4}}}-\e\widehat{y}_\e(x)\vartheta(x,\e\widehat{y}_\e(x))\right)s(x,\e),\vspace{0.2cm}\\

O(x,\e)  = &\displaystyle \left(-f(x,0)+1\right)s(x,\e)-\e^2\frac{K(2k(n-2)+2)(1+\Phi(m_0(x)))}{n\sqrt[n]{|x|^{(2k+1)(n-2)+4}}}\\ 

&\displaystyle +\sum_{l=2}^{n-1}\frac{(-1)^{l}\Phi^{(l)}(m_0(x))K^l}{\sqrt[n]{x^{2lk(n-2)+2l}}}\frac{\e^l}{l!}(-f(x,0)+1).\vspace{0.2cm}\\ 

\end{array}\] 
Now, we shall prove that the functions $L,$ $T$ and $O$ can be bounded. Indeed, by \eqref{eq28} 
and $\eqref{eq22},$ we have that, $L(x,\e)$ can be bounded, choosing $K$ big enough depending on $C_2,$ $L,$ $n,$ $k,$ $\alpha,$ $M,$ $M_{\min},$ and $\vartheta_{\min},$ where
\begin{itemize}
	\item $M$ is such that $|g(x)|\leqslant M|x|^{2k}$ for all $-L\leqslant x\leqslant 0.$
	\item  $\widetilde{M}$  is such that $|\widetilde{g}'(x)|\leqslant \widetilde{M}|x|$ for all $-L\leqslant x\leqslant 0$ with $g'(x)=x^{2k-2}\widetilde{g}'(x).$
	\item $M_{\min}$ is such that $\alpha(2k-1)+\widetilde{g}'(x)\geqslant M_{\min}>0,$ for all $x\in[-L,0].$
	\item $\vartheta_{\min}=\min\{\widehat{y}\vartheta(x,0):-L\leqslant x\leqslant 0,-1\leqslant \widehat{y}\leqslant 1\}.$
\end{itemize}

$$\begin{array}{lllll} &\e\left[m_0'(x)(1+\Phi(m_0(x)))+\dfrac{\Phi'(m_0(x))K(f(x,0)-1)}{\sqrt[n]{x^{2k(n-2)+2}}}-m_0(x)(1+\Phi(m_0(x)))\vartheta(x,0)\right]\vspace{0.2cm}\\ =&\e\left[(m_0'(x)-m_0(x)\vartheta(x,0))(1+\Phi(m_0(x)))+\dfrac{\Phi'(m_0(x))K(f(x,0)-1)}{\sqrt[n]{x^{2k(n-2)+2}}}\right]\vspace{0.2cm}\\ \leqslant&\e\left[(C_2|x|^{-\frac{n-2k+1}{n}}-\vartheta_{\min})\left(\dfrac{2}{1-f(x,0)}\right)-\dfrac{(2\alpha(2k-1)x^{2k-2}+2g'(x))K}{C_2|x|^{-\frac{n-2k+1}{n}}(1-f(x,0))\sqrt[n]{x^{2k(n-2)+2}}}\right]\vspace{0.2cm}\\
\leqslant&\e\left[\dfrac{2C_2|x|^{2k-2+\frac{n-2k+1}{n}}(C_2|x|^{-\frac{n-2k+1}{n}}-\vartheta_{\min})-(2\alpha(2k-1)x^{2k-2}+2x^{2k-2}\widetilde{g}'(x))K}{C_2(1-f(x,0))|x|^{2k-2+\frac{n-2k+1}{n}}}\right]\vspace{0.2cm}\\
\end{array}$$
$$\begin{array}{lllll} 
\leqslant&\e\left[\dfrac{2C_2(C_2+|x|^{\frac{n-2k-1}{n}}|\vartheta_{\min}|)-(2\alpha(2k-1)+2\widetilde{g}'(x))K}{C_2(1-f(x,0))|x|^{\frac{n-2k+1}{n}}}\right]\vspace{0.2cm}\\
\leqslant & \dfrac{2C_2(C_2+L^{\frac{n-2k+1}{n}}|\vartheta_{\min}|)-2M_{\min}K}{C_2(1-f(x,0))}\dfrac{\e}{\sqrt[n]{|x|^{n-2k+1}}}\vspace{0.2cm}\\
\leqslant & \dfrac{2C_2(C_2+L^{\frac{n-2k+1}{n}}|\vartheta_{\min}|)-2M_{\min}K}{C_2(1+L^{2k-1}(\alpha+ML))}\dfrac{\e}{\sqrt[n]{|x|^{n-2k+1}}}\vspace{0.2cm}\\ 
\leqslant & -2\dfrac{\e}{\sqrt[n]{|x|^{n-2k+1}}}.\\
\end{array}$$

Now, we need to bound the function $T(x,\e).$ Using \eqref{eq22} and \eqref{eq28}, we obtain

\[\begin{array}{rcl}\left|\e^2\tfrac{m_0'(x)\Phi'(m_0(x))K}{\sqrt[n]{x^{2k(n-2)+2}}}\right|&\leqslant& d_\e^1\frac{\e}{\sqrt[n]{|x|^{n-2k+1}}},\vspace{0.2cm}\\

\left|\e^3\tfrac{K^2(2k(n-2)+2)\Phi'(m_0(x))}{n\sqrt[n]{|x|^{(4k+1)(n-2)+6}}}\right|&\leqslant& d_\e^2\frac{\e}{\sqrt[n]{|x|^{n-2k+1}}},\vspace{0.2cm}\\

\left|\tfrac{\e^2m_0(x)\vartheta(x,0)\Phi'(m_0(x))K}{\sqrt[n]{x^{2k(n-2)+2}}}\right|&\leqslant& d_\e^3\frac{\e}{\sqrt[n]{|x|^{n-2k+1}}},\vspace{0.2cm}\\

\left|\tfrac{\e^2K\vartheta(x,\e\widehat{y}_\e)(1+\Phi(m_0(x)))}{\sqrt[n]{x^{2k(n-2)+2}}}\right|&\leqslant& d_\e^4\frac{\e}{\sqrt[n]{|x|^{n-2k+1}}},\vspace{0.2cm}\\

\left|\tfrac{\e^3K^2\vartheta(x,\e\widehat{y}_\e)\Phi'(m_0(x))}{\sqrt[n]{x^{4k(n-2)+4}}}\right|&\leqslant& d_\e^5\frac{\e}{\sqrt[n]{|x|^{n-2k+1}}},\vspace{0.2cm}\\

\left|\e m_0(x)r(x,\e)(1+\Phi(m_0(x))
\right|&\leqslant& d_\e^6\frac{\e}{\sqrt[n]{|x|^{n-2k+1}}},\vspace{0.2cm}\\

\left|\e^2\frac{m_0(x)r(x,
\e)\Phi'(m_0(x))K}{\sqrt[n]{x^{2k(n-2)+2}}}
\right|&\leqslant&d_\e^7\frac{\e}{\sqrt[n]{|x|^{n-2k+1}}},\end{array}\] where  
\[\begin{array}{rcl}
d_\e^1& = &\displaystyle(2\alpha(2k-1)+2\widetilde{M}L)K\e^{1-\la\left(\frac{n-2k+1}{n}\right)},\\
d_\e^2& = &\displaystyle\tfrac{K^2(2k(n-2)+2)(2\alpha(2k-1)+2\widetilde{M}L)}{nC_1}\e^{2-\la\left(\frac{(2k+1)(n-2)+4}{n}\right)},\\
d_\e^3& = &\displaystyle\tfrac{|\vartheta_{\max}|(2\alpha(2k-1)+2\widetilde{M}L)K}{C_1}\e,\\
d_\e^4& = &\displaystyle2K|\vartheta_{\max}|\e^{1-\la\left(\frac{(2k-1)(n-1)}{n}\right)},\\
d_\e^5& = &\displaystyle\frac{K^2|\vartheta_{\max}|(2\alpha(2k-1)+2\widetilde{M}L)}{C_1}\e^{2-\la\left(\frac{2k(n-2)+2}{n}\right)},\\
d_\e^6& = &\displaystyle2|\vartheta_{y_m}|(L^{\frac{2k(n-2)+2}{n}}+2dK)\e^{1-\la\left(\frac{(2k-1)(n-1)}{n}\right)},\\
d_\e^7& = &\displaystyle\frac{|\vartheta_{y_m}|(L^{\frac{2k(n-2)+2}{n}}+2dK)(2\alpha(2k-1)+2\widetilde{M}L)K}{C_1}\e^{2-\la\left(\frac{2k(n-2)+2}{n}\right)},\\
   \vartheta_{y_m}& = &\displaystyle\max\{\vartheta_y(x,\widehat{y}):-L\leqslant x\leqslant 0,-1\leqslant \widehat{y}\leqslant 1\},\\
	 \vartheta_{\max}& = &\displaystyle\max\{\widehat{y_1}\vartheta(x,\widehat{y_2}):-L\leqslant x\leqslant 0,-1\leqslant \widehat{y_1},\widehat{y_2}\leqslant 1\}.
	\end{array}
\]
To bound the last terms of $T,$ notice that by \eqref{rf}, we get
\begin{equation*}\label{Texp1}
\Phi^{(l)}(\widehat{y})=\frac{\Phi^{(n)}(1)}{(n-l)!}(\widehat{y}-1)^{n-l}+\mathcal{O}((\widehat{y}-1)^{n-l+1}),\quad 2\leqslant l\leqslant n-1,
\end{equation*} for $\widehat y$ sufficiently near to 1. In the particular case $\widehat{y}=m_0(x)$ for $x\in[-L,0],$ we have
\begin{equation}\label{Texp2}
\Phi^{(l)}(m_0(x))=(m_0(x)-1)^{n-l}\left(\frac{\Phi^{(n)}(1)}{(n-l)!}+\zeta(x)\right),
\end{equation} with $\zeta(x)=\mathcal{O}(m_0(x)-1),$ thus there exists a positive constant $\widehat{M}$ such that $|\zeta(x)|\leqslant \widehat{M}|m_0(x)-1|.$ Therefore, by the above information about $\zeta$ and the first inequation in \eqref{eq28} for $-L\leqslant x\leqslant 0,$ we obtain
$$\left|\Phi^{(l)}(m_0(x))\right|\leqslant C_2^{n-l}|x|^{\frac{(2k-1)(n-l)}{n}}\left(\frac{|\Phi^{(n)}(1)|}{(n-l)!}+\widehat{M}C_2L^{\frac{2k-1}{n}}\right),$$ i.e. for each $l\in[2,n-1]$ we have $|\Phi^{(l)}(m_0(x))|\leqslant C_{l}$ for all $-L\leqslant x\leqslant 0,$ with $C_l=C_2^{n-l}L^{\frac{(2k-1)(n-l)}{n}}\widetilde{C_l}$ and $\widetilde{C_l}=\frac{|\Phi^{(n)}(1)|}{(n-l)!}+\widehat{M}C_2L^{\frac{2k-1}{n}} .$ Consequently, 

\[\begin{array}{rcl}\left|\e \sum_{l=2}^{n-1}\frac{(-1)^{l}\Phi^{(l)}(m_0(x))K^l m'_0(x)}{\sqrt[n]{x^{2lk(n-2)+2l}}}\frac{\e^l}{l!}\right|&\leqslant&d_\e^8\frac{\e}{\sqrt[n]{|x|^{n-2k+1}}},\vspace{0.2cm}\\

\left|\displaystyle \e^2\sum_{l=2}^{n-1}\frac{(-1)^{l}\Phi^{(l)}(m_0(x))K^{l+1}(2k(n-2)+2)}{n\sqrt[n]{x^{(2k(l+1)+1)(n-2)+2l+4}}}\frac{\e^l}{l!}\right|&\leqslant& d_\e^9\frac{\e}{\sqrt[n]{|x|^{n-2k+1}}},\vspace{0.2cm}\\

\left|\displaystyle \sum_{l=2}^{n-1}\frac{(-1)^{l}\Phi^{(l)}(m_0(x))K^l}{\sqrt[n]{x^{2lk(n-2)+2l}}}\frac{\e^l}{l!}\e\widehat{y}_\e(x)\vartheta(x,\e\widehat{y}_\e(x))\right|&\leqslant& d_\e^{10}\frac{\e}{\sqrt[n]{|x|^{n-2k+1}}},\end{array}\] where 
\begin{itemize}
  \item $d_\e^8=\sum_{l=2}^{n-1}\frac{C_{l}K^lC_2}{l!}\e^{l-\la\left(\frac{2kl(n-2)+2l}{n}\right)},$
	\item $d_\e^9=\sum_{l=2}^{n-1}\frac{C_{l}K^{l+1}(2k(n-2)+2)}{l!n}\e^{l+1-\la\left(\frac{2k(l+1)(n-2)+2l+2k+1}{n}\right)},$
	\item $d_\e^{10}=\sum_{l=2}^{n-1}\frac{C_{2}^{n-l}K^l|\vartheta_{\max}|\widetilde{C_l}}{l!}\e^{l-\la\left(\frac{(2k(n-1)+1)(l-1)}{n}\right)}.$
\end{itemize}
Finally, using that, for any $\eta_1>0,$ there exists $C_n>0$ such that
$$|\Phi^{(n)}(\widehat{y})|\leqslant C_n, \hspace{0.1cm} \text{for} \hspace{0.1cm} 1-\eta_1\leqslant\widehat{y}\leqslant 1,$$ and using that, for $\e>0$ small enough 
\begin{equation}\label{yc}
1-\left(C_2L^{\frac{2k-1}{n}}+\e^{1-\la\left(\frac{2k(n-2)+2}{n}\right)}K\right)\leqslant \widehat{y}_c(x)\leqslant 1,
\end{equation} if $-L\leqslant x\leqslant-\e^{\la}$ and also by \eqref{remaiders}, one has
	
$$\left|\left(\e m_0'(x)-\e^2\tfrac{K(2k(n-2)+2)}{n\sqrt[n]{|x|^{(2k+1)(n-2)+4}}}-\e\widehat{y}_\e\vartheta(x,\e\widehat{y}_\e)\right)s(x,\e)\right|\leqslant d_\e^{11}\frac{\e}{\sqrt[n]{|x|^{n-2k+1}}},$$ 
with 
\[\begin{array}{rl}
d_\e^{11}=&\displaystyle\left(C_2\e^{1-\la(\frac{n-2k+1}{n})}+ \dfrac{K(2k(n-2)+2)}{n}\e^{2-\la(\frac{(2k+1)(n-2)+4}{n})}+\e|\vartheta_{\max}|\right)\vspace{0.2cm}\\
&\displaystyle\cdot\frac{C_n K^n}{n!}\e^{n-1-\la\left(\frac{(2kn+1)(n-2)+2k+1}{n}\right)}.\end{array}\] Since $0<\la\leqslant\la^*$ one has that $\lim_{\e\rightarrow 0}d_\e^i=0,$ for all $i\in\{1,\ldots,11\},$ hence for $\e>0$ small enough, we get
$$|T(x,\e)|\leqslant \sum_{i=1}^{11}d_\e^i\frac{\e}{\sqrt[n]{|x|^{n-2k+1}}}\leqslant\frac{1}{2}\frac{\e}{\sqrt[n]{|x|^{n-2k+1}}}.$$ 

Now, we shall prove that the function $O(x,\e)<0$ for all $x\in[-L,-\e^\la]$ and $\e>0$ small enough. Indeed, since for each $n\geqslant 2,$ we know that
$(-1)^n\phi^{(n)}(1)<0,$ then  $(-1)^n\phi^{(n)}(\widehat y)<0,$ for all $\widehat{y}$ sufficiently close to 1 and by \eqref{yc} we obtain that $(-1)^n\phi^{(n)}(\widehat y_c(x))<0,$ for all $x\in[-L,-\e^\la]$ and $\e$ sufficiently enough. Hence, by \eqref{remaiders} we have that $s(x,\e)<0$ for all $x\in[-L,-\e^\la]$ and $\e>0$ small enough. Therefore, $$(-f(x,0)+1)s(x,\e)<0,$$ for all $x\in[-L,-\e^\la]$ and $\e>0$ small enough. After that, using \eqref{Texp2} we can conclude that $(-1)^l\phi^{(l)}(m_0(x))<0,$ for all $x\in[-L,0]$ and $l\in\{2,\ldots,n-1\}.$ Consequently,
$$\sum_{l=2}^{n-1}\frac{(-1)^{l}\Phi^{(l)}(m_0(x))K^l}{\sqrt[n]{x^{2lk(n-2)+2l}}}\frac{\e^l}{l!}(-f(x,0)+1)<0,$$
for all $x\in[-L,-\e^\la]$ and $\e>0$ small enough.
Last of all, as $1+\Phi(m_0(x))>0,$ for all $x\in[-L,0],$ we get
$$-\e^2\frac{K(2k(n-2)+2)(1+\Phi(m_0(x)))}{n\sqrt[n]{|x|^{(2k+1)(n-2)+4}}}<0,$$
for all $x\in[-L,-\e^\la]$ and $\e$ sufficiently small.  In this way, we obtained the result.   

Finally, we conclude that 
\[\begin{array}{rcl}
\langle Z^\Phi_{\e}(x,\widehat{y}_\e(x)),n^-_\e(x)\rangle&\leqslant&L(x,\e)+|T(x,\e)|+O(x,\e)\\
&\leqslant&\left(-2+\frac{1}{2}\right)\frac{\e}{\sqrt[n]{|x|^{n-2k+1}}}<0.
\end{array}\]
Therefore, the vector field $\ov{Z}_\e^\Phi$ points inward $\mathcal{B}$ along $\mathcal{B}^{-}.$

In the border $\mathcal{B}^+$ the vector field $\ov{Z}^\Phi_{\e}$ in \eqref{fastsystem} is of the form 
$$\ov{Z}^\Phi_{\e}=\left(\frac{\e(1+\Phi(m_{0}(x)))}{2},\frac{\e m_0(x)\vartheta(x,\e m_0(x))(1+\Phi(m_0(x)))}{2}\right),$$ and the normal vector is $n^+(x)=(m'_{0}(x),-1),$ thus using the second inequation in \eqref{eq28} for $-L\leqslant x\leqslant-\e^\la,$ we get

$$\begin{array}{lllll} \langle \ov{Z}^\Phi_{\e},n^+(x)\rangle & = &\displaystyle\frac{\e}{2}\Big(1+\Phi(m_{0}(x))\Big)\Big(m'_0(x)-m_0(x)\vartheta(x,\e m_0(x))\Big)\\
&\geqslant & \frac{\e}{2}\Big(\frac{2}{1-f(x,0)}\Big)\Big(m'_{0}(x)-\vartheta_{\max}\Big)\\
&\geqslant & \frac{\e}{2}\Big(\frac{2}{1-f(x,0)}\Big)\Big(\frac{C_1}{L^{\frac{n-2k+1}{n}}}-\vartheta_{\max}\Big)>0,\\ 

\end{array}$$ 
for  $L$ enough small, therefore the flow points inward $\mathcal{B}$ along this border.

Finally, at the boundary $\mathcal{B}^l$ one has that $x'>0$ thus the flow points inward $\mathcal{B}.$

Now, from the \textit{Poincar\'{e}--Bendixson Theorem} we know that any orbit entering $\mathcal{B}$ stays in it  until it reaches $x=-\e^\la.$ Moreover, we know that the invariant manifold $S_{a,\e}$ at $x=-L$ is given by $m(-L,\e)=m_0(-L)+\e m_1(-L)+\mathcal{O}(\e^2).$

Using \eqref{eq28} and since $L$ is small enough one has that $$m'_{0}(-L)-m_0(-L)\vartheta(-L,0)\geqslant\frac{C_1}{L^{\frac{n-2k+1}{n}}}-\vartheta_{\max}>0,$$ thus from \eqref{eq21} $m_1(-L)<0.$ Therefore, adjusting the constants to have $$K\geqslant-L^{\frac{2k(n-2)+2}{n}}m_1(-L),$$ the manifold enters $\mathcal{B}$ and satisfies \eqref{eq24} for $-L\leqslant x\leqslant -\e^{\la}.$

\section*{Acknowledgments}
The authors thank the referees for the constructive comments and
suggestions which led to an improved version of the paper. The authors are also very grateful to Marco A. Teixeira for meaningful discussions and constructive criticism on the manuscript.

DDN is partially supported by FAPESP grants 2018/16430-8, 2018/13481-0, and 2019/10269-3, and by CNPq grants 306649/2018-7 and 438975/ 2018-9. GR is partially supported by FAPESP grant 2020/06708-9.

\bibliographystyle{abbrv}
\bibliography{references1}

\end{document}